\documentclass[11pt, reqno]{amsart}

\usepackage[margin=2cm]{geometry}
\usepackage{amssymb}
\usepackage[OT2, T1]{fontenc}
\usepackage[usenames,dvipsnames]{color}
\usepackage{amsmath}
\usepackage{amsthm}
\usepackage{amsfonts}
\usepackage{mathabx}
\usepackage[mathscr]{eucal}
\usepackage[absolute]{textpos} 
\usepackage{colortbl}
\usepackage{array}
\usepackage{geometry}
\usepackage{hyperref}
\usepackage{epigraph}
\usepackage{amscd}
\usepackage[all,cmtip]{xy}
\usepackage{tikz-cd}
\usepackage{enumitem}
\usepackage{dynkin-diagrams}
\usepackage{cleveref}
\usepackage{mathrsfs}
\usepackage[new]{old-arrows}

\newtheoremstyle{thms}{0.2em}{0.2em}{\itshape}{}{\bfseries}{.}{ }{}
\theoremstyle{thms}

\usepackage[american]{babel}

\usepackage[citestyle=alphabetic,bibstyle=alphabetic]{biblatex}

\newtheoremstyle{thms}{0.2em}{0.2em}{\itshape}{}{\bfseries}{.}{ }{}

\theoremstyle{plain}

\theoremstyle{definition}

\newtheorem{theorem}{Theorem}[section]
\newtheorem{lemma}[theorem]{Lemma}
\newtheorem{corollary}[theorem]{Corollary}
\newtheorem{definition}[theorem]{Definition}
\newtheorem{proposition}[theorem]{Proposition}
\newtheorem{remark}[theorem]{Remark}

\newtheorem{construction}[theorem]{Construction}
\newtheorem{definition-proposition}[theorem]{Definition-Proposition}
\newtheorem{claim}[theorem]{Claim}
\newtheorem{example}[theorem]{Example}

\makeatletter
\newcommand*{\relrelbarsep}{.386ex}
\newcommand*{\relrelbar}{%
  \mathrel{%
    \mathpalette\@relrelbar\relrelbarsep
  }%
}
\newcommand*{\@relrelbar}[2]{%
  \raise#2\hbox to 0pt{$\m@th#1\relbar$\hss}%
  \lower#2\hbox{$\m@th#1\relbar$}%
}
\providecommand*{\rightrightarrowsfill@}{%
  \arrowfill@\relrelbar\relrelbar\rightrightarrows
}
\providecommand*{\leftleftarrowsfill@}{%
  \arrowfill@\leftleftarrows\relrelbar\relrelbar
}
\providecommand*{\xrightrightarrows}[2][]{%
  \ext@arrow 0359\rightrightarrowsfill@{#1}{#2}%
}
\providecommand*{\xleftleftarrows}[2][]{%
  \ext@arrow 3095\leftleftarrowsfill@{#1}{#2}%
}
\makeatother

\makeatletter
\newcommand*{\da@rightarrow}{\mathchar"0\hexnumber@\symAMSa 4B }
\newcommand*{\da@leftarrow}{\mathchar"0\hexnumber@\symAMSa 4C }
\newcommand*{\xdashrightarrow}[2][]{%
  \mathrel{%
    \mathpalette{\da@xarrow{#1}{#2}{}\da@rightarrow{\,}{}}{}%
  }%
}
\newcommand{\xdashleftarrow}[2][]{%
  \mathrel{%
    \mathpalette{\da@xarrow{#1}{#2}\da@leftarrow{}{}{\,}}{}%
  }%
}
\newcommand*{\da@xarrow}[7]{%
  \sbox0{$\ifx#7\scriptstyle\scriptscriptstyle\else\scriptstyle\fi#5#1#6\m@th$}%
  \sbox2{$\ifx#7\scriptstyle\scriptscriptstyle\else\scriptstyle\fi#5#2#6\m@th$}%
  \sbox4{$#7\dabar@\m@th$}%
  \dimen@=\wd0 %
  \ifdim\wd2 >\dimen@
    \dimen@=\wd2 %   
  \fi
  \count@=2 %
  \def\da@bars{\dabar@\dabar@}%
  \@whiledim\count@\wd4<\dimen@\do{%
    \advance\count@\@ne
    \expandafter\def\expandafter\da@bars\expandafter{%
      \da@bars
      \dabar@ 
    }%
  }%  
  \mathrel{#3}%
  \mathrel{%   
    \mathop{\da@bars}\limits
    \ifx\\#1\\%
    \else
      _{\copy0}%
    \fi
    \ifx\\#2\\%
    \else
      ^{\copy2}%
    \fi
  }%   
  \mathrel{#4}%
}
\makeatother

\DeclareMathOperator{\Ad}{\mathrm{Ad}}			% adjoint matrix
\DeclareMathOperator{\ad}{\mathrm{ad}}
\DeclareMathOperator{\Aut}{\mathrm{Aut}}		% automorphism scheme
\DeclareMathOperator{\Bl}{Bl} % blowing-up
\DeclareMathOperator{\Dom}{Dom} % definition domain
\DeclareMathOperator{\Dist}{Dist} %distribution
\DeclareMathOperator{\Diag}{Diag}
\DeclareMathOperator{\End}{End} % The algebra of endomorphisms
\DeclareMathOperator{\fppf}{fppf}
\DeclareMathOperator{\GL}{GL}	
\DeclareMathOperator{\Grass}{Grass}	
\DeclareMathOperator{\Id}{Id} % the indentity
\DeclareMathOperator{\Hilb}{Hilb}
\DeclareMathOperator{\Hom}{Hom}

\DeclareMathOperator{\Mod}{Mod}
\DeclareMathOperator{\Multi}{Multi}
\DeclareMathOperator{\Norm}{Norm}	% The norm
\DeclareMathOperator{\Out}{Out}	% outer automorphism group
\DeclareMathOperator{\PGL}{PGL}		% Projective general linear group

\DeclareMathOperator{\Gal}{Gal}	
\DeclareMathOperator{\pr}{pr}	%  projection
\DeclareMathOperator{\Pic}{Pic}    %Picard group
\DeclareMathOperator{\red}{red} 
\DeclareMathOperator{\Res}{Res}   % Weil restriction

\DeclareMathOperator{\SL}{SL}			% Special linear group
\DeclareMathOperator{\Supp}{Supp} 
\DeclareMathOperator{\Spec}{Spec}		% Spectrum of a ring
\DeclareMathOperator{\SU}{SU}
\DeclareMathOperator{\Sch}{Sch}
\DeclareMathOperator{\Tr}{Tr}	       %trace
\title{Wonderful Embedding for group schemes in Bruhat--Tits theory}
\author{Shang Li}
\address{Tsinghua University, Yau Mathematical Sciences Center, Beijing, China}
\email{shangli@tsinghua.edu.cn}
\date{\today}

\begin{document}

\setcounter{tocdepth}{1}

\maketitle

\begin{abstract}
    For a reductive group $G$ over a discretely valued Henselian field $k$, using valuations of root datum and concave functions, the Bruhat--Tits theory defines an important class of open bounded subgroups of $G(k)$ which are essential objects in representation theory and arithmetic geometry. Moreover, these subgroups are uniquely determined by smooth affine group schemes whose generic fibers are $G$ over the ring of integers of $k$. To study these group schemes, when $G$ is adjoint and quasi-split, we systematically construct wonderful embedding for these group schemes which are uniquely determined by a big cell structure. The way that we construct our wonderful embedding is different from classical methods in the sense that we avoid embedding a group scheme into an ambient space and taking closure. We use an intrinsic and functorial method which is a variant of Artin--Weil method of birational group laws. Beyond the quasi-split case, our wonderful embedding is constructed by étale descent. Moreover our wonderful embedding behaves in a similar way to the classical wonderful compactification of $G$. Our results can serve as a bridge between the theory of wonderful compactifications and the Bruhat--Tits theory. 
\end{abstract}

\tableofcontents

\pagestyle{plain}

\section{Introduction}

Let $G$ be a connected quasi-split semisimple group of adjoint type over a strictly Henselian discretely valued field $k$ with the ring of integers $\mathfrak{o}$, the maximal ideal $\mathfrak{m}$ and the perfect residue field $\kappa$. Let $\overline{G}$ be the wonderful compactification of $G$ after De~Concini and Procesi. This geometric object $\overline{G}$ plays an important role in Lie theory, e.g., \cite{springerintersectioncohomology}, enumerative geometry, e.g., \cite{completesymmetricvarieties}, representation theory, e.g., \cite{geometryofsecondadjointness}, arithmetic geometry, e.g., \cite{XuhuaHelocalmodel} and even theoretical physics, e.g., \cite{theoreticalphysics}, see also the ICM report \cite{SpringerICM} for many other developments of wonderful compactification. Moreover, $\overline{G}$ itself has some remarkable geometric properties and many studies have focused on the geometry of $\overline{G}$. Meanwhile, some variants of wonderful compactifications are invented, for instance, the wonderful embedding for loop groups \cite{wonderfulembeddingloop} and the wonderful compactification for quantum groups \cite{wonderfulcompactificationquatum}. As a generalization of wonderful compactification, for a general reductive group (not necessarily adjoint), we have the theory of equivariant toroidal embeddings which are classified by combinatorial data. 

In order to facilitate applying various useful results of wonderful compactification (and also of toroidal embedding) into Bruhat--Tits theory and arithmetic geometry, in this paper, we establish a theory of wonderful embedding for concave function group schemes in the Bruhat--Tits theory. 

To state our main results, we need to fix some notations. Let $S$ be a maximal $k$-split torus which is contained in a Borel $k$-subgroup $B\subset G$. Let $T$ be the centralizer of $S$ in $G$. Then $T$ is a maximal $k$-torus contained in $B$. Let $B^-$ be the opposite Borel such that $B\bigcap B^-=T$, and let $U$ and $U^-$ be the unipotent radicals of $B$ and $B^-$. Let $\Phi\coloneq \Phi(G,S)$ be the relative root system, and let $\hat{\Phi}=\Phi\bigcup \{0\}$. The choice of $B$ gives a set of simple roots $\Delta\subset \Phi$.

Let $\mathcal{B}(G)$ be the Bruhat--Tits building of $G(k)$, and let $\mathcal{A}(S)$ be the apartment corresponding to $S$. To a point $x\in\mathcal{A}(S)$ and a concave function $f:\hat{\Phi}\rightarrow \mathbb{R}$, we can associate with an open bounded subgroup $G(k)_{x,f}\subset G(k)$ which is a central object of the Bruhat--Tits theory. These groups include parahoric subgroups, Moy--Prasad groups and Schneider--Stuhler groups as special cases. A fundamental algebro-geometric result of \cite{Bruhattits2} and \cite{Yusmoothmodel} is that $G(k)_{x,f}$ admits (necessarily uniquely) an integral model $\mathcal{G}_{x,f}$ over $\mathfrak{o}$. The group scheme $\mathcal{G}_{x,f}$ is rarely reductive, i.e., the unipotent radical $\mathscr{R}_u((\mathcal{G}_{x,f})_{\kappa})$ of the special fiber $(\mathcal{G}_{x,f})_{\kappa}$ is rarely trivial. We denote by $\mathsf{G}_{x,f}$ the maximal reductive quotient of the special fiber $(\mathcal{G}_{x,f})_{\kappa}$.

Let $\mathscr{S}$ be the schematic closure of $S$ in $\mathcal{G}_{x,f}$, and let $\mathrm{S}$ be the isomorphic image of the special fiber $\mathscr{S}_{\kappa}$ in the maximal reductive quotient $\mathsf{G}_{x,f}$. We will adopt the natural identifications of the character lattices and of cocharacter lattices:
\begin{equation}
   \begin{aligned}
        X^*(S)\cong X^*(\mathscr{S})\cong X^*(\mathrm{S});\;
    X_*(S)\cong X_*(\mathscr{S})\cong X_*(\mathrm{S}).
   \end{aligned} 
\end{equation}

\noindent Let $\Phi_{x,f}\coloneq \{a\in \Phi\vert f(a)+f(-a)=0\;\text{and}\; f(a)\in \Gamma_a'\}$ which is identified with the root system of $\mathsf{G}_{x,f}$ with respect to the maximal torus $\mathrm{S}$ under the above identifications, where $\Gamma'_a$ is the set of values of the root subgroup $U_a(k)$ with respect to $x$ (cf. \Cref{subsectionsetofvalues}). Let $\Delta_{x,f}\coloneq \Delta\bigcap \Phi_{x,f}\subset \Phi_{x,f}$ which is a set of simple roots. Moreover, the negative Weyl chamber $\mathfrak{C}\subset X_*(S)_{\mathbb{R}}$ defined by the simple roots $\Delta$ is a cone inside the negative Weyl chamber $\mathfrak{C}_{x,f}\subset X_*(\mathrm{S})_{\mathbb{R}}$ defined by $\Delta_{x,f}$.

One of the main goals of this paper is to establish the following result. For simplicity, in this introduction, we only state it under the assumption that $G$ splits over $k$. The theorem is valid for quasi-split groups after replacing $\mathbb{A}_{1,\mathfrak{o}}$ with certain Weil restriction of it, see \Cref{eqembeddingofbigcell}.

\begin{theorem}\label{introtheorem1}
     There is a unique smooth quasi-projective integral model $\overline{\mathcal{G}_{x,f}}$ over $\mathfrak{o}$ of $\overline{G}$ satisfying 
    \begin{itemize}
    
    \item[(i)] the $(G\times_k G)$-equivariant open immersion $G\hookrightarrow \overline{G}$ extends to a $(\mathcal{G}_{x,f}\times_{\mathfrak{o}} \mathcal{G}_{x,f})$-equivariant open immersion $\mathcal{G}_{x,f}\hookrightarrow \overline{\mathcal{G}_{x,f}}$; 
    
    \item[(ii)] the canonical $k$-open immersion $\overline{\Omega}\coloneq U^-\times_k \prod_{\Delta}\mathbb{A}_{1,k}\times_k U^+\hookrightarrow \overline{G}$ extends to an $\mathfrak{o}$-open immersion 
    $$\overline{\Omega_{x,f}}\coloneq \mathcal{U}^-\times_{\mathfrak{o}}\prod_{\Delta}\mathbb{A}_{1,\mathfrak{o}}^{( f(0))}\times_{\mathfrak{o}}\mathcal{U}^+\longhookrightarrow\overline{\mathcal{G}_{x,f}},$$
    where $\mathcal{U}^-$ and $\mathcal{U}^+$ are the schematic closures of $U^-$ and $U^+$ in $\mathcal{G}_{x,f}$ and $\mathbb{A}_{1,\mathfrak{o}}^{(f(0))}$ is the unique integral model of $\mathbb{A}_{1,k}$ over $\mathfrak{o}$ such that $\mathbb{A}_{1,\mathfrak{o}}^{(f(0))}(\mathfrak{o})=1+\mathfrak{m}^{\lceil f(0)\rceil}$;

    \item[(iii)]  $(\mathcal{G}_{x,f}\times_{\mathfrak{o}}\mathcal{G}_{x,f})\cdot \overline{\Omega_{x,f}}= \overline{\mathcal{G}_{x,f}}$, i.e., the group action morphism is surjective.
\end{itemize}
\noindent Moreover we have 
\begin{itemize}
    \item[(1)] if $f(0)=0$, the quotient sheaf $(\overline{\mathcal{G}_{x,f}})_{\kappa}/(\mathscr{R}_u((\mathcal{G}_{x,f})_{\kappa})\times_{\kappa}\mathscr{R}_u((\mathcal{G}_{x,f})_{\kappa}))$ is represented by a scheme and the natural morphism  $\mathsf{G}_{x,f}\longrightarrow (\overline{\mathcal{G}_{x,f}})_{\kappa}/(\mathscr{R}_u((\mathcal{G}_{x,f})_{\kappa})\times_{\kappa}\mathscr{R}_u((\mathcal{G}_{x,f})_{\kappa}))$ induced by the open immersion $\mathcal{G}_{x,f}\hookrightarrow \overline{\mathcal{G}_{x,f}}$ is an \emph{equivariant toroidal embedding} of the split reductive group $\mathsf{G}_{x,f}$ determined by the cone $\mathfrak{C}$ in the negative Weyl chamber $\mathfrak{C}_{x,f}$ in the sense of \cite[Definition~6.2.2]{BrionKumar};

    \item[(2)]  if $f(0)\textgreater 0$, the special fiber of $\overline{\mathcal{G}_{x,f}}$ degenerates to the special fiber $(\mathcal{G}_{x,f})_{\kappa}$;

    \item[(3)]  $\overline{\mathcal{G}_{x,f}}$ is \emph{projective} over $\mathfrak{o}$ \emph{if and only if} $\mathcal{G}_{x,f}$ is \emph{a reductive group scheme} over $\mathfrak{o}$. This is the case, for instance, if $f=0$ and $x\in \mathcal{B}(G)$ is a hyperspecial point;

    \item[(4)] the boundary $\overline{\mathcal{G}_{x,f}}\backslash \mathcal{G}_{x,f}$ is covered by $(\mathcal{G}_{x,f}\times_{\mathfrak{o}}\mathcal{G}_{x,f})$-stable smooth $\mathfrak{o}$-relative effective Cartier divisors $S_{\alpha}$ for $\alpha\in \Delta$ with $\mathfrak{o}$-relative normal crossings.
    \end{itemize}
\end{theorem}

We first remark that the existence part of \Cref{introtheorem1} is only nontrivial when $f(0)=0$ because the (2) says that, when $f(0)\textgreater 0$, the $\overline{\mathcal{G}_{x,f}}$ is simply a gluing of $\overline{G}$ and $\mathcal{G}_{x,f}$ along the open subscheme $G$. This is similar to the structure of the group scheme $\mathcal{G}_{x,f}$ when $f(0)\textgreater 0$. Moreover, our wonderful embedding $\overline{\mathcal{G}_{x,f}}$ is compatible with dilatation operation on the group scheme $\mathcal{G}_{x,f}$, see \Cref{sectioncompatiblity}.

We study the Picard group of $\overline{\mathcal{G}_{x,f}}$ and its generators, which are similar to those of the classical wonderful compactification $\overline{G}$. 

\begin{theorem}
    The boundary $\overline{\mathcal{G}_{x,f}}\backslash \overline{\Omega_{x,f}}$ is covered by prime effective Cartier divisors 
    $$\mathbf{D}_{\alpha}\coloneq \overline{B\cdot \dot{s}_{\alpha}\cdot B^-},\alpha\in \Delta$$
    where the bar indicates taking the schematic closure in $\overline{\mathcal{G}_{x,f}}$ and $\dot{s}_{\alpha}\in N_G(S)(k)$ is a representative of the simple reflection along $\alpha$ in the relative Weyl group of $G$ with respect to $S$. Moreover the Picard group $\Pic(\overline{\mathcal{G}_{x,f}})$ is freely generated by $\{\mathbf{D}_{\alpha}\vert \alpha\in \Delta\}$.
\end{theorem}

Beyond the quasi-split case, we construct our wonderful embedding by \emph{étale descent}. We now consider an adjoint reductive (not necessarily quasi-split) group $G$ over a Henselian discretely valued field $k$ with perfect residue field and ring of integers $\mathfrak{o}$. Let $K$ be a maximal unramified extension of $k$, and let $\mathcal{O}\subset K$ be the ring of integers. By a theorem of Steinberg, $G_{K}$ is quasi-split over $K$. We choose a maximal $k$-split torus $S$ with $\Phi(S)\coloneq \Phi(S,G)$ and $\hat{\Phi}(S)\coloneq \Phi(S)\bigcup \{0\}$. Let $T$ be a special $k$-torus containing $S$. For a concave function $\widetilde{f}:\hat{\Phi}(T_K)\rightarrow \mathbb{R}$ the concave function obtained by compositing a concave function $f:\hat{\Phi}(S)\longrightarrow \mathbb{R}$ with the restriction map $X^*(T_K)\rightarrow X^*(S_K)=X^*(S)$ and a point $x$ in the apartment $\mathcal{A}(S)$ of the building $\mathcal{B}(G)$ corresponding to $S$, by Bruhat--Tits theory, we have the $\mathfrak{o}$-group scheme $\mathcal{G}_{x,f}$  obtained by étale descent from the $\mathcal{O}$-group scheme $\mathcal{G}_{x,\widetilde{f}}$. Let $\overline{\mathcal{G}_{x,\widetilde{f}}}$ be the integral model over $\mathcal{O}$ by applying \Cref{introtheorem1} to $\mathcal{G}_{x,\widetilde{f}}$.

\begin{theorem}\label{introtheorem2}
    (\Cref{theoremdescent}) The $(\mathcal{G}_{x,\widetilde{f}}\times_{\mathcal{O}}\mathcal{G}_{x,\widetilde{f}})$-scheme $\overline{\mathcal{G}_{x,\widetilde{f}}}$ over $\mathcal{O}$ descends to a smooth quasi-projective $(\mathcal{G}_{x,f}\times_{\mathfrak{o}}\mathcal{G}_{x,f})$-scheme $\overline{\mathcal{G}_{x,f}}$ over $\mathfrak{o}$ which equivariantly contains $\mathcal{G}_{x,f}$ as an open dense subscheme, and the boundary $\overline{\mathcal{G}_{x,f}}\backslash \mathcal{G}_{x,f}$ is an $\mathfrak{o}$-relative Cartier divisor with $\mathfrak{o}$-relative normal crossings. Moreover, $\overline{\mathcal{G}_{x,f}}$ is $\mathfrak{o}$-projective if and only if $\mathcal{G}_{x,f}$ is a reductive group scheme over $\mathfrak{o}$. 
\end{theorem}

Since the scheme $\overline{\mathcal{G}_{x,f}}$ behaves in many aspects similarly to the wonderful compactification $\overline{G}$ (when $f=0$), \Cref{introtheorem1} and \Cref{introtheorem2} can serve as a first step towards the Bruhat--Tits theory for wonderful compactification. It is also expected to expend our results to more general wonderful varieties in the sense of \cite{Lunawonderfulvarieties}.

\begin{remark}
Recall that, when $G$ splits over $k$, the set of the $(G\times_k G)$-orbits of the wonderful compactification $\overline{G}$ are bijective to the set $\Delta\times \{0,1\}$. This has the following affine analogue. 
    If $f(0)=0$, the $(\mathcal{G}_{x,f}(\mathfrak{o})\times\mathcal{G}_{x,f}(\mathfrak{o}))$-orbits of $ \overline{\mathcal{G}_{x,f}}(\mathfrak{o})$ are bijective to the set $\prod_{\alpha\in \Delta}(\Gamma_{\alpha}\bigcap \mathbb{R}_{\geq 0})$. Given such a parallel between $\overline{G}$ and $\overline{\mathcal{G}_{x,f}}$, we ask if there is a classification of $(G(k)_{\mathcal{C},0}\times G(k)_{\mathcal{C},0})$-orbits in $ \overline{\mathcal{G}_{x,0}}(\mathfrak{o})$, where $\mathcal{C}\subset \mathcal{B}(G)$ is a chamber containing $x$ and $G(k)_{\mathcal{C},0}$ is thus an Iwahoric subgroup, as an analogue of the classification of $(B\times_k B)$-orbits in $\overline{G}$ \cite[\S~2.1]{thebehaviouratinfinityofbruhatdecomposition}; if we have a description of $\mathcal{G}_{x,f}(\mathfrak{o})_{\Diag}$-stable pieces in $\overline{\mathcal{G}_{x,f}}(\mathfrak{o})$ as an analogue of \cite{Lusztigparabolicsheaves}.
\end{remark}

This paper is organized as follows. In \Cref{sectiongroupschemeinbruhattits}, we collect some basic results (especially about smooth integral model) in Bruhat--Tits theory which are used in different places in this article. The reader familiar with it can skip this part. In \Cref{sectionwonderfulcompactification} and \Cref{sectionequivarianttoroidal}, we review the basics of the theory of wonderful compactifications and toroidal embeddings. In \Cref{sectionrationalactions}, we introduce some results about rational actions over general base scheme, which serve as our main algebro-geometric tools to establish the existence part of \Cref{introtheorem1}. \Cref{sectionparahoriccompactification} consists of the full proof of \Cref{introtheorem1}. Finally, in \Cref{sectionetaledescent}, we discuss the étale descent of the wonderful embedding constructed in \Cref{sectionparahoriccompactification}, i.e., \Cref{introtheorem2}.

\subsection{Notation and conventions}

For a real number $r\in \mathbb{R}$, we denote by $\lceil r \rceil\in \mathbb{Z}$ the smallest integer greater than or equal to $r$.

\noindent If $X$ is a set equipped with an action of a group $\Gamma$, we denote by $X^{\Gamma}$ the set of elements of $X$ fixed by $\Gamma$.

\noindent  For a scheme $S$, we denote by $\Sch/S$ the category of $S$-schemes endowed with its fppf topology.

\noindent For a group scheme $G$ over a scheme $S$, we will denote by $e\in G(S)$ the identity section.

\noindent We use dotted arrows to depict a rational morphism. For an $S$-rational morphism $f$ between two schemes $X$ and $Y$ over a scheme $S$, a test $S$-scheme $S'$ and a section $x\in X(S')$, we say that $x$ is well defined with respect to $f$ if the image of $x$ in $X_{S'}$ lies in the definition domain of $f_{S'}$. When $Y$ is separated over $S$, we denote by $\Dom(f)$ the definition domain of $f$. For the existence and the uniqueness of the definition domain of a rational morphism, see \cite[exposé~XVIII, définition~1.5.1]{SGA3II}.

\noindent Let $\mathbb{R}$ be the field of real numbers, and let
$\widetilde{\mathbb{R}}\coloneq (\mathbb{R}\times \{0,1\})\bigcup \{\infty\}$. Then $\widetilde{\mathbb{R}}$ is actually a totally ordered commutative monoid which $\mathbb{R}$ as an ordered submonoid, with respect to the oerdered monoid structure defined in \cite[6.4.1]{Bruhattits1}, see also \cite[1.6]{Bruhattitsnewapproach}. We will write $r$ (resp., $r^+$) in place of $(r,0)$ (resp., $(r,1)$).

\noindent Following the terminology of \cite[exposé~XIII, 2.1]{SGA1}, for a scheme $X$ over a scheme $S$, we say that a relative Cartier divisor $D\subset X$ is strictly with $S$-relative normal crossings, if there exists a finite family $(f_i\in \Gamma(X,\mathcal{O}_X))_{i\in I}$ such that 
    \begin{itemize}
        \item $D=\bigcup_{i\in I}V_X(f_i)$;
        \item for every $x \in \Supp(D)$, $X$ is smooth at $x$ over $S$, and 
     the closed subscheme $V((f_i)_{i\in I(x)})\subset X$ is smooth over $S$ of codimension $ \vert I(x)\vert$, where $I(x)=\{i\in I\vert f_i(x)=0\}$. 
    \end{itemize}
    The divisor $D$ is with $S$-relative normal crossings, if étale locally on $X$ it is strictly with $S$-relative normal crossings. 

\noindent For a subset $\Psi$ of a root system $\Phi$, we denote by $\Psi^{\red}$ the set of roots in $\Psi$ which are nondivisible.

\subsection{Acknowledgements}
I thank Kęstutis Česnavičius, Philippe Gille, Jochen,Heinloth, João Lourenço, Ana Bălibanu, Friedrich Knop, Arnaud Mayeux, Vikraman Balaji, Corrado de Concini, Michel Brion, Penghui Li, Ning Guo and Bertrand Rémy for for helpful conversations and correspondence. Part of the present work was obtained during a visit to Institut Fourier at Université Grenoble Alpes; I thank Michel Brion for his hospitality. I thank the organizers of Algebraic Groups workshop (ID:2516) at Oberwolfach for giving me the opportunity to report the main results of this paper. This project has received funding from the European Research Council (ERC) under the European Union's Horizon 2020 research and innovation programme (grant agreement No.~851146) and funding from National Key R\&D Program of China (grant 2024YFA1014700).

\section{Group schemes in Bruhat--Tits theory}\label{sectiongroupschemeinbruhattits}

In this section, we review the basics of Bruhat--Tits theory which will be needed in the sequel.

\subsection{Setup}\label{quasisplitsetup}
Let $k$ be a strictly Henselian discretely valued field with the valuation $\omega:k^{\times}\rightarrow \mathbb{R}$, $\mathfrak{o}\subset k$ the ring of integers, $\mathfrak{m}\subset \mathfrak{o}$ the maximal ideal, $\pi$ a uniformizer, and let $\kappa=\mathfrak{o}/\mathfrak{m}$ be the residue field which is assumed to be perfect (In particular, $\kappa$ is algebraically closed). Let $k_s$ be a separable closure of $k$. We further assume that the valuation is normalized, i.e., $\omega(k^{\times})=\mathbb{Z}$.

Let $G$ be a quasi-split adjoint semisimple group over $k$. We fix a maximal split torus $S\subset G$ and a Borel subgroup $B\subset G$ containing $S$. Let $T$ be the centralizer of $S$, then $T\subset G$ is a maximal $k$-torus, see \cite[4.1.1]{Bruhattits2}. We denote the corresponding sets of simple roots, of positive roots and of roots with respect to $S$ by $\Delta\subset \Phi^+\subset \Phi$. Let $U_a$ be the root subgroup of $a$.
By for instance \cite[Proposition~0.21]{Landvogtcompactification}, there exists an open subscheme $\Omega_{G}$ and an open immersion 
\begin{equation}\label{eqbigcell}
    \prod_{a\in \Phi^{-,\red}} U_a \times_k T\times_k \prod_{a\in \Phi^{+,\red}} U_a\hookrightarrow G
\end{equation}
whose image is $\Omega_{G}$, where the two products are taken over any total order on $\Phi^-$ and $\Phi^+$.

In order to introduce a general class of open bounded subgroups of $G(k)$ (e.g., Moy--Prasad filtration subgroups), we recall the following definitions.

\begin{definition}
    (\cite[6.4.3]{Bruhattits1}) Let $\widehat{\Phi}\coloneq \Phi\bigcup \{0\}$. A function $f: \widehat{\Phi}\rightarrow \widetilde{\mathbb{R}}$ is called concave if, for $a,b\in\widehat{\Phi}$ with $a+b\in\widehat{\Phi}$, we have
    $f(a+b)\leq f(a)+f(b).$ In particular, we have $f(0)\geq 0$.
\end{definition}

\begin{definition}
    (\cite[Definition~8.5.8]{Bruhattitsnewapproach}) For a concave function $f: \widehat{\Phi}\rightarrow \widetilde{\mathbb{R}}$, by \cite[Lemma~8.5.9]{Bruhattitsnewapproach}, we have the concave function $f^+:\widehat{\Phi}\rightarrow \widetilde{\mathbb{R}}$ defined by 
    \begin{equation*}
       f^+(a)= \left\{  \begin{array}{rcl}
            f(a)+, & {\text{if } f(a)+f(-a)=0} \\
             f(a),& {\text{otherwise.}}
        \end{array}   \right.       
    \end{equation*}
\end{definition}

\subsection{Standard filtration of $T(k)$}\label{standardfiltrationtorus}
Recall that, since $G$ is adjoint, $T$ is an induced torus. Let $T(k)_0\subset T(k)$ be the unique Iwahoric subgroup of $T(k)$. We have a decreasing filtration of $T(k)$ parameterized by $\mathbb{R}_{\geq 0}$ (which is known as the \emph{minimal congruent filtration}): 
\begin{itemize}
    \item $T(k)_0$ is the Iwahori subgroup of $T(k)$; in our case, since $T$ is induced, 
    $$T(k)_0\coloneq\{t\in T(k)\vert\; \omega(\chi(t))=0, \;\text{for all}\; \chi\in X^*(T)\; \};$$
    \item for $r\in\mathbb{R}_{>0}$,
$T(k)_r\coloneq \{\;t\in T(k)_0\;\vert \;\omega(\chi(t)-1)\geq r, \;\text{for all}\; \chi\in X^*(T)\;\}.$
\end{itemize}
This filtration was defined by Moy--Prasad \cite{Moyprasad} and Schneider--Stuhler \cite{schneiderstuhler}. For the definition of this filtration for a general $k$-torus, we refer the reader to, for instance, \cite[Definition~B.10.8]{Bruhattitsnewapproach}.

Moreover, since $T$ is induced, by \cite[Proposition~B.3.6]{Bruhattitsnewapproach}, for each $r\in\mathbb{R}_{\geq  0}$, $T(k)_r$ is schematic in the sense that there exists an (necessarily unique) integral model $\mathcal{T}^{(r)}$ of $T$ over $\mathfrak{o}$ such that $\mathcal{T}^{(r)}(\mathfrak{o})=T(k)_r\subset T(k)$. In fact, if $T= \Res_{K/k}(\mathbb{G}_{m,K})$ for some finite separable extension $K/k$, $T(k)_r=\{x\in K^{\times}\vert \omega'(x-1)\geq r \}$, where $\omega'$ is the unique valuation on $K$ extending the valuation $\omega$ on $k$. Then we have $$\mathcal{T}^{(r)}=\Res_{\mathcal{O}_K/\mathfrak{o}}(\mathbb{G}_{m,\mathcal{O}_K}^{(\lceil er \rceil)}),$$ where $e\in \mathbb{Z}$ is the ramification index of $K/k$ and $\mathbb{G}_{m,\mathcal{O}_K}^{(\lceil er \rceil)}$ is the $\lceil er \rceil$th congruent group scheme of $\mathbb{G}_{m,\mathcal{O}_K}$, cf. \Cref{appendixdilatation}.

\subsection{Filtration subgroups of root subgroups}\label{filtrationsubgroup}
For this part, we will follow \cite[Appendix~C]{Bruhattitsnewapproach}.
By the definition \cite[Definition~6.1.24]{Bruhattitsnewapproach}, the elements of the apartment $\mathcal{A}(S)$ are the valuations of the relative root datum of $(G,S)$ which are equipollent to one (hence every) Chevalley valuation. Thus a point $x\in \mathcal{A}(S)$ gives rise to a collection of morphisms 
$$(\phi_a: U_a(k)\rightarrow \mathbb{R}\bigcup\{\infty\})_{a\in \Phi}$$
satisfying certain compatibilities with the group law of $G$, see \cite[6.2.1]{Bruhattits1}. Following \cite[6.2.1, 6.4.1]{Bruhattits1}, for $r\in\mathbb{R}$ and $r^+\in\widetilde{\mathbb{R}}$, we define the root filtration subgroups 
$$U_{a,x,r}\coloneq \phi_a^{-1}([r,\infty])\;\;\;\;\text{and}\;\;\;\; U_{a,x,r^+}\coloneq \bigcup_{s\textgreater r}  U_{a,x,s}.$$

\noindent For a concave function $f:\hat{\Phi}\longrightarrow \widetilde{\mathbb{R}}$ and $x\in \mathcal{A}(S)$, following the usual convention, we let 
$$U_{a,x,f}\coloneq U_{a,x,f(a)}U_{2a,x,f(2a)}\subset U_a(k)\;\;\text{for}\;\;a\in\Phi$$
where, if $2a\notin \Phi$, we let $U_{2a,x,f(2a)}$ to be the identity of $U_a(k)$.

\subsubsection{Set of values}\label{subsectionsetofvalues}

For a point $x=(\phi_a: U_a(k)\rightarrow \mathbb{R}\bigcup\{\infty\})_{a\in \Phi}\in \mathcal{A}(S)$ as above, we introduce
\begin{itemize}
    \item If $2a\notin \Phi$, let $\Gamma_a\coloneq\Gamma'_a\coloneq \phi_a(U_a(k)\backslash\{e\})\subset \mathbb{R}$.
    \item  If $2a\in \Phi$, let $\Gamma_a\coloneq \phi_{a}(U_{a}(k)\backslash\{e\})$, $\Gamma'_a\coloneq \{\sup(\phi_a (uU_{2a}(k)))\vert u\in U_a(k)\backslash\{e\} \}\subset \mathbb{R}$ and $\Gamma_{2a}\coloneq\Gamma'_{2a}\coloneq \phi_{2a}(U_{2a}(k)\backslash\{e\})\subset \mathbb{R}$.
\end{itemize}

\subsubsection{Integral models for root subgroups}
By, for instance, \cite[Proposition~C.5.1]{Bruhattitsnewapproach}, there exists a unique smooth $\mathfrak{o}$-scheme $\mathcal{U}_{a,x,r}$ with connected fibers such that $(\mathcal{U}_{a,x,r})_k=U_a$ and $\mathcal{U}_{a,x,r}(\mathfrak{o})=U_{a,x,r}$. By \cite[Remark~8.3.10]{Bruhattitsnewapproach}, the jump set $\{r\in \mathbb{R}\;\vert\; U_{a,x,r}\neq U_{a,x,r^+} \;\}$ form a discrete subset of $\mathbb{R}$. Hence we also have smooth $\mathfrak{o}$-scheme $\mathcal{U}_{a,x,r^+}$ with connected fibers such that $(\mathcal{U}_{a,x,r^+})_k=U_a$ and $\mathcal{U}_{a,x,r^+}(\mathfrak{o})=U_{a,x,r^+}$. 

Since $U_{a,x,r^+}\subset U_{a,x,r}$, by the extension principle \cite[Proposition~0.3]{Landvogtcompactification}, the identity of $U_a$ (uniquely) extends to an $\mathfrak{o}$-morphism of group schemes $\mathcal{U}_{a,x,r^+}\rightarrow \mathcal{U}_{a,x,r}$. This morphism is not necessarily an immersion in general. We denote by $(\mathcal{U}_{a,x,r})_{\kappa}^+$ the image of the special fiber $(\mathcal{U}_{a,x,r^+})_{\kappa}$ in $(\mathcal{U}_{a,x,r})_{\kappa}$.

Since the concrete construction of $\mathcal{U}_{a,x,r}$ will be needed in the proof of \Cref{boundarydivisor}, we now recall them. For this, since $U_{a,y,r}=U_{a,x,r-\langle a,x-y\rangle}$ for any $y\in \mathcal{A}(S)$, we can assume that $x$ is the adjusted Chevalley valuation given by a Chevalley--Steinberg system $X\coloneq \{X_{\alpha}\in (\mathfrak{g}_{k_s})_{\alpha} \vert \alpha\in \widetilde{\Phi}\}$, where $\mathfrak{g}_{k_s}$ is the Lie algebra of $G_{k_s}$ cf. \cite[Definition~2.9.15, Construction~6.1.18]{Bruhattitsnewapproach}. In the following construction, we will still write $\omega$ for any extension of the valuation $\omega$ on $k$.

\begin{construction}\label{constructionofunipotentintegralmodel}
    \begin{itemize}
        \item[(1)] If $a$ is neither divisible nor multipliable, then $X$ gives an isomorphism of the root subgroup $U_a$ and $\Res_{L/k}(\mathbb{G}_{a,L})$ for a finite separable extension $L/k$. We write $\mathbb{G}_{a,L}=\Spec(L[t])$ with $t$ an indeterminate. Under this isomorphism, $U_{a,x,r}$ is identified with $\{l\in L\vert \omega(l)\geq r\}\subset \Res_{L/k}(\mathbb{G}_{a,l})(k)$. Then we have the integral model
        $$\mathcal{U}_{a,x,r}=\Res_{\mathcal{O}_L/\mathfrak{o}}(\Spec(\mathcal{O}_L[\pi_L^{-r}t])),$$
        where $\mathcal{O}_L$ is the ring of integers of $L$; $\pi_L\in \mathcal{O}_L$ is a uniformizer.
        
         \item[(2)] If $a$ is divisible, i.e., $a/2\in \Phi$, then $X$ gives an isomorphism of the root subgroup $U_a$ and $\Res_{L/k}(R_{L'/L}^0(\mathbb{G}_{a,L'}))$ for a finite separable extension $L/k$, where $L'/L$ is a separable quadratic extension and $R_{L'/L}^0(\mathbb{G}_{a,L'})\coloneq \ker(\Tr_{L'/L}:\Res_{L'/L}(\mathbb{G}_{a,L'})\rightarrow \mathbb{G}_{a,L})$. Under this isomorphism, $U_{a,x,r}$ is identified with $\{l\in L'\vert \Tr_{L'/L}(l)=0, \omega(l)\geq r+\mu\}$, where $\mu=\text{sup}\{\omega(x)\vert x\in L', \Tr_{L'/L}(x)=1\}$. By choosing $\iota\in L'$ such that $\Tr(\iota)=0$, we have an isomorphism of $L$-group scheme
        \begin{align*}
            \mathbb{G}_{a,L}&\longrightarrow R_{L'/L}^0(\mathbb{G}_{a,L'})\\
               u&\longmapsto \iota\cdot u.
        \end{align*}
        Thus, under the above isomorphism, we can construct the integral model 
        $$\mathcal{U}_{a,x,r}= \Res_{\mathcal{O}_L/\mathfrak{o}}(\Spec(\mathcal{O}_L[\pi_L^{\omega(\iota)-r-\mu}t]).$$
        
        \item[(3)] If $a$ is multipliable, i.e., $2a\in \Phi$, then $X$ gives an isomorphism of the root subgroup $U_a$ and $\Res_{L/k}(U_{L'/L})$ for a finite separable extension $L/k$, where $L'/L$ is a separable quadratic extension (we let $\sigma$ be the nontrivial involution in $\Gal(L'/L)$) and $U_{L'/L}$ is an $L$-algebraic group sending an $L$-algebra $R$ to the set $\{u,v\in R\otimes_L L'\vert v+\overline{v}=u\overline{u}\}$ (bar denote the map $\Id_R\otimes\sigma$). Under this isomorphism, $U_{a,x,r}$ is identified with $\{(u,v)\in U_{L'/L}(L)\vert\omega(v)\geq 2r+\mu\}$, where $\mu=\text{sup}\{\omega(x)\vert x\in L', \Tr_{L'/L}(x)=1\}$. By choosing $\lambda\in L'$ such that $\Tr_{L'/L}(\lambda)=1$,
        we have an isomorphism of $L$-group scheme 
        \begin{align*}
            U_{L'/L}& \longrightarrow \Res_{L'/L}(\mathbb{G}_{a,L'})\times_L R_{L'/L}^0(\mathbb{G}_{a,L'})\\
            (u,v)&\longmapsto (u,v-u\overline{u}\lambda),
        \end{align*}
        where $(u,v)$ is understood as a section of $U_{L'/L}$ valued in a test $L$-algebra and $R_{L'/L}^0$ is defined in (2). Thus, under the above isomorphism, we can construct the integral model 
        $$\mathcal{U}_{a,x,r}= \Res_{\mathcal{O}_L/\mathfrak{o}}(\Spec(\mathcal{O}_L[\pi_L^{-r}t]\times_{\mathcal{O}_L}\mathcal{U}_{2a,x,2r}),$$
    \end{itemize}
    where $\mathcal{U}_{2a,x,2r}$ is defined in (2).
\end{construction}

\subsection{Integral model of $G$}\label{subsectionintegralmodelofG}
Let $\mathcal{B}(G)$ be the (reduced) Bruhat--Tits building of $G(k)$ in the sense of \cite{Bruhattits2}. Note that, since $G$ is adjoint, $\mathcal{B}(G)$ coincides with the enlarged building. Let $\mathcal{A}(S)\subset \mathcal{B}(G)$ be the apartment corresponding to the maximal $k$-split torus $S$. 

We define $G_{x,f}\subset G(k)$ to be the subgroup generated by $T(k)_{f(0)}$ and $U_{a, x, f}$ for all $a\in \Phi$. Note that a subgroup of $G(k)$ generated by bounded subgroups may fail to be bounded, see \cite[Example~8.2.7]{Bruhattitsnewapproach} for an example. However, it can be shown as in \cite[Corollary~7.3.14]{Bruhattitsnewapproach} that $G_{x,f}$ is an open bounded subgroup of $G(k)$.  If $f=0$, the $G_{x,0}$ is usually called a parahoric subgroup. If $x$ lies in a chamber, $G_{x,0}$ is also called an Iwahori subgroup. See \cite[Appendix, Definition~1]{Twistedloopgroupsaffineflagvarieties} and \cite[définition~5.2.6]{Bruhattits2} for different (but equivalent) definitions of a parahoric subgroup. We recall the following fundamentally important algebro-geometric result.  

\begin{theorem}\label{Thparahoricgrpsche}
    There exists a unique smooth affine $\mathfrak{o}$-group scheme $\mathcal{G}_{x,f}$ of finite type with connected special fiber such that
    \item[(i)] the generic fiber of $\mathcal{G}_{x,f}$ is isomorphic to $G$;
    
    \item[(ii)] $\mathcal{G}_{x,f}(\mathfrak{o})=G_{x,f}\subset G(k);$

    \item[(iii)] the product morphism of $\mathcal{G}_{x,f}$ induces an open immersion 
    \begin{equation}\label{eqparahoricbigcell}
        \Omega_{x,f}\coloneq \prod_{a\in\Phi^{-,\red}}\mathcal{U}_{a,x,f}\times_\mathfrak{o} \mathcal{T}^{(f(0))} \times_\mathfrak{o} \prod_{a\in\Phi^{+,\red}}\mathcal{U}_{a,x,f} \hookrightarrow \mathcal{G}_{x,f}
    \end{equation}
    whose generic fiber is $\Omega_G$ and whose image does not depend on the choice of the order on $\Phi^-$ and $\Phi^+$ in the above products.
\end{theorem}

If $f(0)= 0$, \Cref{Thparahoricgrpsche} was announced in \cite[\S~3.4]{Titscorvallis1977} and was proved via representation theory by Bruhat and Tits in \cite{Bruhattits2}. As indicated by Bruhat and Tits, Landvogt gives a different proof based on the Artin--Weil theorem in \cite[\S~5]{Landvogtcompactification}. Without the assumption $f(0)=0$, Yu constructs $\mathcal{G}_{x,f}$ by a systematic use of the theory of dilatation and group smoothening \cite{Yusmoothmodel}. See also \cite{BruhatTitsclassic1}, \cite{BruhatTitsclassic2}, \cite{groupschemeG2} and \cite{groupschemeF4E6} for concrete interpretations for parahoric group schemes of classical groups and exceptional groups.

\begin{remark}
    Note that $\mathcal{G}_{x,f}$ should not be viewed only as an $\mathfrak{o}$-scheme; the way that $G$ is identified with the generic fiber $(\mathcal{G}_{x,f})_{k}$ in \Cref{Thparahoricgrpsche} (i) is also important. For instance, if $G=\GL_{n,k}$ for $n\in \mathbb{Z}_{\geq 1}$, a maximal bounded subgroup is conjugate to $\GL_{n,\mathfrak{o}}(\mathfrak{o})$ by an element $g\in \GL_{n,k}(k)$ (see, for instance, \cite[Proposition~15.1.31]{Bruhattitsnewapproach}). Hence the corresponding integral model is $\GL_{n,\mathfrak{o}}$ while the identification of $G=\GL_{n,k}$ with the generic fiber of $\GL_{n,\mathfrak{o}}$ differs from the natural embedding by the $g^{-1}$-conjugate.
\end{remark}

Let us now recall the structure of the unipotent radical of the special fiber $(\mathcal{G}_{x,f})_{\kappa}$.

\begin{proposition}\label{unipotentradicalspecialfiber}
    (\cite[Corollary~8.5.12]{Bruhattitsnewapproach})
    \begin{itemize}
        \item[(1)] The unipotent radical $ \mathscr{R}_u((\mathcal{G}_{x,f})_{\kappa})$ is the subscheme 
     $$\prod_{a\in \Phi^{-,\red}}(\mathcal{U}_{a,x,f})_{\kappa}^+\times_{\kappa}\overline{\mathscr{T}^+}\times_{\kappa}\prod_{a\in\Phi^{+,\red}}(\mathcal{U}_{a,x,f})_{\kappa}^+\subset (\Omega_{x,f})_{\kappa}\subset (\mathcal{G}_{x,f})_{\kappa},$$
    where $\overline{\mathscr{T}^+}$ is the unipotent radical $\mathscr{R}_u((\mathcal{T}^{f(0)})_{\kappa})$, the $\Omega_{x,f}\subset \mathcal{G}_{x,f}$ is the open subscheme in \Cref{eqparahoricbigcell} and the two products $\prod_{a\in \Phi^{-,\red}}$ and $\prod_{a\in\Phi^{+,\red}}$ are taken in any order.
    
     \item[(2)]  If $f(0)\textgreater 0$, the special fiber $(\mathcal{G}_{x,f})_{\kappa}$ is unipotent, and \Cref{eqparahoricbigcell} induces a $\kappa$-isomorphism
     $$\prod_{a\in \Phi^{-,\red}}(\mathcal{U}_{a,x,f})_{\kappa}\times_{\kappa} (\mathcal{T}^{(f(0))})_{\kappa}\times_{\kappa}\prod_{a\in\Phi^{+,\red}}(\mathcal{U}_{a,x,f})_{\kappa} \xlongrightarrow{\cong}(\mathcal{G}_{x,f})_{\kappa}.$$
    \end{itemize}
\end{proposition}

\begin{remark}
    When $f(0)\textgreater 0$, by \Cref{unipotentradicalspecialfiber} (2), the group scheme $\mathcal{G}_{x,f}$ can be realized as the Zariski gluing of $G$ and $\Omega_{x,f}$ along the common open $U^-TU^+\subset G$. However, this gluing process \emph{a priori} does not directly imply the condition of \Cref{Thparahoricgrpsche} (ii); it dose if we further assume that $k$ is \emph{complete} by \cite[proposition~6.4.48]{Bruhattits1}. 
\end{remark}

\section{Wonderful compactification $\overline{G}$}\label{sectionwonderfulcompactification}

\subsection{Split case}\label{splitwonderfulcomp}
Given a split reductive group of adjoint type $G$ over a field $k$, we have the following three canonical open embeddings:

$$\xymatrix{
       & & \mathbb{P}(\End(V_\lambda)) & \\
&(G\times_k G)/\Diag(G)\ar@{^{(}->}[ur]\ar@{^{(}->}[r] \ar@{^{(}->}[dr] &\Grass(\dim(G),\mathfrak{g}\oplus\mathfrak{g})\\
  & &\Hilb(G/B\times_k G/B) ,}$$
where, for the first arrow, we refer to \cite[\S~6.1]{BrionKumar} and $V_\lambda$ is an irreducible representation of the simply connected cover $G^{sc}$ with a regular dominant weight as the highest weight, see \cite[Lemma~6.1.1, Remark~6.1.2]{BrionKumar}; for the second arrow, we refer to \cite[\S~6]{deconciniprocesi} and $\mathfrak{g}$ is the Lie algebra of $G$; for the third arrow, we refer to \cite[\S~2]{compactificationHilbertsch} and $B$ is a Borel subgroup of $G$ (in fact, for the third embedding, one can replace $B$ by any parabolic subgroup $P\subset G$ such that $G$ acts faithfully on $G/P$; all exceptions for such $P$ are known, see \cite[théorème~1]{demazureautomorphismborel}). The wonderful compactification $\overline{G}$ of $G$ is realized as the schematic closure of $G$ along the above three embeddings. Moreover, in characteristic zero, $\overline{G}$ can also be constructed as a GIT quotient of the Vinberg monoid of $G^{sc}$, see \cite[\S~8, 6]{Vinbergmonoid}. In \cite{li2023equivariant}, a functorial intrinsic method based on the Artin--Weil theory of birational group laws is used to construct $\overline{G}$. In fact, we can have a relative version of wonderful compactification:

\begin{theorem}\label{threlativewonderfulcompsplitcase}
    Suppose that we have a split reductive group scheme $\mathbf{G}$ over a scheme $S$, a maximal $S$-split torus $\mathbf{T}\subset \mathbf{G}$ and a Borel subgroup $\mathbf{B}\subset \mathbf{G}$ containing $\mathbf{T}$. Let $\mathbf{B}^-$ be the unique opposite Borel subgroup such that $\mathbf{B}\bigcap\mathbf{B}^-=\mathbf{T}$ (\cite[exposé~XXVI, théorème~4.3.2.(a)]{SGA3III}), and let $\mathbf{U}^+$ and $\mathbf{U}^-$ be the unipotent radicals of $\mathbf{B}$ and $\mathbf{B}^-$ respectively, and let $\Delta$ be the set of simple roots associated to $\mathbf{B}$.
    
    Then there exists a smooth projective scheme $\overline{\mathbf{G}}$ over $S$ together with a $(\mathbf{G}\times_S \mathbf{G})$-action such that 
    \item[(i)] $\overline{\mathbf{G}}$ equivariantly contains $\mathbf{G}$ as an open subscheme;
    \item[(ii)] each geometric fiber of $\overline{\mathbf{G}}$ is identified with the wonderful compactification of the corresponding geometric fiber of $\mathbf{G}$;
    \item[(iii)] $\overline{\mathbf{G}}$ contains an open subscheme $\overline{\Omega}\coloneq \mathbf{U}^-\times_S \prod_{\Delta}\mathbb{A}_1\times_S \mathbf{U}^+$ so that $\mathcal{X}=(\mathbf{G}\times_S\mathbf{G})\cdot\overline{\Omega}$ holds as an equality of sheaves over the category of $S$-schemes, and we will write $\overline{\mathbf{T}}\coloneq \prod_{\Delta}\mathbb{A}_1$;
    \item[(iv)] the open immersion $\Omega_{\mathbf{G}}\coloneq \mathbf{U}^-\times_S \mathbf{T}\times_S \mathbf{U}^+\longhookrightarrow \overline{\mathbf{G}}$ is given by
    \begin{align*}
        \mathbf{U}^-\times_S \mathbf{T}\times_S \mathbf{U}^+&\longhookrightarrow \overline{\Omega}\coloneq \mathbf{U}^-\times_S \prod_{\Delta}\mathbb{A}_1\times_S \mathbf{U}^+\\
        (u^-,t,u^+)&\longmapsto (u^-,(\alpha(t)^{-1})_{\alpha \in\Delta},u^+).
    \end{align*}
\end{theorem}

\subsection{Quasi-split case}\label{quasisplitwonderfulcomp}

In this subsection, we keep the notations in \Cref{quasisplitsetup}.
By \cite[exposé~XXII, corollaire~2.3]{SGA3II}, we can find a finite Galois extension $\widetilde{k}/k$ such that $G_{\widetilde{k}}$ splits. Let $B^-$ be the opposite Borel such that $B^-\bigcap B=T$. Let $U^+$ and $U^-$ be the unipotent radicals of $B$ and $B^-$ respectively. Since $T_{k_s}$ is a split torus, we let 
$$\widetilde{\Phi}\subset X^*(T)\coloneq \Hom_{k_s}(T_{k_s},\mathbb{G}_{m,k_s})\;\;\text{and}\;\;\widetilde{\Delta}\coloneq\{a_1,...,a_l\}\subset\widetilde{\Phi}$$ 
be the absolute root system of $G$ relative to $T$ and the set of simple roots associated to the Borel subgroup $B$, where $l$ is the rank of $G$. The natural action of the Galois group $\Gal(k_s/k)$ on $k_s\otimes_k k[T]$ induces a $\Gal(k_s/k)$-action on $X^*(T)$ so that we have
$$(\gamma\cdot \alpha)(t)=\gamma(\alpha(\gamma^{-1}t))$$
for $\gamma\in \Gal(k_s/k)$, $\alpha\in X^*(T)$ and $t\in T(k_s)$. For a character $\alpha\in X^*(T)$, we denote its stabilizer by $\Sigma_{\alpha}\subset \Gal(k_s/k)$ whose fixed field is written as $k_{\alpha}$. We let $\mathfrak{m}_{\alpha}\subset\mathfrak{o}_{\alpha}\subset k_{\alpha}$ be the maximal ideal and the ring of integers, and let $e_{\alpha}\in\mathbb{Z}_{\geq 0}$ be the ramification index of the field extension $k_{\alpha}/k$. Because of the Borel $k$-subgroup $B$, we know that $\widetilde{\Delta}$ is stable under the action of $\Gal(k_s/k)$. Let 
$$\Phi\subset X^*(S)\coloneq \Hom_{k_s}(S_{k_s},\mathbb{G}_{m,k_s})$$
be the relative root system of $G$ relative to $S$.
Since $G$ is quasi-split, the kernel of the restriction $X^*(T)\rightarrow X^*(S)$ dose not contain any element of $\widetilde{\Phi}$ and the restriction of $\widetilde{\Delta}$ is a base $\Delta$ of $\Phi$, see \cite[\S~6]{boreltits}. For a $\alpha\in \Delta$, we denote by $\widetilde{\Delta}_{\alpha}$ the set of absolute simple roots in $\widetilde{\Delta}$ whose restriction is $\alpha$ which is a $\Gal(k_s/k)$-orbit, see \emph{loc.~cit.}

\begin{proposition}\label{propquasisplitwonderfulcomp}
    The wonderful compactification $\overline{G_{\widetilde{k}}}$ over $\widetilde{k}$ (that we recall in \Cref{splitwonderfulcomp}) descends to a projective smooth scheme $\overline{G}$ over $k$ such that the canonical open immersion 
    \begin{equation*}
        \Omega_{G_{\widetilde{k}}}\coloneq (U^-)_{\widetilde{k}}\times_{\widetilde{k}} \overline{T_{\widetilde{k}}}\times_{\widetilde{k}} (U^+)_{\widetilde{k}}\longhookrightarrow \overline{G_{\widetilde{k}}} 
    \end{equation*}
    which we recall in \Cref{threlativewonderfulcompsplitcase} (iii) descends to an open immersion
    \begin{equation*}
        \overline{\Omega}\coloneq U^-\times_k \overline{T} \times_k U^+\longhookrightarrow \overline{G},
    \end{equation*}
    where $\overline{T}$ is descended from $\overline{T_{\widetilde{k}}}=\prod_{\widetilde{\Delta}}\mathbb{A}_1$ by the natural action of $\widetilde{\Delta}$ on $\Delta$. More concretely, we have 
    $$\overline{T}\cong \prod_{\alpha\in\Delta}\Res_{k_{\alpha/k}}\mathbb{A}_{1,k_{\alpha}}.$$
\end{proposition}

\begin{proof}
    The Galois group $\Gal(\widetilde{k}/k)$ naturally acts on $(U^-)_{\widetilde{k}}$ and $(U^+)_{\widetilde{k}}$ and also it acts on $\overline{T_{\widetilde{k}}}=\prod_{\widetilde{\Delta}}\mathbb{A}_{1,\widetilde{k}}$ by the permutation action of $\Gal(k_s/k)$ on $\widetilde{\Delta}$. By \cite[\S~6.2]{li2023equivariant}, this action of $\Gal(\widetilde{k}/k)$ on $\Omega_{G_{\widetilde{k}}}$ (necessarily uniquely) extends to an action on $\overline{G_{\widetilde{k}}}$. Then, by \cite[Corollary~6.4]{li2023equivariant}, this Galois action is an effective descent datum. The last claim follows from, for instance, \cite[\S~1.5.16]{Bruhattits2}.
\end{proof}

\begin{remark}
    In fact, the descent result for the wonderful compactification holds not only for descent data which extends the outer descent data of the groups, but also for those which come from arbitrary descent data of the group. However, in the later general case, we may lose the big cell structure as in \Cref{propquasisplitwonderfulcomp} (i), see \cite[\S~6]{li2023equivariant} for details.
\end{remark}

\section{Equivariant toroidal embedding}\label{sectionequivarianttoroidal}

In this section, we recall some basics of the theory of toroidal embeddings for reductive groups which will be used in this paper. Our main reference is \cite{BrionKumar}.

Let $G$ be a connected reductive group over an algebraically closed field $k$, let $G_{\ad}$ be the adjoint quotient of $G$, and let $\overline{G_{\ad}}$ be the wonderful compactification of $G_{\ad}$. Let $T\subset G$ be a maximal split torus, and let $B$ and $B^-$ be two opposite Borel subgroups such that $B\bigcap B^-=T$, whose unipotent radicals are $U^+,U^-$.

Let $X$ be a normal $k$-variety equipped with a $(G\times_k G)$-action such that $X$ equivariantly contains the symmetric variety $(G\times_k G)/\Diag(G)$ as an open subvariety. 

\begin{definition}
    (\cite[Definition~6.2.2]{BrionKumar}) $X$ is called a toroidal embedding of $G$ if there exists a $k$-morphism $X\rightarrow \overline{G_{\ad}}$ satisfying the following commutative diagram
    $$\xymatrix{
        G \ar@{^{(}->}[r]\ar@{->>}[d] & X\ar[d] \\
        G_{\ad}\ar@{^{(}->}[r] & \overline{G_{\ad}}.
       }$$
\end{definition}

The following results shows that 
a toroidal embedding of $G$ behaves very similarly to the wonderful compactification, and the toroidal embeddings of $G$ are classified by fans supported in the negative Weyl chamber which is defined by the Borel subgroup $B$.

\begin{theorem}\label{toroidalembeddingtheorem}
    There is a bijection
    \begin{align*}
        \{\text{fans supported in the negative Weyl chamber}\}\;\;&\longleftrightarrow\;\; \{\text{toroidal embeddings of $G$}\}/\sim \\
        \sigma \;\;\;\;\;\;\;\;\;\;\;\;\;\;\;\;\;\;\;\;&\longmapsto \;\;\;\;\;\;\;\;\;\;\;\;\;\;X_{\sigma}.
    \end{align*}
    The toroidal embedding $X_{\sigma}$ contains an open subscheme $X_{\sigma,0}$ such that 
    \begin{itemize}
        \item $X_{\sigma,0}$ intersects properly with every $(G\times_k G)$-orbit of $X_{\sigma}$;
        \item we have an $k$-isomorphism $$X_{\sigma,0}\cong U^-\times_k T_{\sigma}\times_k U^+,$$
        where the $T_{\sigma}$ is the toric variety of $T$ defined by the fan $\sigma$, cf. \cite[Chapter~I]{toroidalembedding}.
    \end{itemize}
\end{theorem}

\begin{proof}
    The theorem follows from \cite[Proposition~6.2.3, 6.2.4]{BrionKumar}.
\end{proof}

\section{General results about rational actions}\label{sectionrationalactions}

The purpose of this section is roughly to extend a rational action of a group scheme $G$ on a scheme $Y$ to an actual action of $G$ on an algebraic space (sometimes a scheme, depending on the geometry of $Y$) which is birational to the given scheme $Y$. The results of this section are essentially reformulations of the results of \cite[\S~3,\S~4]{li2023equivariant} in a slightly more general context whose ideas stem from Weil \cite{Weilbirationalgrouplaw} and Artin \cite[exposé~XVIII]{SGA3II}. We expect that the general framework established in this section will open the door to construct various other equivariant embeddings for group schemes over general base schemes.

\subsection{Setup}\label{rationalactionsetup}
In this section, we consider a smooth group scheme $G$ over a scheme $S$, a scheme $Y$ which is flat and finitely presented over $S$ and an $S$-rational action of $G$ on $Y$, i.e., an $S$-rational morphism 
$$A: G\times_S Y\dashrightarrow Y$$
satisfying 
\begin{itemize}
    \item $\{e\}\times_S Y\subset \Dom(A)$ and $A(e,y)=y$ for any section $y$ of $Y$, where $e\in G(S)$ is the identity section;
    \item (Associativity) we have the following two $S$-rational morphisms which coincide:
    \begin{align*}
        G\times_S G\times_S Y &\longrightarrow Y \;\;\;\;\;\;\;\;\;\;\;\;\;\;\;\;\;\;\;\;\;\; G\times_S G\times_S Y \longrightarrow Y\\
        (g_1,g_2,y)&\longmapsto A (g_1g_2,y)  \;\;\;\;\;\;\;\;\;\;\;\;\;\;\;\; (g_1,g_2,y)\longmapsto A (g_1, A(g_2,y)).
    \end{align*}    
\end{itemize}

\noindent Then $\Dom(A) \subset G\times_S Y$ is $Y$-dense with respect to the projection onto $Y$. If we have a section $y\in Y(S')$, we will denote by $\Dom(A)_y\subset G_{S'}$ the base change of $\Dom(A)$, viewed as a $Y$-scheme via the projection, along $y$.

\subsection{An equivalence relation}
The following definition is, in spirit, similar to \cite[exposé~XVIII, 3.2.3]{SGA3II} and \cite[Definition~4.1]{li2023equivariant}.

\begin{definition}\label{equivalencerelation}
    For a test $S$-scheme $S'$ and two sections 
    $$(g_1,y_1),(g_2,y_2)\in (G\times_S Y)(S'),$$
    we say that $(g_1,y_1)$ and $(g_2,y_2)$ are equivalent if there exist an fppf cover $S''\rightarrow S'$ and a section $a\in G(S'')$ such that $A(ag_1,y_1)$ and $A(ag_2,y_2)$ are both well-defined (meaning $(ag_1,y_1),(ag_2,y_2)\in \Dom(A_{S''})$) and are equal. If so, we will write $$(g_1,y_1)\sim_{A} (g_2,y_2).$$
\end{definition}

We aim at proving the above relation is an equivalence relation. As a preparation, we will need the following lemma which says that two equivalent sections can be tested by any section that can bring them into the definition domain of the corresponding rational action. 

\begin{lemma}\label{testinglemma}
    We consider two sections $(g_1, y_1),(g_2,y_2)\in (G\times_S Y)(S)$ and $a\in G(S')$ where $S'\rightarrow S$ is an fppf-covering such that $A(ag_1,y_1)=A(ag_2,y_2)$. Then, for any $a'\in G(S'')$, where $S''$ is an $S$-scheme, if $A(a'g_1,y_1)$ and $A(a'g_2,y_2)$ are both well-defined, then, they are equal.
\end{lemma}

\begin{proof}
    Consider two $S'$-rational morphism $H_1: G_{S'}\dashrightarrow Y_{S'}$ and $H_2: G_{S'}\dashrightarrow Y_{S'}$ given by
    $$H_1(g)= A(gg_1,y_1)  \;\;\; H_1(g)= A(gg_2,y_2)$$
    where $g$ is a section of $G$ valued in an $S'$-scheme. By the associativity of $A$, we have $H_1=H_2$ as an equality of $S'$-rational morphisms because of 
    $$H_1(g)= A(gg_1,y_1)=A(ga^{-1}ag_1, y_1)=A(ga^{-1},A(ag_1, y_1))=A(ga^{-1},A(ag_2, y_2))=A(gg_2,y_2)=H_2(g).$$
    Since $S'\rightarrow S$ is an fppf cover, by \cite[exposé~XVIII, proposition~1.6]{SGA3II}, we have an equality of $S$-rational morphisms $H_1=H_2$. Thus $A(a'g_1,y_1)=A(a'g_2,y_2)$.
\end{proof}

\begin{lemma}\label{lemmaequivalence}
    The relation in \Cref{equivalencerelation} is an equivalence relation.
\end{lemma}

\begin{proof}
    It is clear that the relation is \emph{symmetric}.

    To show the \emph{reflexivity}, for a test $S$-scheme $S'$ and a section $(g,y)\in (G\times_S Y)(S')$, we have that $A(g^{-1}g, y)=A(e, y)$ is well-defined.

    To show the \emph{transitivity}, consider the following two pairs of sections
    $$(a_1, y_1)\sim (a_2, y_2),(a_2, y_2)\sim (a_3, y_3)\in (G\times_S Y)(S').$$
    Since $\{e\}\times_S Y\subset \Dom(A)$, the base changes $\Dom(A)_{y_1},\Dom(A)_{y_2},\Dom(A)_{y_3}\subset G_{S'}$ are all $S'$-dense open subschemes. Then 
    $\bigcap_{i=1,2,3}\Dom(A)_{y_i}a_i^{-1}\subset G_{S'}$ is also $S'$-dense. By \cite[exposé~XVIII,proposition~1.7]{SGA3II}, we can find a section $a\in G(S')$ such that $A(aa_1, y_1)$,$A(aa_2, y_2)$ and $A(aa_3, y_3)$ are all well-defined. We conclude by \Cref{testinglemma}.
\end{proof}

\subsection{Solution to a rational action}

\begin{definition}\label{definitionwonderfulembedding}
    We define the quotient sheaf over $\Sch$
    $$\overline{Y}\coloneq(G\times_S Y)/\sim_A$$
    with respect to the equivalence relation defined in \Cref{equivalencerelation}.
\end{definition}

\begin{remark}\label{remarkfuncotriality}
    \Cref{definitionwonderfulembedding} has a natural functoriality in the sense that, given an (not necessarily flat) $S$-scheme $S'$, we have the isomorphism
    $$\overline{Y}_{S'}\cong (G_{S'}\times_{S'} Y_{S'})/\sim_{A_{S'}}.$$
\end{remark}

We will show that $\overline{Y}$ is an algebraic space over $S$. We define an $S$-rational morphism
\begin{align}\label{definitionofphi}
    \phi: G\times_S G\times_S Y\;\;&\dashrightarrow \;\;\;\;\;\;\;\;\;\;\;Y\\
         (g_1, g_2, y)\;\;\;\;\;\;&\longmapsto A(g_1^{-1}, A(g_2,y)),\nonumber
\end{align}
and let $\Gamma\subset G\times_S G\times_S Y\times_S Y$ be the graph of the $S$-morphism $\phi\vert_{\Dom(\phi)}$. Then $\Gamma$ is naturally endowed with two projections
\begin{equation}\label{relationequation}
    \Gamma \xrightrightarrows[\pr_{23}]{\pr_{14}} G\times_S Y.
\end{equation}

\begin{theorem}\label{reinterpretation}
   The quotient sheaf of $G\times_S Y$ with respect to \Cref{relationequation} is isomorphic to $\overline{Y}$ as $\fppf$-sheaves over $\Sch/S$.
\end{theorem}

Before coming to the proof of \Cref{reinterpretation}, as a preparation, we need the following lemma.

\begin{lemma}\label{multiplicationlemma}
    If $(g_1,g_2, y, x)\in \Gamma(S)$, then $(g_2, y)\sim (g_1, x)$.
\end{lemma}

\begin{proof}
    Since $\{e\}\times_S Y\subset \Dom(A)$, the open subscheme $S'\coloneq \Dom(A)_xg_1^{-1}\bigcap\Dom(A)_yg_2^{-1}\subset G$ is $S$-dense. Hence $S'\rightarrow S$ is an fppf cover and the open immersion gives a section $a\in G(S')$ such that $A(ag_2, y)$ and $A(ag_1, x)$ are both well-defined. 

    We have the following equations of $S'$-rational morphisms from $(G\times_S G\times_S Y)_{S'}$ to $Y_{S'}$
    \begin{align*}
        A(ah_1, \phi(h_1, h_2, z))= A(ah_1, A(h_1^{-1},A(h_2, z)))=A(a, A(h_2, z))=A(ah_2, z),
    \end{align*}
    where $(h_1,h_2, z)$ are variables. Since $(g_1,g_2, y, x)\in \Gamma(S)$, we have $\phi(g_1,g_2, y)=x$. Hence $A(ag_2, y)=A(ag_1,x)$, as desired.
\end{proof}

\begin{proof}[Proof of \Cref{reinterpretation}]
    We need to show that two sections of $G\times_S Y$ are equivalent in the sense of \Cref{equivalencerelation} if and only if they come from a section of $\Gamma$ via $\pr_{14}$ and $\pr_{23}$.

    Let $S'$ be a test $S$-scheme, and consider two sections $(g_1, y_1), (g_2,y_2)\in(G\times_S Y)(S')$. Since we have $\Dom(\phi)_{S'}\subset \Dom(\phi_{S'})$, by using \Cref{multiplicationlemma}, the condition that $(g_2, g_1, y_1 ,y_2)\in\Gamma(S')$ gives rise to $(g_1, y_1)\sim(g_2,y_2)$.

    The converse direction is more involved. Suppose that, for a test $S$-scheme $S'$, we have two equivalent sections $(g_1, y_1)\sim(g_2,y_2)\in (G\times_S Y)(S')$. 
    We aim at proving $(g_2, g_1,y_1, y_2)\in \Gamma(S')$. For this, it suffices to localize on $S'$. Since $\Gamma$ is locally of finite presentation over $S$, by a limit argument, we can assume that $S'$ is the spectrum of a strictly Henselian local ring. Moreover, by \cite[proposition~18.8.8~(iii) (iv)]{EGAIV4}, the strict henselization $\widetilde{S}$ of $S$ at the image of the geometric closed point of $S'$ is flat over $S$. Since, by \cite[2.5, Proposition~6]{BLR}, $\Dom(\phi)_{\widetilde{S}}=\Dom(\phi_{\widetilde{S}})$, we can also assume that $S$ is strictly Henselian and the morphism $S'\rightarrow S$ is local.
    \begin{claim}\label{claim1}
        There exist a flat and finitely presented morphism $S_0\rightarrow S$ and a section $h\in G(S_0)$ such that 
        $$A(hg_1,y_1),\;A(hg_2, y_2)\;\;\text{and}\;\;A((hg_2)^{-1},A(hg_1, y_1))$$
        are all well-defined.
    \end{claim}
    Granted the claim, we consider $S_0$-rational morphism
    \begin{align*}
        \phi':(G\times_S G\times_S Y)_{S_0}\;&\dashrightarrow\;\;\;\;\;\; Y_{S_0}\\
       (f_1,f_2, x) \;\;\;\;\;\;&\longmapsto    A((hf_1)^{-1},A(hf_2, x)).
    \end{align*}
    It is clear that $\phi_{S_0}=\phi'$, and we have $(g_2, g_1,y_1)\in \Dom(\phi')(S'\times_S S_0)$ because, in each step we apply $A$ to form $\phi'$, the objects are well-defined by design. Hence, by flat descent, $(g_2, g_1,y_1)\in \Dom(\phi_{S_0})(S'\times_S S_0)=\Dom(\phi)(S')$. Moreover, by \Cref{testinglemma}, we have $A(hg_1,y_1)=A(hg_2, y_2)$, hence, $(g_2, g_1,y_1, y_2)\in \Gamma(S')$.

    \begin{proof}[Proof of \Cref{claim1}]
     We define an $S$-dense open subscheme
     $$\mathcal{V}\coloneq \Dom(A)\bigcap \psi^{-1}(\Dom(A)),$$
     where $\psi: G\times_S Y\longrightarrow G\times_S Y$ is an $S$-rational morphism given by sending $(g,x)$ to $(g^{-1},A(g,x))$. It is clear that $\{e\}\times Y\subset \mathcal{V}$. In particular the open $\mathcal{V}$ is $Y$-dense. We only need to show that there exist a flat and finitely presented morphism $S_0\rightarrow S$ and an $S$-morphism $h:S_0\rightarrow G$ whose base change to $S'$ lands in the $S'$-dense open subscheme
     $$\mathcal{W}'\coloneq \mathcal{V}_{y_1} g_1^{-1}\bigcap\mathcal{V}_{y_2} g_2^{-1} \subset G\times_S S',$$
     during this reduction procedure, \Cref{testinglemma} is used.
     For this, we adopt the idea of the proof of \cite[corollaire~17.16.2]{EGAIV4}.

     Let $s'$ (resp., $s$) be the closed point of $S'$ (resp., $S$). Since $\mathcal{W}'$ is $S'$-dense, we can assume that the special fiber $\mathcal{W}'_{s'}$ contains an open subscheme of the form $\mathcal{W}\times_{k(s)}k(s')$ where $\mathcal{W}$ is a nonempty open subscheme of $G_s$ (this follows from a limit argument and \cite[01JR, 01UA]{stacks-project}). We choose a closed point $h\in \mathcal{W}\subset  G$. We can choose a system of parameters $(\widetilde{t_i})_{1\leq i\leq n}$ in the local ring $\mathcal{O}_{G_s,h}$, because the special fiber $G_s$ is a regular scheme. Then we can find an affine open neighborhood $V$ containing $h$ and $n$ sections $(t_i)_{1\leq i\leq n}\subset \Gamma (V,\mathcal{O}_{G})$ lifting $(\widetilde{t_i})_{1\leq i\leq n}$. Let $V'\subset V$ be the closed subscheme cut out by $(t_i)_{1\leq i\leq n}$. By the local criterion for flatness (see \cite[théorème~11.3.8 b') and c)]{EGAIV3}), after shrinking $V'$ if needed, we can assume that $V'$ is flat over $S$. Since $(\widetilde{t_i})_{1\leq i\leq n}$ is a system of parameters,  $\mathcal{O}_{V'_s, h}$ is artinian. Since $h$ is a closed point in $V'_s$, it is isolated in $V'_s$. Hence $\Spec(\mathcal{O}_{V',h})$ is quasi-finite over $S$. As $S$ is Henselian, by the Zariski main theorem, the $S$-scheme $\Spec(\mathcal{O}_{V',h})$ is even finite, see \cite[théorème~18.5.11 c')]{EGAIV4}. We take $\Spec(\mathcal{O}_{V',h})$ as $S_0$ and consider the natural $S$-morphism $\epsilon\colon\Spec(\mathcal{O}_{V',h})\rightarrow G$.

    Now we show that the $S'$-morphism $\epsilon\times \Id_{S'}\colon S_0\times_S S'\rightarrow G\times_S S'$ has image in $\mathcal{W'}$. Since $\epsilon$ sends the closed point of $\Spec(\mathcal{O}_{V',h})$ to $h\in \mathcal{W}$ and, by \cite[proposition~2.4.4]{EGA1}, $\epsilon$ preserves generizations (in the sense of \cite[0061]{stacks-project}) of the closed point of $\Spec(\mathcal{O}_{V',h})$, the image of the closed fiber $\epsilon_s$ lies in $\mathcal{W}$. 
    Hence, the special fiber $(\epsilon\times_S \Id_{S'})_{s'}$ has image in $\mathcal{W}\times_{k(s)}k(s')\subset \mathcal{W}'$. Now since $S'$ is Henselian, by \cite[proposition~18.5.9 (ii)]{EGAIV4}, the $S'$-finite scheme $\Spec(S_0\times_S S')$ decomposes as the disjoint union of some local schemes. We already know that the closed points of these local schemes are mapped into $\mathcal{W}'$ by $\epsilon\times_S \Id_{S'}$. 
    Hence, by \cite[proposition~2.4.4]{EGA1}, the whole space $\Spec(S_0\times_S S')$ also lands in $\mathcal{W}'$, as desired. 
    \end{proof}

\end{proof}

\begin{corollary}\label{algebraicspace}
    The quotient sheaf $\overline{Y}$ is an algebraic space over $S$.
\end{corollary}

\begin{proof}
    By the Artin's result (see \cite[corollaire~(10.4)]{Champsalgebrique} or \cite[04S6~(2)]{stacks-project} or \cite[théorème~3.1.1]{anantharaman}) which says that the quotient of an algebraic space with respect to an fppf-relation is an algebraic space, \Cref{reinterpretation} reduces us to showing that the two $S$-morphisms $\pr_{14}, \pr_{23}:\Gamma\rightarrow G\times_S Y$ are flat and locally of finite presentation. Since $G$ is flat and locally of finite presentation over $S$, by base change, so is $\pr_{23}$. To show that $\pr_{14}$ is so, it suffices to notice that $\Gamma$ is symmetric under the permutation $(g_1,g_2, y,x)\longmapsto (g_2,g_1, x,y)$, which follows from \Cref{reinterpretation} and the relation $\sim_A$ is symmetric \Cref{lemmaequivalence}.
\end{proof}

We now endow $\overline{Y}$ with a $G$-action.

\begin{definition-proposition}\label{groupaction}
    Let $S'$ be a test $S$-scheme. For a section $g\in G(S')$ and a section $\overline{y}\in\overline{Y}(S')$ represented by a section $(f, x)\in (G\times_S Y)(S'')$ where $S''\rightarrow S'$ is an fppf-cover, we define $g\overline{y}\in \overline{Y}(S')$ to be the section represented by $(gf, x)\in (G\times_S Y)(S'')$. We claim that $g\overline{y}$ does not depend on the choice of a representative of $\overline{y}$.
\end{definition-proposition}

\begin{proof}
    Suppose that $\overline{y}$ has two representatives $(f_1, x_1)\in (G\times_S Y)(S_1'')$ and $(f_2, x_2)\in (G\times_S Y)(S_2'')$, where $S_1''$ and $S_2''$ are fppf-cover of $S'$. By pulling back to $S_1''\times_{S'}S_2''$, we can assume $S_1''=S_2''$ which we rename to $S''$. Since $\{e\}\times Y\subset \Dom(A)$, by \cite[exposé~XVIII, proposition~1.7]{SGA3II}, we can find a section $a$ of $G$ valued in some fppf cover of $S''$ such that $A(agf_1,x_1)$ and $A(agf_2,x_2)$ are both well-defined. We conclude by appealing to \Cref{testinglemma}.
\end{proof}

\begin{proposition}\label{openimmersion}
    The morphism $j: Y\rightarrow \overline{Y}$ which sends a section $y\in Y(S')$ to the equivalence class represented by $(e, y)\in Y(S')$ is an open immersion, where $S'$ is a test $S$-scheme.
\end{proposition}

\begin{proof}
    First, note that $j$ is a monomorphism. Actually, if there are two sections $y_1, y_2\in Y(S')$ such that $(e, y_1)\sim (e, y_2)$, by \Cref{testinglemma}, we have $y_1=y_2$.

    Next we prove that $j$ is formally smooth. Given an infinitesimal thickening $H'\longhookrightarrow H$ and the commutative diagram
    $$\xymatrix{
H \ar@{^{(}->}[d] \ar[r]^y & Y\ar[d]^{j}\\
H' \ar[r]_{(g,x)}   \ar@{.>}[ur]^{?}       &\overline{Y},}$$
where $(g, x)\in(G\times_S Y)(H')$,
we need to find a section in $Y(H')$ fitting into the above commutative diagram. Since $j$ is a monomorphism and by our assumption $Y$ is finitely presented over $S$, by working étale locally over $H'$ and a limit argument, we can assume that $H'$ is the spectrum of a strictly Henselian local ring. Since $\{e\}\times Y\subset \Dom(A)$, the open subscheme $M\coloneq \Dom(A_H)_y\bigcap (\Dom(A_{H'})_xg^{-1})_H\subset G\times_S H$ is $H$-dense. Hence $M$ is $H'$-dense as well. Now \cite[\S~2.3, Proposition~5]{BLR} gives a section $m\in M(H')$. By the construction of $M$, we have that 
$A(mg,x)$ and $A(m,y)$ are both well-defined, and by \Cref{testinglemma}, they are equal. Hence by \Cref{reinterpretation}, we have that 
$$y\in \phi(m,mg,x)\in Y(H'),$$
where $\phi$ is given by \Cref{definitionofphi}.

Finally,we show that $j$ is an open immersion. Since by \Cref{reinterpretation}, $\overline{Y}$ is locally of finite type over $S$, and so is $Y$, the monomorphism $j$ is locally of finite presentation by \cite[06Q6]{stacks-project}. Hence, by \cite[0DP0]{stacks-project}, $j$ is smooth. Furthermore, by \cite[0B8A]{stacks-project} (due to David~Rydh), $j$ is representable by schemes, and so an open immersion by \cite[025G]{stacks-project}.
\end{proof}

%
%\begin{remark}
    %By \cite[Proposition~9.3.5]{Bruhattitsnewapproach}, if $S'\subset G$ is a maximal $k$-split torus such that the corresponding apartment $\mathcal{A}(S')$ also contains $x$, then there exists an element $g\in \mathcal{P}_{x}(\mathfrak{o})$ such that $gSg^{-1}=S'$. Hence, $\overline{\mathcal{P}_{x}}$ does not depend on the choice of the apartments containing $x$.
%\end{remark}

\section{Integral model of $\overline{G}$: quasi-split case}\label{sectionparahoriccompactification}

The goal of this section is to prove \Cref{introtheorem1} and study various geometric properties of the embeddings in \Cref{introtheorem1}. The proof of uniqueness is given in \Cref{uniqueness}. When $f(0)=0$, to establish the existence of the desired integral model of the wonderful compactification $\overline{G}$ in \Cref{introtheorem1}, we show the key result \Cref{theoremrationalaction} in \Cref{subsectionrationalaction} so that the machinery of \Cref{sectionrationalactions} can be applied to the construction of embeddings. When $f(0)\textgreater 0$, the construction is discussed in \Cref{f0biggerthat0}. We also explore various geometric properties like boundary divisors, compatibility with dilatation and Picard group in the rest of \Cref{sectionparahoriccompactification}.

\subsection{Group setup}\label{groupsetup}
We shall keep the notation as in \Cref{quasisplitwonderfulcomp}. Let $\widetilde{\Delta}\subset \widetilde{\Phi}^+$ and $\widetilde{\Phi}^-$ be the preimages of $\Delta\subset \Phi^+$ and $\widetilde{\Phi}^-$ along the restriction map $\widetilde{\Phi}\rightarrow \Phi$. We fix an element in $\widetilde{a}\in \widetilde{\Delta}_{a}$ for each $a\in \Delta$. This gives a lifting $\widetilde{r}\in \widetilde{\Phi}$ for each relative root $r\in \Phi$. For an absolute root $\alpha\in \widetilde{\Phi}$, let $O(\alpha)\subset\widetilde{\Phi}$ be the $\Gal(k_s/k)$-orbit of $\alpha$.

\subsubsection{Extension principle}
The following well-known extension principle will be frequently used in this paper. For the convenience of the reader, we shortly recall the statement. Let $\mathcal{X}$ and $\mathcal{Y}$ be two smooth affine scheme over $\mathfrak{o}$, and let $g: \mathcal{X}_k\rightarrow \mathcal{Y}_k$ be a morphism over $k$.

\begin{proposition}\label{extensionprinciple}
(\cite[1.7.3~c)]{Bruhattits2}, \cite[Proposition~0.3]{Landvogtcompactification}, \cite[Lemma~2.10.13]{Bruhattitsnewapproach}) If the residue field $\kappa$ of $\mathfrak{o}$ is separable closed and $g(\mathcal{X}(\mathfrak{o}))\subset \mathcal{Y}(\mathfrak{o})$, then $g$ extends uniquely to an $\mathfrak{o}$-morphism $\mathcal{X}\rightarrow\mathcal{Y}$.
\end{proposition}

Note that \Cref{extensionprinciple} fails if $\mathcal{Y}$ is not assumed to be affine over $\mathfrak{o}$, see \cite[Example~2.10.14]{Bruhattitsnewapproach} for an example.

\subsubsection{Extension of a toroidal embedding}

Since $G$ is adjoint, the maximal torus $T$ is an induced torus, see \cite[proposition~4.4.16]{Bruhattits2}. We now elaborate on this fact. For an $a\in \Delta$, we have the $k_s$-morphism
\begin{align*}
\nu_{k_s}:T_{k_s}&\longrightarrow (\Res_{k_{\widetilde{a}}/k}(\mathbb{G}_{m,k_{\widetilde{a}}}))_{k_s}=\mathbb{G}_{m,k_s}^{[k_{\widetilde{a}}:k]}\\ 
t&\longmapsto (\alpha(t)^{-1})_{\alpha\in \widetilde{\Delta}_a}
\end{align*}
which is $\Gal(k_s/k)$-equivariant. By Galois descent (see, for instance, \cite[Corollary~11.2.9]{linearalgrpSpringer}), we get a morphism of $k$-groups
\begin{align*}
    \nu_a: T\longrightarrow \Res_{k_{\widetilde{a}}/k}(\mathbb{G}_{m,k_{\widetilde{a}}}).
\end{align*}
Since $G$ is of adjoint type, the morphism of $k$-group 
$$\nu\coloneq (\nu_a)_{a\in \Delta}:T\longrightarrow \prod_{a\in \Delta}\Res_{k_{\widetilde{a}}/k}(\mathbb{G}_{m,k_{\widetilde{a}}})$$
is an isomorphism. 

We denote by $\mathfrak{o}_{\widetilde{a}}$ the ring of integers of $k_{\widetilde{a}}$. We have the connected lft-Néron model (i.e., the relative neutral component of lft-Néron model) $\mathcal{T}$ and $\mathcal{T}'$ for $T$ and $\prod_{a\in \Delta}\Res_{k_{\widetilde{a}}/k}(\mathbb{G}_{m,k_{\widetilde{a}}})$ respectively. In fact $$\mathcal{T'}\cong\prod_{a\in \Delta}\Res_{\mathfrak{o}_{\widetilde{a}}/\mathfrak{o}}(\mathbb{G}_{m,\mathfrak{o}_{\widetilde{a}}}).$$

\noindent By \cite[proposition~4.4.4]{Bruhattits2}, the morphism of $k$-induced tori $(\nu_a)_{a\in \Delta}$ extends uniquely to an $\mathfrak{o}$-isomorphism of group schemes 
\begin{align}\label{equationisomorphismoftorusintegralmodel}
    \mathcal{T}\longrightarrow \prod_{a\in \Delta}\Res_{\mathfrak{o}_{\widetilde{a}}/\mathfrak{o}}(\mathbb{G}_{m,\mathfrak{o}_{\widetilde{a}}}).
\end{align}
By postcomposing with the canonical open immersion $\prod_{a\in \Delta}\Res_{\mathfrak{o}_{\widetilde{a}}/\mathfrak{o}}(\mathbb{G}_{m,\mathfrak{o}_{\widetilde{a}}})\hookrightarrow \prod_{a\in \Delta}\Res_{\mathfrak{o}_{\widetilde{a}}/\mathfrak{o}}(\mathbb{A}_{1,\mathfrak{o}_{\widetilde{a}}})$ (\cite[Proposition~A.5.2 (4)]{CGP}) and \Cref{dilitationproperty}, for $r\in \mathbb{R}_{\geq 0}$, we obtain an open immersion over $\mathfrak{o}$
\begin{align}\label{definitionofnubar}
    \overline{\nu^{(r)}}:\mathcal{T}^{(r)}\hookrightarrow \overline{\mathcal{T}^{(r)}}\coloneq \prod_{a\in \Delta}\Res_{\mathfrak{o}_{\widetilde{a}}/\mathfrak{o}}(\mathbb{A}^{(r)}_{1,\mathfrak{o}_{\widetilde{a}}}),
\end{align}
where $\mathcal{T}^{(r)}$ is defined in \Cref{standardfiltrationtorus} and $\mathbb{A}^{(r)}_{1,\mathfrak{o}_{\widetilde{a}}}$ is the unique integral model of $\mathbb{A}_{1,k_{\widetilde{a}}}$ over $\mathfrak{o}_{\widetilde{a}}$ such that $\mathbb{A}^{(r)}_{1,\mathfrak{o}_{\widetilde{a}}}(\mathfrak{o}_{\widetilde{a}})=1+\mathfrak{m}_{\widetilde{a}}^{\lceil e_{\widetilde{a}}r\rceil}$. The group scheme $\mathcal{T}^{(r)}$ is usually called the $\lceil r\rceil$th congruence group scheme of $\mathcal{T}$, cf., \cite[Definition~A.5.12]{Bruhattitsnewapproach}. Note that we have $(\overline{\mathcal{T}^{(r)}})_{k}\cong \overline{T}$, see \Cref{propquasisplitwonderfulcomp}.
Using the open immersion $\overline{\nu^{(r)}}$, we can further define the following open immersion:
\begin{equation}\label{eqembeddingofbigcell}
		\begin{aligned}
			\mu: \Omega_{x,f}\coloneq\prod_{a\in\Phi^{-,\red}}\mathcal{U}_{a,x,f}\times_{\mathfrak{o}}\mathcal{T}^{(f(0))}\times_{\mathfrak{o}}\prod_{a\in\Phi^{+,\red}}\mathcal{U}_{a,x,f}&\longhookrightarrow \overline{\Omega_{x,f}}\coloneq\prod_{a\in\Phi^{-,\red}}\mathcal{U}_{a,x,f}\times_{\mathfrak{o}}\overline{\mathcal{T}^{(f(0))}}\times_{\mathfrak{o}}\prod_{a\in\Phi^{+,\red}}\mathcal{U}_{a,x,f}\\
			(u^-,t,u^+)&\longmapsto (u^-,\overline{\nu^{(f(0))}}(t),u^+).
		\end{aligned}
\end{equation}
\noindent In the following, if $f(0)=0$, we will omit the superscript $(f(0))$ in \Cref{eqembeddingofbigcell}.

\subsubsection{Extension of morphisms of tori}
For an absolute root $\alpha=\sum_{\alpha_i\in\widetilde{\Delta}}n_i(-\alpha_i)\in \widetilde{\Phi}^-$ with $n_i\in \mathbb{Z}_{\geq 0}$, let
\begin{align*}
    m_{\alpha}': \overline{T_{k_s}}=\displaystyle \prod_{\alpha_i\in \widetilde{\Delta}}\mathbb{A}_{1,k_s} &\longrightarrow \displaystyle\prod_{\beta\in O(\alpha)}\mathbb{A}_{1,k_s}\\
   (t_i)_{\alpha_i\in \widetilde{\Delta}} &\longmapsto (t_1^{b_1}\cdot t_2^{b_2}\cdot...\cdot t_l^{b_l})_{\beta= \sum b_i(-\alpha_i)\in O(\alpha)}.
\end{align*}
We have the following two commutative diagrams:

% https://q.uiver.app/#q=WzAsOCxbMCwwLCJBIl0sWzAsMSwiRCJdLFsxLDEsIkMiXSxbMywwLCJFIl0sWzMsMSwiRiJdLFs0LDEsIkciXSxbNCwwLCJUIl0sWzEsMCwiQiJdLFswLDFdLFsxLDJdLFszLDZdLFs2LDVdLFszLDRdLFs0LDVdLFs3LDJdLFswLDddLFsxNCwxMiwiRMOoc2NldCIsMCx7InNob3J0ZW4iOnsic291cmNlIjozMCwidGFyZ2V0IjozMH0sImxldmVsIjoxLCJzdHlsZSI6eyJib2R5Ijp7Im5hbWUiOiJzcXVpZ2dseSJ9fX1dXQ==
\begin{equation}\label{multiplicationdescent}
     \begin{tikzcd}
	T_{k_s}           & \overline{T_{k_s}}                                 && T & \overline{T} \\
	\displaystyle\prod_{\beta\in O(\alpha)}\mathbb{G}_{m,k_s} & \displaystyle\prod_{\beta\in O(\alpha)}\mathbb{A}_{1,k_s}                       && \Res_{k_{\alpha}/k}\mathbb{G}_{m,k_{\alpha}} & \Res_{k_{\alpha}/k}\mathbb{A}_{1,k_{\alpha}}
	\arrow["\nu_{k_s}",hook,from=1-1, to=1-2]
	\arrow["(\beta)_{\beta\in O(\alpha)}"', from=1-1, to=2-1]
	\arrow["m_{\alpha}'"',""{name=0, anchor=center, inner sep=0}, from=1-2, to=2-2]
	\arrow["\nu",hook,from=1-4, to=1-5]
	\arrow[""{name=1, anchor=center, inner sep=0}, from=1-4, to=2-4]
	\arrow["m_{\alpha}"',from=1-5, to=2-5]
	\arrow[hook,from=2-1, to=2-2]
	\arrow[hook,from=2-4, to=2-5]
	\arrow["{\text{Galois Descent}}", shorten <=16pt, shorten >=16pt, Rightarrow, from=0, to=1]
     \end{tikzcd}
\end{equation}
where the components $\mathbb{A}_{1,k_{\alpha}}$ (resp., $\mathbb{G}_{m,k_s}$) in the product $\displaystyle\prod_{\beta\in O(\alpha)}\mathbb{A}_{1,k_s}$ (resp., $\displaystyle\prod_{\beta\in O(\alpha)}\mathbb{G}_{m,k_s}$) are transitively permuted by the natural action of $\Gal(k_s/k)$ on the orbit $O(\alpha)$.

\begin{corollary}\label{multiplicationintegral}
    The $k$-morphism $m_{\alpha}: \overline{T}\longrightarrow \Res_{k_{\alpha}/k}\mathbb{A}_{1,k_{\alpha}}$ which is defined by \Cref{multiplicationdescent} extends (necessarily uniquely) to an $\mathfrak{o}$-morphism $\overline{\mathcal{T}}\longrightarrow \Res_{\mathfrak{o}_{\alpha}/\mathfrak{o}}\mathbb{A}_{1,\mathfrak{o}_{\alpha}}$.
\end{corollary}

\begin{proof}
    Let us write 
    $$\alpha=\displaystyle\sum_{a\in\Delta}\sum_{\sigma\in \Gal(k_s/k)/\Gal(k_s/k_{\widetilde{a}})}n_{\sigma(\widetilde{a})}\sigma(\widetilde{a}),\;n_{\sigma(\widetilde{a})}\in\mathbb{Z}_{\geq 0}. $$
    By \Cref{extensionprinciple}, it suffices to show that $m_{\alpha}(\overline{\mathcal{T}}(\mathfrak{o}))\subset \Res_{\mathfrak{o}_{\alpha}/\mathfrak{o}}\mathbb{A}_{1,\mathfrak{o}_{\alpha}}(\mathfrak{o})=\mathfrak{o}_{\alpha}$. This follows from the definition of $m_{\alpha}'$: for $y=(y_{\widetilde{a}}\in\mathfrak{o}_{\widetilde{a}})_{a\in \Delta}\in \overline{\mathcal{T}}(\mathfrak{o})$, we have that
    $$m_{\alpha}(y)=\displaystyle\prod_{a\in\Delta} \displaystyle\prod_{\sigma\in \Gal(k_s/k)/\Gal(k_s/k_{\widetilde{a}})}   \sigma(y_{\widetilde{a}})^{n_{\sigma(\widetilde{a})}}    \in \mathfrak{o}_{\alpha},$$
    as desired.
\end{proof}

\subsection{Rational action}\label{subsectionrationalaction}
The ultimate goal of \Cref{subsectionrationalaction} is to prove \Cref{theoremrationalaction}. The proof is based on two explicit analyses of $\SL_2$ and $\SU_3$ (\Cref{suslcase}). We use the techniques of distribution to combine the previous analyses together to form \Cref{distributioninvariant} which will be crucial in the proof of \Cref{theoremrationalaction}.

The point $x\in \mathcal{A}(S)$ is by definition a valuation $(x_{a}: U_a(k)\longrightarrow \mathbb{R})_{a\in\Phi}$ of the relative root datum of $(G,S)$. Here, following \cite{Bruhattits1}, the root datum means the datum $(\{U_a(k)\}_{a\in \Phi}, T(k))$, which differs from the root datum used in \cite{SGA3III} to classify split reductive group schemes. 

\subsubsection{Chevalley--Steinberg system}\label{ChevalleySteinbergsystem}

By \cite[\S~4.1.3]{Bruhattits2}\cite[Proposition~4.4]{Landvogtcompactification} or \cite[Proposition~4.4]{Landvogtcompactification}, we can fix a Chevalley--Steinberg system of $G$ (with respect to the maximal torus $T$)
$$(\widetilde{\chi}_{\gamma}: \mathbb{G}_{a,k_s}\longrightarrow \widetilde{U}_{\gamma})_{\gamma\in \widetilde{\Phi}}$$
where $\widetilde{U}_{\gamma}\subset G_{k_s}$ is the root subgroup associated to the absolute root $\gamma$. In particular, by definition, we have that
\begin{itemize}
    \item[(1)] if $\alpha\in \widetilde{\Phi}$ and $(\alpha\vert_S)/2\notin \Phi$, for any $\sigma\in \Gal(k_s/k)$, $\widetilde{\chi}_{\sigma(\alpha)}=\sigma\circ \widetilde{\chi}_{\alpha}\circ \sigma^{-1}$. Hence $\widetilde{\chi}_{\alpha}$ is defined over $k_{\alpha}$.
    \item[(2)] if $\alpha\in \widetilde{\Phi}$ and $(\alpha\vert_S)/2\in \Phi$, then there exist $\beta, \beta'\in \widetilde{\Phi}$ such that $\beta\vert_S=\beta'\vert_S=(\alpha\vert_S)/2$ and $\alpha=\beta+\beta'$. Then, for any $\sigma\in \Gal(k_s/k)$, there exists an $\epsilon\in \{\pm 1\}$ such that $\widetilde{\chi}_{\sigma(\alpha)}=\sigma\circ \widetilde{\chi}_{\alpha}\circ (\epsilon\sigma^{-1})$, and $\epsilon=-1$ if and only if $\sigma$ acts non-trivially on $k_{\beta}=k_{\beta'}$. Hence $\widetilde{\chi}_{\alpha}$ is defined over $k_{\beta}$.
\end{itemize}

By \cite[\S~4.2.2]{Bruhattits2}, this Chevalley--Steinberg system gives rise to a valuation of $\varphi$ of the relative root datum of $(G,S)$ via (quasi-split) descent. By the definition of the apartment $\mathcal{A}(S)$ \cite[\S~4.2.4]{Bruhattits2}, $x$ is equipollent to $\varphi$ in the sense that there exists a $v\in X_{\ast}(T)\otimes_{\mathbb{Z}}\mathbb{R}$ such that 
$$x=\varphi+v.$$

Our choice of $\varphi$ serves as a temporary tool to establish \Cref{theoremrationalaction} which does not depend on the choice of any Chevalley--Steinberg system.

\subsubsection{Two commutator computations}
In the construction of parahoric group scheme \cite{Bruhattits2} and \cite{Landvogtcompactification} and also in the proof of existence theorem for split reductive groups \cite[exposé~XXV]{SGA3III} and \cite{Chvealleysemisimplegroup}, the behavior of forming commutator of two opposite root subgroups plays a fundamental role. Here we need a similar result whose proof is based on two fundamental cases of $\SL_2$ and $\SU_3$. The following lemma is inspired by \cite[4.5.8~proposition]{Bruhattits2} and \cite[Lemma~5.9, Lemma~5.11]{Landvogtcompactification}.

\begin{lemma}\label{suslcase}
    Suppose that $a\in \Phi^{+,\red}$.
     There exists $d_{a}\in\mathfrak{o}[\mathcal{U}_{ a,x, f}\times_{\mathfrak{o}}\overline{\mathcal{T}}\times_{\mathfrak{o}} \mathcal{U}_{-a,x,f}]$ such that 
    \item[(1).] $\mathcal{U}_{a,x, f^+}(\mathfrak{o})\times\overline{\mathcal{T}}(\mathfrak{o})\times \mathcal{U}_{-a,x, f^+}(\mathfrak{o})\subset D_{\mathcal{U}_{a,x, f}\times_{\mathfrak{o}}\overline{\mathcal{T}}\times_{\mathfrak{o}} \mathcal{U}_{-a,x, f}}(d_{a})(\mathfrak{o}),$
    \item[(2).] $e\times_{k}  \overline{T}\times_{k} e\subset D_{U_{a}\times_{k}\overline{T}\times_{k} U_{-a}}(d_{a}),$
    \item[(3).] The $k$-morphism 
    $$D_{U_a\times_k \overline{T} \times_k U_{-a}}(d_a)\rightarrow U_{-a}\times_k \overline{T} \times_k U_{a}\subset \overline{G}$$
    given by $(G\times_k G)$-action on $\overline{G}$ extends to an $\mathfrak{o}$-morphism
    $$\beta_{a}: D_{\mathcal{U}_{a,x, f}\times_{\mathfrak{o}}\overline{\mathcal{T}}\times_{\mathfrak{o}} \mathcal{U}_{-a,x, f}}(d_{a})\longrightarrow \mathcal{U}_{-a,x, f}\times_{\mathfrak{o}}\overline{\mathcal{T}}\times_{\mathfrak{o}} \mathcal{U}_{a,x, f}$$
    such that $\beta_a(\mathcal{U}_{-a,x, f^+}(\mathfrak{o})\times\overline{\mathcal{T}}(\mathfrak{o})\times \mathcal{U}_{a,x, f^+}(\mathfrak{o}))\subset \mathcal{U}_{a,x, f^+}(\mathfrak{o})\times\overline{\mathcal{T}}(\mathfrak{o})\times \mathcal{U}_{-a,x, f^+}(\mathfrak{o}).$
\end{lemma}

\begin{proof}
    Let $\pi: G^a \longrightarrow \langle U_a, U_{-a} \rangle$ be the simply connected covering. In particular, $\pi$ maps the root subgroups of $G_a$ isomorphically onto $U_a$ and $U_{-a}$. Then $G^a$ is a quasi-split semisimple group over $k$, which contains $\pi^{-1}(T)$ as a maximal $k$-torus. Since the preimages of $a$ in $\widetilde{\Phi}$ are transitively permuted by $\Gal(k_s/k)$ (\cite[4.1.2]{Bruhattits2}), by verifying the Dynkin diagrams, there are only two possibilities of $G^a$.
    
    \textbf{Case 1: $2a\notin \Phi$}; 
    By \cite[1. Case at page 41]{Landvogtcompactification} or \cite[4.1.4]{Bruhattits2}, $G^a_{k_s}$ is isomorphic to a product of $\SL_{2,k_s}$ indexed by the preimages of $a$ in $\widetilde{\Phi}$ which are transitively permuted by the Galois group $\Gal(k_s/k)$. Let $G^{\widetilde{a}}_{k_s}$ be the factor of $G^a_{k_s}$ indexed by $\widetilde{a}$, which is defined over $k_{\widetilde{a}}$. We denote by $G^{\widetilde{a}}$ the $k_{\widetilde{a}}$-form of $G^{\widetilde{a}}_{k_s}$. By descent (\cite[1.5.16]{Bruhattits2}), we have 
    $$G^a\cong \Res_{k_{{\widetilde{a}}}/k}(G^{\widetilde{a}}).$$
    Then the $\SL_2$-crutch is available. More precisely, by \cite[4.1.5]{Bruhattits2}, there is a unique $k_{{\widetilde{a}}}$-isomorphism $\xi_{{\widetilde{a}}}: \SL_{2,k_{\widetilde{a}}}\rightarrow G^{{\widetilde{a}}}$ such that 
    $$\widetilde{\chi}_{\pm\widetilde{a}}=\pi_{k_{\widetilde{a}}} \cdot \xi_{\widetilde{a}}\cdot y_{\pm},$$
    where $\widetilde{\chi}_{\pm\widetilde{a}}$ is defined over $k_{\widetilde{a}}$ (\Cref{ChevalleySteinbergsystem} (1)) and $y_{\pm}: \mathbb{G}_{a,k_{\widetilde{a}}}\longrightarrow \SL_{2,k_{\widetilde{a}}}$ are the canonical pinning of $\SL_{2,k_{\widetilde{a}}}$. Let
    \begin{align*}
      T_a: \Res_{k_{\widetilde{a}}/k}\mathbb{G}_{m,k_{\widetilde{a}}} \longrightarrow \Res_{k_{\widetilde{a}}/k}(\SL_{2,k_{\widetilde{a}}}) \xlongrightarrow{\Res_{k_{\widetilde{a}}/k}(\xi_a)}     \Res_{k_{\widetilde{a}}/k}G^{\widetilde{a}}\cong G^a\xlongrightarrow{\pi}  \langle U_a, U_{-a}\rangle\bigcap T,
    \end{align*}
    where the first arrow is the Weil restriction of the canonical morphism $\mathbb{G}_{m,k_{\alpha}}\rightarrow \SL_{2,k_{\alpha}}$ by sending $t$ to $\begin{bmatrix}
        t & 0 \\ 0 & t^{-1}
    \end{bmatrix}.$
    By \emph{loc. cit.,} 
    we have the isomorphisms over $k$:
    $$\chi_{\pm a}\coloneq \Res_{k_{\widetilde{a}}/k}(\widetilde{\chi}_{\pm\widetilde{a}}):\Res_{k_{\widetilde{a}}/k}(\mathbb{G}_{a,k_{\widetilde{a}}})\rightarrow U_{\pm a},$$
    where we recall that by our choice $\widetilde{\chi}_{\pm\widetilde{a}}$ is defined over $k_{\widetilde{a}}$ (\Cref{ChevalleySteinbergsystem} (1)). By Galois descent and \cite[4.4.19~(1)]{Bruhattits2}, we have, for any $u\in k_{\widetilde{a}}$ and $t\in T(k)$,
    \begin{align}\label{torusaction1}
        t\cdot \chi_{\pm a}(u)\cdot t^{-1}=\chi_{\pm a}(m_{-\widetilde{a}}(\nu(t))^{\mp 1}u).
    \end{align}
    In this case, we define 
    \begin{align*}
        d_a: U_{a}\times_k \overline{T}\times_k U_{-a} &\longrightarrow \mathbb{A}_{1,k}\\
        (\chi_a(u), \overline{t}, \chi_{-a}(u'))&\longmapsto \Norm_{k_{\widetilde{a}}/k}(1-m_{-\widetilde{a}}(\overline{t}) uu'),
    \end{align*}
    where $u,u'\in k_{\widetilde{a}}$, $\overline{t}\in \overline{T}(k)$ and $m_{-\widetilde{a}}$ is defined by \Cref{multiplicationdescent}. We further define a $k$-morphism 
    \begin{equation}\label{definitiontheta1}
        \begin{aligned}
        \theta_a: D_{U_a\times_k \overline{T} \times_k U_{-a}}(d_a)&\longrightarrow U_{-a}\overline{T} U_a\subset \overline{G}\\
         (\chi_a(u), \overline{t}, \chi_{-a}(u'))&\longmapsto \chi_{-a}\left(\frac{m_{-\widetilde{a}}(\overline{t})u'}{1- m_{-\widetilde{a}}(\overline{t})uu'}\right)T_a(1-m_{-\widetilde{a}}(\overline{t})uu')\overline{t} \chi_{a}\left(\frac{m_{-\widetilde{a}}(\overline{t})u}{1-m_{-\widetilde{a}}(\overline{t}) uu'}\right),
    \end{aligned}
    \end{equation}
    where $u, u'\in k_{\widetilde{a}}$ and $\overline{t}\in \overline{T}(k)$.
    We now verify that $\theta_a$ coincides with the $(G\times_k G)$-action on $\overline{G}$. For this, since $T$ is dense in $\overline{T}$, it suffices to assume that $\overline{t}\in T(k)$, i.e., there exists a $t\in T(k)$ such that $\overline{t}=\nu(t)$. Then the claim follows from the following computations: for any $(\chi_a(u), \overline{t}, \chi_{-a}(u'))\in D_{U_a\times_k \overline{T} \times_k U_{-a}}(d_a)$,
    \begin{equation*}
    \begin{split}
         \chi_{a}(u)\cdot \overline{t}\cdot \chi_{-a}(u') &= \chi_{a}(u)\cdot (t \cdot \chi_{-a}(u')\cdot t^{-1})\cdot t\\
        &=\chi_{a}(u)\cdot \chi_{-a}(m_{-\widetilde{a}}(\nu(t))u')\cdot t \\
        &= \chi_{-a}\left(\frac{m_{-\widetilde{a}}(\nu(t))u'}{1-m_{-\widetilde{a}}(\nu(t))uu'}\right)\cdot T_a(1-m_{-\widetilde{a}}(\nu(t))uu')\cdot \chi_{a}\left(\frac{u}{1-m_{-\widetilde{a}}(\nu(t))uu'}\right)\cdot t\\
        &= \chi_{-a}\left(\frac{m_{-\widetilde{a}}(\nu(t))u'}{1-m_{-\widetilde{a}}(\nu(t))uu'}\right)\cdot T_a(1-m_{-\widetilde{a}}(\nu(t))uu')\cdot t\cdot \chi_{a}\left(\frac{m_{-\widetilde{a}}(\nu(t^{-1}))^{-1}u}{1-m_{-\widetilde{a}}(\nu(t))uu'}\right)\\
        &= \chi_{-a}\left(\frac{m_{-\widetilde{a}}(\nu(t))u'}{1-m_{-\widetilde{a}}(\nu(t))uu'}\right)\cdot T_a(1-m_{-\widetilde{a}}(\nu(t))uu')\cdot t\cdot \chi_{a}\left(\frac{m_{-\widetilde{a}}(\nu(t))u}{1-m_{-\widetilde{a}}(\nu(t))uu'}\right), 
    \end{split}
    \end{equation*}
    where in the second and fourth equalities, we use \Cref{torusaction1}; in the third equality, we use \cite[4.1.6~(2)]{Bruhattits2} which is simply a calculation in $\SL_2$; in the last equality, the fact that $m_{-\widetilde{a}}(\nu(t))=(m_{-\widetilde{a}}(\nu(t^{-1})))^{-1}$ is used.

    To see $d_a\in\mathfrak{o}[\mathcal{U}_{-a,x,f}\times_{\mathfrak{o}}\overline{\mathcal{T}}\times_{\mathfrak{o}} \mathcal{U}_{a,x,f}]$, by the extension principle \Cref{extensionprinciple}, it suffices to show that
    $$d_a(\mathcal{U}_{-a,x,f}(\mathfrak{o})\times\overline{\mathcal{T}}(\mathfrak{o})\times \mathcal{U}_{a,x,f}(\mathfrak{o}))\subset \mathfrak{o}.$$
    By \cite[4.2.2]{Bruhattits2} or \cite[4.8, 4.9]{Landvogtcompactification}, 
    $$x_a(\chi_a(u))=\omega(u)+v(a), \;u\in k_{\widetilde{a}}.$$ By \Cref{multiplicationintegral}, for $\overline{t}\in \overline{\mathcal{T}}(\mathfrak{o})$, we have $m_{-\widetilde{a}}(\overline{t})\in \mathfrak{o}_{\widetilde{a}}$.
    Hence, for $u, u'\in k_{\widetilde{a}}$ with $\omega(u)+ v(a)\geq f(a)$ (resp., $\textgreater f(a) $) and $\omega(u')+ v(-a)\geq f(-a)$ (resp., $\textgreater f(a) $), since $f$ is a concave function, 
    we have $\omega(m_{-\widetilde{a}}(\overline{t})uu')\geq 0$ (resp., $\textgreater 0 $), hence $1-m_{-\widetilde{a}}(\overline{t})uu'\in\mathfrak{o}_{\widetilde{a}}$ (resp., $\in \mathfrak{o}_{\widetilde{a}}^\times$). Therefore $d_a(\chi_a(u), \overline{t}, \chi_{-a}(u'))=\Norm_{k_{\widetilde{a}}/k}(1-m_{-\widetilde{a}}(\overline{t}) uu')\in \mathfrak{o}$ (resp., $\in \mathfrak{o}^\times$). Hence $d_a\in\mathfrak{o}[\mathcal{U}_{-a,x,f}\times_{\mathfrak{o}}\overline{\mathcal{T}}\times_{\mathfrak{o}} \mathcal{U}_{a,x,f}]$ and (1). For claim (2), it suffices to note that, if $u=u'=0$, then $d_a(\chi_a(u), \overline{t}, \chi_{-a}(u'))=1$.

    To see that the $k$-morphism $\theta_a$ extends to an $\mathfrak{o}$-morphism $ D_{\mathcal{U}_{a,x, f}\times_{\mathfrak{o}}\overline{\mathcal{T}}\times_{\mathfrak{o}} \mathcal{U}_{-a,x, f}}(d_{a})\longrightarrow \mathcal{U}_{-a,x, f}\times_{\mathfrak{o}}\overline{\mathcal{T}}\times_{\mathfrak{o}} \mathcal{U}_{a,x, f}$, by \Cref{extensionprinciple}, it suffices to check that 
     $$\theta_a(D_{\mathcal{U}_{a,x, f}\times_{\mathfrak{o}}\overline{\mathcal{T}}\times_{\mathfrak{o}} \mathcal{U}_{-a,x, f}}(d_{a})(\mathfrak{o}))\subset \mathcal{U}_{-a,x, f}(\mathfrak{o})\times\overline{\mathcal{T}}(\mathfrak{o})\times \mathcal{U}_{a,x, f}(\mathfrak{o}).$$
     For this, it suffices to combine \Cref{definitiontheta1} with \Cref{multiplicationintegral}.
    
    \textbf{Case 2: $2a\in \Phi$}. The preimage of $a$ in $\widetilde{\Phi}$ consists of a set $I$ of pairs of orthogonal absolute roots. In this case, $G^a_{k_s}$ is isomorphic to a product of $\SL_3$ indexed by $I$ and $\Gal(k_s/k)$ acts on $G^a_{k_s}$ by permuting the factors, and we fix a factor $G^{\widetilde{a},\widetilde{a}'}_{k_s}\subset G^a_{k_s}$ associated to two absolute roots $\widetilde{a},\widetilde{a}'\in\widetilde{\Phi}$ satisfying $\widetilde{a}\vert_S=\widetilde{a}'\vert_S=a$ and $\widetilde{a}+\widetilde{a}'\in\widetilde{\Phi}$. Since $\Gal(k_s/k_{\widetilde{a}+\widetilde{a}'})$ stabilizes $\widetilde{a}+\widetilde{a}'$, hence $G^{\widetilde{a},\widetilde{a}'}$ is defined over $k_{\widetilde{a}+\widetilde{a}'}$; we denote by $G^{\widetilde{a},\widetilde{a}'}$ the $k_{\widetilde{a}+\widetilde{a}'}$-form of $G^{\widetilde{a},\widetilde{a}'}_{k_s}$. Since $\Gal(k_s/k_{\widetilde{a}})$ acts trivially on $\{\widetilde{a},\widetilde{a}'\}$, hence $G^{\widetilde{a},\widetilde{a}'}$ splits over $k_{\widetilde{a}}$ (see, for instance, \cite[16.2.1]{linearalgrpSpringer}). Also, we have that $k_{\widetilde{a}}=k_{\widetilde{a}'}$ is a quadratic separable extension of $k_{\widetilde{a}+\widetilde{a}'}$, and we let $\sigma\in \Gal(k_{\widetilde{a}}/k_{\widetilde{a}+\widetilde{a}'})$ be the nontrivial element. Thus $G^{\widetilde{a},\widetilde{a}'}$ is isomorphic to the special unitary group $\SU_{3,k_{\widetilde{a}}/k_{\widetilde{a}+\widetilde{a}'}}$ of the hermitian form 
    $$h:(x_{-1},x_0,x_1)\mapsto \sigma(x_{-1})x_1+\sigma(x_0)x_0+\sigma(x_1)x_{-1},$$
    where $(x_{-1},x_0,x_1)$ runs over $k_{\widetilde{a}}$-vector space $k_{\widetilde{a}}^3$, see \cite[4.1.4~cas II]{Bruhattits2} or \cite[2. Case at page~43]{Landvogtcompactification}.
    Furthermore, by descent (\cite[1.5.16]{Bruhattits2}), $G^a$ is isomorphic to $\Res_{k_{\widetilde{a}+\widetilde{a}'}/k}(G^{\widetilde{a},\widetilde{a}'})$ over $k$. Now we have the following $\SU_3$-crutch.

    Let $H_0(k_{\widetilde{a}},k_{\widetilde{a}+\widetilde{a}'})$ be the $k_{\widetilde{a}+\widetilde{a}'}$-group scheme defined in \cite[4.1.9]{Bruhattits2} or \cite[Lemma~4.13]{Landvogtcompactification} with the $k_{\widetilde{a}+\widetilde{a}'}$-points $\{(u,v)\in k_{\widetilde{a}}\times k_{\widetilde{a}}\vert u\sigma(u)=v+\sigma(v)\}$. We consider the canonical parameterizations of root subgroups of $\SU_{3,k_{\widetilde{a}}/k_{\widetilde{a}+\widetilde{a}'}}$ (corresponding to two indivisible roots) and the maximal torus:
    \begin{equation*}
        y_+: H_0(k_{\widetilde{a}},k_{\widetilde{a}+\widetilde{a}'})\longrightarrow \SU_{3,k_{\widetilde{a}}/k_{\widetilde{a}+\widetilde{a}'}}, \;\; (u,v)\longmapsto 
        \begin{pmatrix}
        1 & -\sigma(u) & -v\\
        0 & 1 & u\\
        0 & 0 & 1
       \end{pmatrix};
    \end{equation*}
    \begin{equation*}
        y_-: H_0(k_{\widetilde{a}},k_{\widetilde{a}+\widetilde{a}'})\longrightarrow \SU_{3,k_{\widetilde{a}}/k_{\widetilde{a}+\widetilde{a}'}}, \;\; (u,v)\longmapsto 
        \begin{pmatrix}
        1 & 0 & 0\\
        u & 1 & 0\\
        -v& -\sigma(u) & 1
       \end{pmatrix};
    \end{equation*}
    \begin{equation*}
        z: \Res_{k_{\widetilde{a}}/k_{\widetilde{a}+\widetilde{a}'}}(\mathbb{G}_{m,k_{\widetilde{a}}})\longrightarrow \SU_{3,k_{\widetilde{a}}/k_{\widetilde{a}+\widetilde{a}'}}, \;\; t\longmapsto 
        \begin{pmatrix}
        t & 0 & 0\\
        0 & t^{-1}\sigma(t) & 0\\
        0 & 0 & \sigma(t)^{-1}
       \end{pmatrix}.
    \end{equation*}
By \cite[4.1.9]{Bruhattits2} or \cite[4.20]{Landvogtcompactification}, let $\xi: \SU_{3,k_{\widetilde{a}}/k_{\widetilde{a}+\widetilde{a}'}}\longrightarrow G^{\widetilde{a},\widetilde{a}'}$ be the unique morphism satisfying 
$$\pi_{k_{k_{\widetilde{a}+\widetilde{a}'}}}\circ \xi \circ y_{\pm}=\widetilde{\chi}_{\pm\widetilde{a},\pm\widetilde{a}'}(u,v) \coloneq \widetilde{\chi}_{\pm\widetilde{a}}(u)\widetilde{\chi}_{\pm\widetilde{a}\pm\widetilde{a}'}(-v)\widetilde{\chi}_{\pm\widetilde{a}'}(\sigma(u)),$$
where $\widetilde{\chi}_{\pm\widetilde{a}}$ and $\widetilde{\chi}_{\pm\widetilde{a}'}$ are defined over $k_{\widetilde{a}}$ (\Cref{ChevalleySteinbergsystem} (1) and (2)), in particular, $\widetilde{\chi}_{\pm\widetilde{a},\pm\widetilde{a}'}$ is defined over $k_{\widetilde{a}+\widetilde{a}'}$. Also let 
$$T_a\coloneq \pi\circ\Res_{k_{\widetilde{a}+\widetilde{a}'}/k}(\xi \circ z): \Res_{k_{\widetilde{a}}/k}(\mathbb{G}_{m,k_{\widetilde{a}}})\longrightarrow \langle U_a, U_{-a} \rangle\bigcap T\subset G.$$
Moreover, we have the $k$-isomorphisms 
$$\chi_{\pm a}\coloneq \Res_{k_{\widetilde{a}+\widetilde{a}'}/k}(\widetilde{\chi}_{\pm\widetilde{a},\pm\widetilde{a}'}):\Res_{k_{\widetilde{a}+\widetilde{a}'}/k}(H_0(k_{\widetilde{a}},k_{\widetilde{a}+\widetilde{a}'}))\longrightarrow U_{\pm a},$$ see \cite[page~45]{Landvogtcompactification} or \cite[4.1.9]{Bruhattits2}.
Since $t\cdot \widetilde{\chi}_{\pm\widetilde{a},\pm\widetilde{a}'}(u,v)\cdot t^{-1}=\widetilde{\chi}_{\pm\widetilde{a},\pm\widetilde{a}'}(\widetilde{a}(t)^{\pm1}u,\widetilde{a}(t)^{\pm1}\widetilde{a}'(t)^{\pm1}v)$ for $(u,v)\in H_0(k_{\widetilde{a}},k_{\widetilde{a}+\widetilde{a}'})(k_{\widetilde{a}+\widetilde{a}'})$ and $t\in T(k_{\widetilde{a}})$, by Galois descent (\Cref{multiplicationdescent}), we have that 
\begin{align}\label{torusactionSU}
    t\cdot \chi_{\pm a}(u,v)\cdot t^{-1}=\chi_{\pm a}(m_{- \widetilde{a}}(\nu(t))^{\mp 1}u,m_{-\widetilde{a}}(\nu(t))^{\mp 1}m_{-\widetilde{a}'}(\nu(t))^{\mp 1}v),
\end{align}
where $(u,v)\in \Res_{k_{\widetilde{a}+\widetilde{a}'}/k} (H_0(k_{\widetilde{a}},k_{\widetilde{a}+\widetilde{a}'}))(k)$.
Now we define 
\begin{align*}
        d_{a}: U_{a}\times_k \overline{T}\times_k U_{-a} &\longrightarrow \mathbb{A}_{1,k}\\
        (\chi_a(u,v), \overline{t}, \chi_{-a}(u',v'))&\longmapsto \Norm_{k_{\widetilde{a}/k}}(\epsilon), 
    \end{align*}
where $\epsilon\coloneq 1-m_{-\widetilde{a}}(\overline{t})\sigma(u)u'+m_{-\widetilde{a}}(\overline{t}) \sigma(m_{-\widetilde{a}}(\overline{t}))vv'$. We further define a $k$-morphism
\begin{align*}
        \theta_a: D_{U_a\times_k \overline{T} \times_k U_{-a}}(d_{a}) \longrightarrow U_{-a}\overline{T} U_a\subset \overline{G} \;\;\text{by sending} \;\; (\chi_a(u,v), \overline{t}, \chi_{-a}(u',v'))\;\; \text{to}
    \end{align*}
    $$\chi_{-a}\left(\frac{m_{-\widetilde{a}}(\overline{t})(u'-\sigma(m_{-\widetilde{a}}(\overline{t}))uv')}{\epsilon},\frac{m_{-\widetilde{a}}(\overline{t})\sigma(m_{-\widetilde{a}}(\overline{t}))v'}{\epsilon}\right)\cdot T_a(\epsilon)\cdot \overline{t} \cdot \chi_{a}\left(\frac{m_{-\widetilde{a}}(\overline{t})(u-m_{-\widetilde{a}}(\overline{t})\sigma(v)u')}{\sigma(\epsilon)},\frac{m_{-\widetilde{a}}(\overline{t})\sigma(m_{-\widetilde{a}}(\overline{t}))v}{\epsilon}\right),$$
    where $u,v, u',v'\in k_{\widetilde{a}}$ and $\overline{t}\in \overline{T}(k)$.

    We now verify that $\theta_a$ coincides with the $(G\times_k G)$-action on $\overline{G}$. For this, since $T$ is open dense in $\overline{T}$, it suffices to assume that $\overline{t}\in T(k)$, i.e., there exists a $t\in T(k)$ such that $\overline{t}=\nu(t)$. Note that this gives 
    $$\epsilon=1-m_{-\widetilde{a}}(\nu(t))\sigma(u)u'+m_{-\widetilde{a}}(\nu(t)) \sigma(m_{-\widetilde{a}}(\nu(t)))vv'.$$ 
    Then, the claim follows from the following computations (and the equality $\sigma(m_{-\widetilde{a}'}(\overline{t}))=m_{-\widetilde{a}}(\overline{t})$):
    \begin{equation*}
        \begin{split}
             &\chi_a(u,v)\cdot \overline{t}\cdot \chi_{-a}(u',v') = \chi_a(u,v)\cdot (t\cdot \chi_{-a}(u',v')\cdot t^{-1})\cdot t\\
            &=\chi_a(u,v)\cdot \chi_{-a}(m_{-\widetilde{a}}(\nu(t))u', m_{-\widetilde{a}}(\nu(t))m_{-\widetilde{a}'}(\nu(t))v')\cdot t\\
            &= \chi_{-a}\left(\frac{m_{-\widetilde{a}}(\nu(t))u'-m_{-\widetilde{a}}(\nu(t))m_{-\widetilde{a}'}(\nu(t))uv'}{\epsilon}, \frac{m_{-\widetilde{a}}(\nu(t))m_{-\widetilde{a}'}(\nu(t))v'}{\epsilon}\right)\cdot T_a(\epsilon)\cdot \\
            &\chi_a\left(\frac{u-m_{-\widetilde{a}}(\nu(t))u'\sigma(v)}{\sigma(\epsilon)},\frac{v}{\epsilon}\right) \cdot t\\
            &=\chi_{-a}\left(\frac{m_{-\widetilde{a}}(\nu(t))u'-m_{-\widetilde{a}}(\nu(t))m_{-\widetilde{a}'}(\nu(t))uv'}{\epsilon}, \frac{m_{-\widetilde{a}}(\nu(t))m_{-\widetilde{a}'}(\nu(t))v'}{\epsilon}\right)\cdot T_a(\epsilon)\cdot t\\
            &\cdot \chi_a\left(\frac{m_{-\widetilde{a}}(\nu(t))(u-m_{-\widetilde{a}}(\nu(t))u'\sigma(v))}{\sigma(\epsilon)},\frac{m_{-\widetilde{a}}(\nu(t))m_{-\widetilde{a}'}(\nu(t))v}{\epsilon}\right), 
        \end{split}
    \end{equation*}
    where the second and the last equalities follow from \Cref{torusactionSU}; the third equality uses \cite[4.1.12]{Bruhattits2} which follows from a simple computation in $\SU_3$; in the last equality, we use the fact that $m_{-\widetilde{a}}(\nu(t))=(m_{-\widetilde{a}}(\nu(t^{-1})))^{-1}$ and $m_{-\widetilde{a}'}(\nu(t))=(m_{-\widetilde{a}'}(\nu(t^{-1})))^{-1}$.

    To see $d_{a}\in\mathfrak{o}[\mathcal{U}_{-a,x,f}\times_\mathfrak{o}\overline{\mathcal{T}}\times_\mathfrak{o} \mathcal{U}_{a,x,f}]$, by the extension principle \Cref{extensionprinciple}, it suffices to show that 
    $$d_a(\mathcal{U}_{-a,x,f}(\mathfrak{o})\times\overline{\mathcal{T}}(\mathfrak{o})\times \mathcal{U}_{a,x,f}(\mathfrak{o}))\subset \mathfrak{o}.$$ 
     Let $\gamma\coloneq -\sup\{\omega(b)\vert b\in k_\alpha, \Tr_{k_\alpha/k_{\alpha+\alpha'}}(b)=1\}/2$ which is nonnegative by \cite[Lemme~4.3.3(iii)]{Bruhattits2}. 
    For $\chi_a(u,v)\in \mathcal{U}_{a,x,f}(\mathfrak{o})$ and  $\chi_{-a}(u',v')\in \mathcal{U}_{-a,x,f}(\mathfrak{o})$, by \cite[4.3.5~(3)]{Bruhattits2}, we have 
    $$\omega(u)\geq f(a)+\gamma,\;\omega(v)\geq f(2a),\;\omega(u')\geq f(-a)+\gamma,\;\omega(v')\geq f(-2a).$$
    By \Cref{multiplicationintegral}, for a $\overline{t}\in \overline{\mathcal{T}}(\mathfrak{o})$, we have that, $m_{-\widetilde{a}}(\overline{t})\in\mathfrak{o}_\alpha$. Since $f$ is a concave function, we obtain $$\epsilon=1-m_{-\widetilde{a}}(\overline{t})\sigma(u)u'+m_{-\widetilde{a}}(\overline{t}) \sigma(m_{-\widetilde{a}}(\overline{t}))vv'\in\mathfrak{o}_{\alpha},$$
    hence $\Norm_{k_{\alpha+\alpha'}/k}(\epsilon)\in\mathfrak{o}$, as desired.

    To see that the $k$-morphism $\theta_a$ extends to an $\mathfrak{o}$-morphism $ D_{\mathcal{U}_{a,x, f}\times_{\mathfrak{o}}\overline{\mathcal{T}}\times_{\mathfrak{o}} \mathcal{U}_{-a,x, f}}(d_{a })\longrightarrow \mathcal{U}_{-a,x, f}\times_{\mathfrak{o}}\overline{\mathcal{T}}\times_{\mathfrak{o}} \mathcal{U}_{a,x,f}$, by \Cref{extensionprinciple}, it suffices to check that 
     $$\theta_a(D_{\mathcal{U}_{a,x, f}\times_{\mathfrak{o}}\overline{\mathcal{T}}\times_{\mathfrak{o}} \mathcal{U}_{-a,x, f}}(d_{a })(\mathfrak{o}))\subset \mathcal{U}_{-a, x, f}(\mathfrak{o})\times\overline{\mathcal{T}}(\mathfrak{o})\times \mathcal{U}_{a,x, f}(\mathfrak{o}).$$
     For this, as in the Case 1, by the definition of $\theta_a$, we only need to combine the definition of $\theta_a$ with \Cref{multiplicationintegral} and the extension principle \Cref{extensionprinciple}.
\end{proof}

\subsubsection{Distribution computation}
Now we unify the commutator relations in \Cref{suslcase} via the language of distribution which is recalled in \Cref{appendixdistribution}. The following result is inspired by \cite[\S~3.4]{Bruhattits2} and \cite[Proposition~5.13]{Landvogtcompactification}.

For a set of positive roots  $\Psi\subset \Phi$, let 
\begin{equation}\label{integralunipotentconvention}
    \mathcal{U}_{\Psi,x,f}\coloneq \displaystyle\prod_{r\in\Psi^{\red}}\mathcal{U}_{r,x,f}
\end{equation}
which does not depend on the choice of an order on $\Psi^{\red}$. We will use the natural inclusion $ \overline{\mathcal{T}}(\mathfrak{o})\subset \overline{T}(k),\; \mathcal{U}_{\Psi,x,f}(\mathfrak{o})\subset U_{\Psi}(k)$ (see, for instance, \cite[0.2]{Landvogtcompactification}).
 We denote the group action on $\overline{G}$ by 
$$\Multi: G\times_k \overline{G}\times_k  G\longrightarrow \overline{G}.$$

\begin{proposition}\label{distributioninvariant} 
    For any set of positive roots  $\Psi\subset \Phi$, we have that
    \begin{itemize}
        \item[(1)] $\Multi(\mathcal{U}_{-\Psi,x,f^+}(\mathfrak{o})\times\overline{\mathcal{T}}(\mathfrak{o})\times\mathcal{U}_{\Psi,x,f^+}(\mathfrak{o}))\subset \mathcal{U}_{\Phi^-,x,f^+}(\mathfrak{o})\times\overline{\mathcal{T}}(\mathfrak{o})\times\mathcal{U}_{\Phi^+,x,f^+}(\mathfrak{o})\subset \overline{\Omega_{x,f}}(\mathfrak{o})$;
        
        \item[(2)]  for $g_1\in\mathcal{U}_{-\Psi,x,f^+}(\mathfrak{o}), g_2\in \mathcal{U}_{\Psi,x,f^+}(\mathfrak{o})$ and $ \tau\in \overline{\mathcal{T}}(\mathfrak{o})$, inside $\Dist(\overline{\Omega},g_1 \tau g_2)\cong\Dist(\overline{\Omega_{x,f}},g_1 \tau g_2)\otimes_{\mathfrak{o}}k$ 
        we have 
        $d(\Multi)_{(g_1,\tau,g_2)}(\Dist(\mathcal{U}_{-\Psi,x,f},g_1)\otimes_{\mathfrak{o}} \Dist(\overline{\mathcal{T}},\tau)\otimes_{\mathfrak{o}}\Dist(\mathcal{U}_{\Psi,x,f},g_2))\subset \Dist(\overline{\Omega_{x,f}},g_1 \tau g_2).$
    \end{itemize}
\end{proposition}

\begin{proof}
    Suppose that $\Psi\subset\Phi$ is a set of positive roots. We proceed by induction on the cardinality of $\Psi\bigcap \Phi^-$. When $\Psi\bigcap \Phi^-=\varnothing$, there is nothing to prove. Now suppose that $\Psi\bigcap \Phi^-\neq\varnothing$, then there exists an indivisible root $a\in \Psi\bigcap \Phi^-$. Let us consider 
    $$\Psi'\coloneq (\Psi\backslash\{\mathbb{R}_{>0}a\bigcap \Psi\})\bigcup\{-(\mathbb{R}_{>0}a\bigcap \Psi)\},$$
    which is a set of positive root of $\Phi$ by \cite[chapitre~VI, n$^o$~1.4 proposition~13, n$^o$~1.6 proposition~17]{bourbakigal}. We have the following commutative diagram of rational morphisms
    $$\begin{tikzcd}[sep=small]
\displaystyle\prod_{r\in-\Psi^{\red}\backslash -a} U_r\times_k (U_{-a}\times_k \overline{T}\times_k U_a)\times_k \displaystyle\prod_{r\in\Psi^{\red}\backslash a} U_r\arrow[rd] \arrow[rr,"B_a"] && \displaystyle\prod_{r\in-\Psi'^{\red}} U_r\times_k \overline{T}\times_k \displaystyle\prod_{r\in-\Psi'^{\red}} U_r \arrow[ld]\\
& \overline{\Omega}\coloneq U^-\times_k \overline{T} \times_k U^+\subset \overline{G} 
\end{tikzcd}$$
where $B_a$ rationally sends $U_{-a}\times_k\overline{T}\times_k U_a$ to $U_{a}\times_k\overline{T}\times_k U_{-a}$ via $\Multi$ and keeps the other components unchanged and the other two arrows indicate the rational morphism given by $\Multi$.
   
   The claim (1) follows from \Cref{suslcase}, the induction hypothesis and the above commutative diagram.
   To verify (2), let $g_1\in\mathcal{U}_{-\Psi,x,f^+}(\mathfrak{o}), g_2\in \mathcal{U}_{\Psi,x,f^+}(\mathfrak{o}), \tau\in \overline{\mathcal{T}}(\mathfrak{o})$. We write 
    $$g_1=u^-\cdot v^+\in(\prod_{r\in -\Psi^{\red}\backslash -a} \mathcal{U}_{r,x,f})(\mathfrak{o})\times \mathcal{U}_{-a,x,f}(\mathfrak{o})$$
    $$g_2=v^-\cdot u^+\in\mathcal{U}_{a,x,f^+}(\mathfrak{o})\times (\prod_{r\in \Psi^{\red}\backslash a} \mathcal{U}_{r,x,f^+})(\mathfrak{o}) $$
    By \Cref{suslcase}, we can write $\beta_a(v^+,\tau,v^-)=\widetilde{v}^-\tau'\widetilde{v}^+\in(\mathcal{U}_{a,x,f^+}\times_{\mathfrak{o}}\overline{\mathcal{T}}\times_{\mathfrak{o}}\mathcal{U}_{-a,x,f^+})(\mathfrak{o})$ and let 
    $$g_1'\coloneq u^-\widetilde{v}^-\in \mathcal{U}_{-\Psi',x,f^+}(\mathfrak{o})\;\;\;\text{and}\;\;\; g_2'\coloneq\widetilde{v}^+u^+\in\mathcal{U}_{\Psi',x,f^+}(\mathfrak{o}).$$

   We have that (the tensor product is over $\mathfrak{o}$)
    \begin{equation*}
        \begin{split}
          & d(\Multi)_{(g_1,\tau,g_2)}(\Dist(\mathcal{U}_{-\Psi,x,f},g_1)\otimes \Dist(\overline{\mathcal{T}},\tau)\otimes\Dist(\mathcal{U}_{\Psi,x,f},g_2))\\
          = &d(\Multi)_{(g_1,\tau,g_2)}(\Dist(\mathcal{U}_{-\Psi\backslash -a,x,f},u^-)\otimes\Dist(\mathcal{U}_{-a,x,f},v^+)\otimes\Dist(\overline{\mathcal{T}},\tau)\\
          &\otimes\Dist(\mathcal{U}_{a,x,f} ,v^-)\otimes\Dist(\mathcal{U}_{\Psi\backslash a,x,f},u^+))\\
          =& d(\Multi)_{(u^-,v^+\tau v^-,u^+)}(\Dist(\mathcal{U}_{-\Psi\backslash -a,x,f},u^-)\otimes d(\Multi)_{(v^+,\tau,v^-)}(\Dist(\mathcal{U}_{-a,x,f},v^+)\otimes\Dist(\overline{\mathcal{T}},\tau)\\  
          &\otimes\Dist(\mathcal{U}_{a,x,f} ,v^-))\otimes\Dist(\mathcal{U}_{\Psi\backslash a,x,f},u^+))\\
          =& d(\Multi)_{(u^-,v^+\tau v^-,u^+)}(\Dist(\mathcal{U}_{-\Psi\backslash -a,x,f},u^-)\otimes d(\Multi\circ (\beta_a)_k)_{(v^+,\tau,v^-)}(\Dist(\mathcal{U}_{-a,x,f},v^+)\otimes\Dist(\overline{\mathcal{T}},\tau)\\  
          &\otimes\Dist(\mathcal{U}_{a,x,f} ,v^-))\otimes\Dist(\mathcal{U}_{\Psi\backslash a,x,f},u^+))\\
        \subset & d(\Multi)_{(u^-, \widetilde{v}^-\tau'\widetilde{v}^+,u^+)}(\Dist(\mathcal{U}_{-\Psi\backslash -a,x,f},u^-)\otimes d(\Multi)_{(\widetilde{v}^-,\overline{\tau},\widetilde{v}^+)}(\Dist(\mathcal{U}_{a,x,f},\widetilde{v}^-)\otimes\Dist(\overline{\mathcal{T}},\tau')\\
          &\otimes\Dist(\mathcal{U}_{-a,x,f} ,\widetilde{v}^+))\otimes\Dist(\mathcal{U}_{\Psi\backslash a,x,f},u^+))\\
        =&d(\Multi)_{(g_1',\tau',g_2')}(\Dist(\mathcal{U}_{-\Psi',x,f},g_1')\otimes \Dist(\overline{\mathcal{T}},\tau')\otimes\Dist(\mathcal{U}_{\Psi',x,f},g_2'))\\
       \subset&  \Dist(\overline{\Omega_{x,f}},g_1 \tau g_2)
        \end{split}
    \end{equation*}
    hold, 
    where the second and fifth equalities come from the associativity of the group action; in the third equality, $\beta_a$ comes from \Cref{suslcase}; the fourth inclusion follows from \Cref{suslcase} (3) and \Cref{corollaryextension}, and also we implicitly use \Cref{distributionlocalproperty}; the last inclusion follows from the induction hypothesis.
\end{proof}

\subsubsection{Reconstruction of rational action}
Now we are in the position to define a rational action and study its definition domain. The following result is inspired by \cite[Page~60]{Landvogtcompactification}. Recall that the open $\mathfrak{o}$-dense subscheme $\Omega_{x,f}\subset\mathcal{G}_{x,f}$ is viewed as an open subscheme of $\overline{\Omega
_{x,f}}$ via \Cref{eqembeddingofbigcell}.

\begin{theorem}\label{theoremrationalaction}
    The $\mathfrak{o}$-rational morphism
    $$A: \mathcal{G}_{x,f}\times_\mathfrak{o}\overline{\Omega_{x,f}}\times_\mathfrak{o}\mathcal{G}_{x,f}\dashrightarrow \overline{\Omega_{x,f}}$$ defined by the $\mathfrak{o}$-rational action of $\mathcal{G}_{x,f}\times_{\mathfrak{o}}\mathcal{G}_{x,f}$ on $\Omega_{x,f}$
    satisfies
        \item[(1)] $\mathcal{U}_{\Phi^+,x,f^+}(\mathfrak{o})\times \overline{\mathcal{T}}(\mathfrak{o})\times \mathcal{U}_{\Phi^-,x,f^+}(\mathfrak{o})\subset \Dom(A)(\mathfrak{o})$. In particular, since $\kappa$ is assumed to be perfect (hence is also algebraically closed), by the lifting property \cite[\S~2.3, Proposition~5]{BLR}, we have $$(\mathcal{U}_{\Phi^+,x,f^+})_{\kappa}\times_{\kappa}(\overline{\mathcal{T}})_{\kappa}\times_{\kappa}(\mathcal{U}_{\Phi^-,x,f^+})_{\kappa}\subset \Dom(A);$$
        \item[(2)]           
        $A(\mathcal{U}_{\Phi^+,x,f^+}(\mathfrak{o})\times \overline{\mathcal{T}}(\mathfrak{o})\times \mathcal{U}_{\Phi^-,x,f^+}(\mathfrak{o}))\subset \mathcal{U}_{\Phi^-,x,f^+}(\mathfrak{o})\times \overline{\mathcal{T}}(\mathfrak{o})\times \mathcal{U}_{\Phi^+,x,f^+}(\mathfrak{o});$
        \item[(3)] $e\times_{\mathfrak{o}} \overline{\mathcal{T}} \times_{\mathfrak{o}} e\subset \Dom(A)$.
\end{theorem}

\begin{proof}
    For a section 
    $\tau\in (\mathcal{U}_{\Phi^+,x,f^+}\times_{\mathfrak{o}}\overline{\mathcal{T}}\times_{\mathfrak{o}}\mathcal{U}_{\Phi^-,x,f^+})(\mathfrak{o}),$ we will construct an $\mathfrak{o}$-rational morphism $$A_{\tau}:\mathcal{G}_{x,f}\times_{\mathfrak{o}} \overline{\Omega_{x,f}}\times_{\mathfrak{o}}\mathcal{G}_{x,f}\dashrightarrow \overline{\Omega_{x,f}}$$ 
    such that $A=A_{\tau}$ holds as an equality of $\mathfrak{o}$-rational morphisms and $\tau\in \Dom(A_{\tau})(\mathfrak{o})$.
    
     We can find $g\in k[U^-\times_k\overline{T}\times_k U^+]$ and $g'\in k[U^+\times_k\overline{T}\times_k U^-]$ such that 
    $$\tau \in D_{U^+\times_k\overline{T} \times_k U^-}(g')(k),\;\Multi(\tau)\in D_{U^-\times_k\overline{T}\times_k U^+}(g)(k),$$
    $$\Multi(D_{U^+\times_k\overline{T} \times_k U^-}(g'))\subset D_{U^-\times_k\overline{T}\times_k U^+}(g)\subset\overline{G}.$$
    Without loss of generality, we can assume that $$g\in \mathfrak{o}[\mathcal{U}_{\Phi^-,x,f}\times_\mathfrak{o}\overline{\mathcal{T}}\times_\mathfrak{o}\mathcal{U}_{\Phi^+,x,f}]\backslash \pi\mathfrak{o}[\mathcal{U}_{\Phi^-,x,f}\times_\mathfrak{o}\overline{\mathcal{T}}\times_\mathfrak{o}\mathcal{U}_{\Phi^+,x,f}],$$
    $$g'\in \mathfrak{o}[\mathcal{U}_{\Phi^+,x,f}\times_\mathfrak{o}\overline{\mathcal{T}}\times_\mathfrak{o} \mathcal{U}_{\Phi^-,x,f}]\backslash \pi\mathfrak{o}[\mathcal{U}_{\Phi^+,x,f}\times_\mathfrak{o}\overline{\mathcal{T}}\times_\mathfrak{o}\mathcal{U}_{\Phi^-,x,f}].$$
    By the extension principle \Cref{extensionprinciple}, we have that 
    \begin{equation}\label{equationidentity}
        \begin{aligned}
            \Multi(\tau)\in \Spec(\mathfrak{o}[\mathcal{U}_{\Phi^-,x,f}\times_\mathfrak{o}\overline{\mathcal{T}}\times_\mathfrak{o}\mathcal{U}_{\Phi^+,x,f}]_g)(\mathfrak{o}),\;\;
            \tau\in \Spec(\mathfrak{o}[\mathcal{U}_{\Phi^+,x,f}\times_\mathfrak{o}\overline{\mathcal{T}}\times_\mathfrak{o} \mathcal{U}_{\Phi^-,x,f}]_{g'})(\mathfrak{o}).
        \end{aligned}
    \end{equation}
    Combining \Cref{distributioninvariant} with \Cref{corollaryextension}, we obtain an $\mathfrak{o}$-morphism
    $$E:D_{\mathcal{U}_{\Phi^+,x,f}\times_\mathfrak{o}\overline{\mathcal{T}}\times_\mathfrak{o}\mathcal{U}_{\Phi^-,x,f}}(g')\longrightarrow D_{\mathcal{U}_{\Phi^-,x,f}\times_\mathfrak{o}\overline{\mathcal{T}}\times_\mathfrak{o} \mathcal{U}_{\Phi^+,x,f}}(g)$$
    whose generic fiber is given by the group action on wonderful compactification.
    
    Now, since $\Omega_{x,f}\subset \mathcal{G}_{x,f}$ is fiberwise dense, it suffices to define $A_{\tau}$ as the composition of the $\mathfrak{o}$-rational morphisms:
 
%\begin{displaymath}
    \xymatrix{
\Omega_{x,f}\times_\mathfrak{o}\overline{\Omega_{x,f}}\times_\mathfrak{o}\Omega_{x,f}= (\mathcal{U}_{\Phi^-,x,f}\times_\mathfrak{o}\mathcal{T}\times_\mathfrak{o} \mathcal{U}_{\Phi^+,x,f})\times_\mathfrak{o}(\mathcal{U}_{\Phi^-,x,f}\times_\mathfrak{o}\overline{\mathcal{T}}\times_\mathfrak{o}\mathcal{U}_{\Phi^+,x,f})\times_\mathfrak{o}(\mathcal{U}_{\Phi^-,x,f}\times_\mathfrak{o}\mathcal{T}\times_\mathfrak{o} \mathcal{U}_{\Phi^+,x,f})
    }
    \xymatrix@=-0.5em@R=4ex{
 = &(\mathcal{U}_{\Phi^-,x,f} \times_\mathfrak{o}\mathcal{T})\ar[d]^{\Id}  & \times_\mathfrak{o} & (\mathcal{U}_{\Phi^+,x,f}\times_\mathfrak{o} \mathcal{U}_{\Phi^-,x,f}) \ar@{-->}[d]^{\mathcal{M}} &\times_\mathfrak{o} &\overline{\mathcal{T}} \ar[d]^{\Id}&\times_\mathfrak{o} & (\mathcal{U}_{\Phi^+,x,f}\times_\mathfrak{o}\mathcal{U}_{\Phi^-,x,f})\ar@{-->}[d]^{\mathcal{M}}&\times_\mathfrak{o} &(\mathcal{T}\times_\mathfrak{o} \mathcal{U}_{\Phi^+,x,f})\ar[d]^{\Id} \\
      &  (\mathcal{U}_{\Phi^-,x,f}\times_\mathfrak{o}\mathcal{T}) & \times_\mathfrak{o} & (\mathcal{U}_{\Phi^-,x,f}\times_\mathfrak{o} \mathcal{T}\times_\mathfrak{o}\mathcal{U}_{\Phi^+,x,f})&\times_\mathfrak{o} &\overline{\mathcal{T}} &\times_\mathfrak{o}& (\mathcal{U}_{\Phi^-,x,f}\times_\mathfrak{o} \mathcal{T}\times_\mathfrak{o}\mathcal{U}_{\Phi^+,x,f})&\times_\mathfrak{o} &(\mathcal{T}\times_\mathfrak{o} \mathcal{U}_{\Phi^+,x,f})
      }

    \xymatrix@=-0.5em@R=4ex{
=&  (\mathcal{U}_{\Phi^-,x,f}\times_\mathfrak{o}\mathcal{T}  \times_\mathfrak{o}  \mathcal{U}_{\Phi^-,x,f}\times_\mathfrak{o} \mathcal{T})\ar[d]^{\Id}&\times_\mathfrak{o}&(\mathcal{U}_{\Phi^+,x,f}\times_\mathfrak{o} \overline{\mathcal{T}} \times_\mathfrak{o}\mathcal{U}_{\Phi^-,x,f})\ar@{-->}[d]^{E}&\times_\mathfrak{o} &(\mathcal{T}\times_\mathfrak{o}\mathcal{U}_{\Phi^+,x,f}\times_\mathfrak{o} \mathcal{T}\times_\mathfrak{o} \mathcal{U}_{\Phi^+,x,f})\ar[d]^{\Id}\\
& (\mathcal{U}_{\Phi^-,x,f}\times_\mathfrak{o}\mathcal{T}  \times_\mathfrak{o}  \mathcal{U}_{\Phi^-,x,f}\times_\mathfrak{o} \mathcal{T})&\times_\mathfrak{o}&(\mathcal{U}_{\Phi^-,x,f}\times_\mathfrak{o} \overline{\mathcal{T}} \times_\mathfrak{o} \mathcal{U}_{\Phi^+,x,f})&\times_\mathfrak{o} &(\mathcal{T}\times_\mathfrak{o}\mathcal{U}_{\Phi^+,x,f}\times_\mathfrak{o} \mathcal{T}\times_\mathfrak{o} \mathcal{U}_{\Phi^+,x,f})
    }

\xymatrix{
 & ((u_1^-,t_1,u_2^-, t_2), (u_0^-, \overline{t}, u_0^+), (t_3,u_1^+, t_4, u_2^+))\ar@{|->}[d] \\
 &(u_1^-(t_1u_2^-t_1^{-1})(t_1t_2u_0^-t_2^{-1}t_1^{-1}), t_1t_2\overline{t}t_3t_4, (t_4^{-1}t_3^{-1}u_0^+t_3t_4)(t_4^{-1}u_1^+t_4)u_2^+)\\
 &\in \mathcal{U}_{\Phi^-,x,f}\times_\mathfrak{o} \overline{\mathcal{T}} \times_\mathfrak{o} \mathcal{U}_{\Phi^+,x,f}=\overline{\Omega_{x,f}},
}
\noindent where $\mathcal{M}$ is the $\mathfrak{o}$-rational morphism given by the group law of the group scheme $\mathcal{G}_{x,f}$. The facts that $A_{\tau}=A$ and that $\tau\in \Dom(A_{\tau})(\mathfrak{o})$ follow from the construction of $E$. 

By varying $\tau$, the claim (1) follows. The claim (2) follows from \Cref{distributioninvariant} (1).
For the claim (3), by flat base change, we have $e\times_k \overline{T}\times_k e\subset \Dom(A_k)=\Dom(A)_k$. For the inclusion of the special fibers, since the residue field $\kappa$ is algebraically closed, we conclude that $(e\times_{\mathfrak{o}} \overline{\mathcal{T}} \times_{\mathfrak{o}} e)_{\kappa}\subset \Dom(A)_{\kappa}$ by the claim (1) and the lifting property \cite[\S~2.3, Proposition~5]{BLR}.
%The claim (2) follows from (1). More precisely, since $e_k\times \overline{T} \times e_k\subset \Dom(A)$, it suffices to show that, for each closed point $p\in (\overline{\mathcal{T}})_{\kappa}$, we have $e_{\kappa}\times p\times e_{\kappa} \Dom(A)$. By, for instance, \cite[Chapter~III, \S~6, Lemma~3]{localfieldsserre}, $\mathfrak{o}_{\widetilde{a}}$ is a finite free $\mathfrak{o}$-module for each $a\in \Delta$. Then, by the definition of $\overline{\mathcal{T}}$ \Cref{definitionofnubar}, we can write $\overline{\mathcal{T}}= \prod_I\mathbb{A}_{1,\mathfrak{o}}$ with $I$ a finite index set and also for the special fiber $\overline{\mathcal{T}}_{\kappa}=\prod_I\mathbb{A}_{1,\kappa}$. We denote by $p_i\in\mathbb{A}_{1,\kappa}$ the $i$th-component of $p$ which is a closed point. We can assume that $p_i$ is defined by an irreducible monic polynomial $f_i\in \kappa[T]$ (see for instance \cite[Example~5.1]{GortzWedhorn}). We choose a monic lifting $\widetilde{f}\in \mathfrak{o}[T]$ and form $R_f\coloneq\mathfrak{o}[T]/(\widetilde{f})$ which is a local ring with the residue field $\kappa[T]/(f_i)$ \cite[Chapter~I, \S~6, Lemma~4]{localfieldsserre}. The henselization $R_f^h$ is a local $\mathfrak{o}$-algebra whose residue field remains $\kappa[T]/(f_i)$ \cite[théorème~18.6.6~(iii)]{EGAIV4}. By \cite[\S~2.3, Proposition~5]{BLR},the closed point $p_i$ lifts to an $R_f^h$-point of the $i$th-component $\mathbb{A}_{1,\mathfrak{o}}$ for each $i\in I$. Now we conclude (2) by (1).
\end{proof}

\subsection{Integral model of $\overline{G}$}
Now we are in the position to construct the desired integral model of the wonderful compactification $\overline{G}$ in \Cref{introtheorem1}. The construction is divided into two cases depending on if $f(0)=0$ or $f(0)\textgreater 0$. The two cases can be also related to each other via dilatation, which will be discussed in \Cref{sectioncompatiblity}.

\subsubsection{The case when $f(0)=0$}\label{f0equalto0}
Now we apply \Cref{equivalencerelation} to the rational morphism $A$ in \Cref{theoremrationalaction} to get the equivalence relation $\sim_A$ on $\mathcal{G}_{x,f}\times_\mathfrak{o}\overline{\Omega_{x,f}}\times_\mathfrak{o}\mathcal{G}_{x,f}$. Then we further apply \Cref{definitionwonderfulembedding} to get a quotient sheaf over the category $\Sch/\mathfrak{o}$, which will be denoted by $\overline{\mathcal{G}_{x,f}}$. By \Cref{algebraicspace}, $\overline{\mathcal{G}_{x,f}}$ is an algebraic space over $\mathfrak{o}$. By \Cref{groupaction}, $\overline{\mathcal{G}_{x,f}}$ is endowed with a well-defined action of $\mathcal{G}_{x,f}\times_\mathfrak{o}\mathcal{G}_{x,f}$.

\begin{proposition}\label{quasiprojectivityandsmoothness}
    \begin{itemize}
        \item[(1)] The algebraic space $\overline{\mathcal{G}_{x,f}}$ is covered by the translates of $\overline{\Omega_{x,f}}$ by the sections in $(\mathcal{G}_{x,f}\times_{\mathfrak{o}}\mathcal{G}_{x,f})(\mathfrak{o})$. In particular, by the smoothness of $\overline{\Omega_{x,f}}$ over $\mathfrak{o}$, $\overline{\mathcal{G}_{x,f}}$ is smooth over $\mathfrak{o}$.
        \item[(2)] 
        The algebraic space $\overline{\mathcal{G}_{x,f}}$ is quasi-projective over $\mathfrak{o}$. In particular, $\overline{\mathcal{G}_{x,f}}$ is represented by a scheme.
    \end{itemize}
\end{proposition}

\begin{proof}

    To show (1), let $S'$ be a test $S$-scheme, and let $\overline{x}\in \overline{\mathcal{G}_{x,f}}(S')$ represented by $(g_1,\omega,g_2)\in (\mathcal{G}_{x,f}\times_{\mathfrak{o}} \overline{\Omega_{x,f}}\times_{\mathfrak{o}}\mathcal{G}_{x,f})(S'')$ where $S''$ is an fppf-cover of $S'$. It suffices to find a Zariski open covering $S'=\bigcup_{\nu\in N} S'_{\nu}$ such that, for each open $S'_{\nu}$, there exists a section $(c_{\nu}^1, c_{\nu}^2)\in(\mathcal{G}_{x,f}\times_{\mathfrak{o}}\mathcal{G}_{x,f})(\mathfrak{o})$ such that $\omega'\coloneq \phi(c_{\nu}^1, c_{\nu}^2,g_1\vert_{S'_{\nu}}, g_2\vert_{S'_{\nu}},\omega\vert_{S'_{\nu}})$ is well-defined, where $\phi$ is obtained by letting $G=\mathcal{G}_{x,f}\times_{\mathfrak{o}} \mathcal{G}_{x,f}$ and $Y=\overline{\Omega_{x,f}}$ in \Cref{definitionofphi}. If so, by \Cref{reinterpretation}, we have that $(c_{\nu}^1,\omega',c_{\nu}^2)\sim (g_1\vert_{S'_{\nu}},\omega\vert_{S'_{\nu}},g_2\vert_{S'_{\nu}})$ which means that $\overline{x}\vert_{S'_{\nu}}$ lies in $(c_{\nu}^1, c_{\nu}^2)\cdot \overline{\Omega_{x,f}}$, as desired. 

    Without loss of generality, we can assume that $S''=S'$.
    Since now the problem is local on $S'$, by working Zariski locally on $S'$ and a limit argument, we can assume that $S'$ is a local scheme. By \Cref{theoremrationalaction}.1 and the definition of $\phi$ (\Cref{definitionofphi}), the condition on the desired section $c\coloneq(c^1,c^2)$ amounts to requiring $c(S')\subset U$ where $U\subset (\mathcal{G}_{x,f}\times_{\mathfrak{o}}\mathcal{G}_{x,f})\times_{\mathfrak{o}} S'$ is an open $S'$-dense subscheme. Since $S'$ is now a local scheme, in this case, it suffices to require that the closed point of $S'$ be sent into $U$ by $c$. Then we are reduced to the case where $S'=\Spec(k)$ with $k$ a field. Thus, since $\mathcal{G}_{x,f}$ is smooth over $\mathfrak{o}$, the existence of such a section $c$ is ensured by Artin's \cite[\S~5.3, Lemma~7]{BLR} which asserts that, for any point $t\in \Spec(\mathfrak{o})$, the set $\{a(t)\vert a\in \mathcal{G}_{x,f}(\mathfrak{o})\}$ is dense in a connected component of the fiber $(\mathcal{G}_{x,f})_t$.

    For (2), by \Cref{openimmersion}, $\overline{\Omega_{x,f}}$ is an open subscheme of $\overline{\mathcal{G}_{x,f}}$, and by the definition of $\overline{\mathcal{G}_{x,f}}$, we have that $\overline{\mathcal{G}_{x,f}}=(\mathcal{G}_{x,f}\times_{\mathfrak{o}}\mathcal{G}_{x,f})\cdot\overline{\Omega_{x,f}}$. The \emph{quasi-projectivity} of $\overline{\mathcal{G}_{x,f}}$ over $\mathfrak{o}$ follows from \cite[\S~6.6, Theorem~2~(d)]{BLR}.
\end{proof}

\begin{proposition}\label{groupopenimmersion}
    The morphism $i: \mathcal{G}_{x,f}\longrightarrow \overline{\mathcal{G}_{x,f}}$ by sending $g\in \mathcal{G}_{x,f}(S')$ to the section represented by $(g,e,e)$ in $\overline{\mathcal{G}_{x,f}}(S')$ is an open immersion. Moreover, $\mathcal{G}_{x,f}$ is dense in $\overline{\mathcal{G}_{x,f}}$.

    The generic fiber $(\overline{\mathcal{G}_{x,f}})_k$ is isomorphic to the wonderful compactification $\overline{G}$.
\end{proposition}

\begin{proof}
    To show that $i$ is an open immersion, by \cite[025G]{stacks-project}, it suffices to verify that $i$ is a flat monomorphism, locally of finite presentation. 
    \begin{claim}\label{claim3}
        The group scheme $\mathcal{G}_{x,f}$ is covered by the translates of $\Omega_{x,f}$ by the sections in $\mathcal{G}_{x,f}(\mathfrak{o})$.
    \end{claim}
    \noindent We do not find an appropriate reference for this claim, for completeness, we include a proof.
    \begin{proof}[Proof of \Cref{claim3}]
        For any point $g\in\mathcal{G}_{x,f}$, we can assume that $g\in\mathcal{G}_{x,f}(k)$ where $k$ is a field over $\mathfrak{o}$. Since $\Omega_{x,f}$ is $\mathfrak{o}$-dense in $\mathcal{G}_{x,f}$, the open subscheme $(\Omega_{x,f})_k\cdot g^{-1}$ is dense in $(\mathcal{G}_{x,f})_k$. Then, again by appealing to Artin's \cite[\S~5.3, Lemma~7]{BLR} which asserts that for any point $t\in \Spec(\mathfrak{o})$ the set $\{a(t)\vert a\in \mathcal{G}_{x,f}(\mathfrak{o})\}$ is dense in a connected component of the fiber $(\mathcal{G}_{x,f})_t$, we can find a section $c\in \mathcal{G}_{x,f}(\mathfrak{o})$ whose base change $c_k$ has image in $(\Omega_{x,f})_k\cdot g^{-1}$. This means that $g\in c^{-1}\cdot \Omega_{x,f}$. 
    \end{proof}
    \noindent  Recall that $\Omega_{x,f}$ is an open subscheme of the open subscheme $\overline{\Omega_{x,f}}\subset \overline{\mathcal{G}_{x,f}}$ via \Cref{eqembeddingofbigcell}. Hence, by \Cref{claim3}, $i$ is flat. Since $\mathcal{G}_{x,f}$ and $\overline{\mathcal{G}_{x,f}}$ are both locally of finite presentation over $\Spec(\mathfrak{o})$, by \cite[06Q6]{stacks-project}, so is $i$.
    
    To show that $i$ is a monomorphism, for a test $\mathfrak{o}$-scheme $S'$, we consider two sections $g_1,g_2\in \mathcal{G}_{x,f}(S')$ such that $(g_1,e,e)\sim (g_2,e,e)$. Then we apply \cite[exposé~XVIII, proposition~1.7]{SGA3II} to the $S'$-dense open subscheme $C\coloneq(\Omega_{x,f})_{S'}\cdot g_1^{-1}\bigcap (\Omega_{x,f})_{S'}\cdot g_2^{-1}\subset (\mathcal{G}_{x,f})_{S'}$ to get an fppf-cover $S''\rightarrow S'$ and a section $c\in C(S'')$. Then $cg_1, cg_2\in \Omega_{x,f}(S'')$, and, by the definition of the rational morphism $A$ (see \Cref{theoremrationalaction}) and \Cref{testinglemma}, we have that $cg_1=cg_2$, hence also $g_1=g_2$.

    The density of $\mathcal{G}_{x,f}$ in $\overline{\mathcal{G}_{x,f}}$ follows from the density of $\Omega_{x,f}$ in $\overline{\Omega_{x,f}}$ and \Cref{quasiprojectivityandsmoothness}(1).

    To investigate the generic fiber of $\overline{\mathcal{G}_{x,f}}$, by the functoriality \Cref{remarkfuncotriality}, $(\overline{\mathcal{G}_{x,f}})_k$ is isomorphic to the quotient of $G\times_k \overline{\Omega}\times_k G$ with respect to the equivalence relation given by applying \Cref{equivalencerelation} to the rational morphism $A_k: G\times_k\overline{\Omega}\times_k G\dashrightarrow \overline{\Omega}$. Then the last claim follows from the facts that $A_k$ coincides with the rational morphism defined by the group law of $G$ and that $\overline{G}=(G\times_k G)\cdot\overline{\Omega}$ (which follows from \Cref{threlativewonderfulcompsplitcase} (iii) and \Cref{propquasisplitwonderfulcomp}).
\end{proof}

\subsubsection{The case when $f(0)\textgreater 0$}\label{f0biggerthat0}
If $f(0)\textgreater 0$, to build the desired $\overline{\mathcal{G}_{x,f}}$ in \Cref{introtheorem1}, it suffices to define $\overline{\mathcal{G}_{x,f}}$ to be the gluing of $\overline{G}$ and $\mathcal{G}_{x,f}$. The fact that the open immersion $G\hookrightarrow \overline{G}$ extends to an open immersion $\mathcal{G}_{x,f}\hookrightarrow\overline{\mathcal{G}_{x,f}}$ follows from the open immersion in \Cref{eqembeddingofbigcell} and the structure of special fiber $(\mathcal{G}_{x,f})_{\kappa}$ \Cref{unipotentradicalspecialfiber} (2).
The condition (ii) of \Cref{introtheorem1} follows from the computation \Cref{dilatationofaffineline}.

\begin{remark}
    Note that, since $\mathfrak{o}$ is a local ring, it is a direct consequence of \Cref{introtheorem1} (ii) and (2) that we have the bijections of $\mathfrak{o}$-points (when $f(0)\textgreater 0$)
    $$\overline{\mathcal{G}_{x,f}}(\mathfrak{o})\longleftrightarrow \Big(\prod_{a\in\Phi^{-,\red}}\mathcal{U}_{a,x,f}(\mathfrak{o})\Big)\times \overline{\mathcal{T}^{(f(0))}}(\mathfrak{o})\times \Big(\prod_{a\in\Phi^{+,\red}}\mathcal{U}_{a,x,f}(\mathfrak{o})\Big)\longleftrightarrow \mathcal{G}_{x,f}(\mathfrak{o}),$$
    where the second bijection follows from the definition of $\overline{\mathcal{T}^{(f(0))}}$.
\end{remark}

\subsection{Special fiber of the wonderful embedding $\overline{\mathcal{G}_{x,f}}$}\label{subsectionspecialfiber}

In this section, we study the special fiber of $\overline{\mathcal{G}_{x,f}}$. Since, by construction, \Cref{introtheorem1} (2) is clear when $f(0)\textgreater 0$, we will assume that $f(0)=0$ in the rest of \Cref{subsectionspecialfiber}.

By the functoriality of \Cref{definitionwonderfulembedding}, the special fiber $(\overline{\mathcal{G}_{x,f}})_{\kappa}$ is isomorphic to the quotient sheaf $(\mathcal{G}_{x,f})_{\kappa}\times_{\kappa}(\overline{\Omega_{x,f}})_{\kappa}\times_{\kappa}(\mathcal{G}_{x,f})_{\kappa}/\sim_{\kappa}$ over $\Sch/\kappa$, where the equivalence relation $\sim_{\kappa}$ is defined by the base change $A_{\kappa}:(\mathcal{G}_{x,f})_{\kappa}\times_{\kappa}(\overline{\Omega_{x,f}})_{\kappa}\times_{\kappa}(\mathcal{G}_{x,f})_{\kappa}\dashrightarrow (\overline{\Omega_{x}})_{\kappa}$ of the $\mathfrak{o}$-rational morphism $A$ in \Cref{theoremrationalaction} in the same way as \Cref{equivalencerelation}. 

 We keep the notations of \Cref{groupsetup}. For $a\in\Phi$, let $e_{\widetilde{a}}$ be the ramification index of the field extension $k_{\widetilde{a}}/k$, and let $\pi_{\widetilde{a}}\in \mathfrak{o}_{\widetilde{a}}$ be a uniformizer. Since $k$ is strictly Henselian, by for instance \cite[Chapter~III, \S~5, Corollary~3]{localfieldsserre}, the field extension $k_{\widetilde{a}}/k$ is totally ramified. Hence, by \cite[Chapter~III, \S~6, Lemma~3]{localfieldsserre}, we can choose $\{1, \pi_{\widetilde{a}},\pi_{\widetilde{a}}^2,\cdots,\pi_{\widetilde{a}}^{e_{\widetilde{a}}-1} \}$ as an $\mathfrak{o}$-basis of $\mathfrak{o}_{\widetilde{a}}$, and we denote by $\{\overline{1}, \overline{\pi_{\widetilde{a}}},\overline{\pi_{\widetilde{a}}^2},\cdots,\overline{\pi_{\widetilde{a}}^{e_{\widetilde{a}}-1}} \}$ the image of this basis in the quotient module $\mathfrak{o}_{\widetilde{a}}/\pi$. By a computation following the definition of Weil restriction, we have an isomorphism of $\kappa$-group schemes:
\begin{align}\label{iso2}
    \mathbb{G}_{m,\kappa} \times_{\kappa} \prod_{e_{\widetilde{a}}-1}\mathbb{A}_{1,\kappa}
      \;\;&\xlongrightarrow{\cong} 
    (\Res_{\mathfrak{o}_{\widetilde{a}}/\mathfrak{o}}(\mathbb{G}_{m,\mathfrak{o}_{\widetilde{a}}}))_{\kappa}=\Res_{(\mathfrak{o}_{\widetilde{a}}/\pi)/ \kappa}(\mathbb{G}_{m,\mathfrak{o}_{\widetilde{a}}/\pi})\\
    (x_0, x_1,\cdots,x_{(e_{\widetilde{a}}-1)})&\longmapsto x_0\otimes \overline{1}+ x_0x_1\otimes \overline{\pi_{\widetilde{a}}}+\cdots + x_0x_{(e_{\widetilde{a}}-1)}\otimes \overline{\pi_{\widetilde{a}}^{e_{\widetilde{a}}-1}} \in R\otimes_{\kappa}(\mathfrak{o}_{\widetilde{a}}/\pi), \nonumber
\end{align}
where $R$ is a test $\kappa$-algebra and the group law of $\prod_{e_{\widetilde{a}}-1}\mathbb{A}_{1,\kappa}$ is given by the following embedding:
\begin{align*}
    \prod_{e_{\widetilde{a}}-1}\mathbb{A}_{1,\kappa}&\;\;\;\;\longhookrightarrow \;\;\;\;\;\;\;\;\;\;\;\;\;\;\;\;\GL_{e_{\widetilde{a}},\kappa}\\
    (x_1,...,x_{(e_{\widetilde{a}}-1)})\;\;\;&\longmapsto  
    \begin{pmatrix}
        1 & x_1 &x_2  &x_3    & \cdots &  x_{(e_{\widetilde{a}}-2)}  & x_{(e_{\widetilde{a}}-1)}\\
        0 &1    &x_1  &x_2    & \cdots &  x_{(e_{\widetilde{a}}-3)}  &  x_{(e_{\widetilde{a}}-2)}\\
        0 &0    &1    &x_1    & \cdots &  x_{(e_{\widetilde{a}}-4)}  &  x_{(e_{\widetilde{a}}-3)}\\
        \vdots  &\vdots&\vdots &\vdots &\ddots  &\vdots    &\vdots \\
        0 &0 &0 &0  &\cdots    &x_1       &x_2\\
        0 &0 &0 &0  &\cdots &1    &x_1 \\
        0 &0 &0 &0  &\cdots         &0  & 1
    \end{pmatrix}.
\end{align*}

\noindent Hence we have that $ \mathscr{R}_u((\Res_{\mathfrak{o}_{\widetilde{a}}/\mathfrak{o}}(\mathbb{G}_{m,\mathfrak{o}_{\widetilde{a}}}))_{\kappa})= \prod_{e_{\widetilde{a}}-1}\mathbb{A}_{1,\kappa}$ under the identification of \Cref{iso2}. We let $\overline{\mathscr{T}^+}\coloneq \mathscr{R}_u(\mathcal{T}_{\kappa})$ which is isomorphic to $\prod_{a\in \Delta}(\prod_{e_{\widetilde{a}}-1}\mathbb{A}_{1,\kappa})$ via \Cref{equationisomorphismoftorusintegralmodel} and \Cref{iso2}, and let $\overline{\mathscr{S}}\coloneq \prod_{a\in\Delta}\mathbb{A}_{1,\kappa}\subset (\overline{\mathcal{T}})_{\kappa}$, where each $\mathbb{A}_{1,\kappa}\subset \Res_{(\mathfrak{o}_{\widetilde{a}}/\pi)/ \kappa}(\mathbb{A}_{1,\mathfrak{o}_{\widetilde{a}}/\pi})$ is the first component with respect to the ordered basis $\{\overline{1}, \overline{\pi_{\widetilde{a}}},\overline{\pi_{\widetilde{a}}^2},\cdots,\overline{\pi_{\widetilde{a}}^{e_{\widetilde{a}}-1}} \}$. Then we have
$$(\overline{\mathcal{T}})_{\kappa}=\overline{\mathscr{S}}\times_{\kappa}\overline{\mathscr{T}^+}.$$

Let $\mathscr{S}$ be the schematic closure of $S$ in $\mathcal{G}_{x,f}$, and let $\mathrm{S}$ be the isomorphic image of the special fiber $\mathscr{S}_{\kappa}$ in the maximal reductive quotient $\mathsf{G}_{x,f}$. By \cite[Theorem~8.5.14~(1)]{Bruhattitsnewapproach}, $\mathrm{S}\subset \mathsf{G}_{x,f}$ is a maximal torus. We will adopt the natural identifications of the character lattices and of cocharacter lattices:
\begin{align}
    X^*(S)\cong X^*(\mathscr{S})\cong X^*(\mathrm{S});\\
    X_*(S)\cong X_*(\mathscr{S})\cong X_*(\mathrm{S}).
\end{align}

\noindent Let $\Phi_{x,f}\coloneq \{a\in \Phi\vert f(a)+f(-a)=0\;\text{and}\; f(a)\in \Gamma_a'\}$ which is identified with the root system of $\mathsf{G}_{x,f}$ with respect to the maximal torus $\mathrm{S}$ under the above identifications, cf. \cite[Theorem~8.5.14~(3)]{Bruhattitsnewapproach}, where $\Gamma_a'$ is recalled in \Cref{subsectionsetofvalues}. Let $\Delta_{x,f}\coloneq \Delta\bigcap \Phi_{x,f}\subset \Phi_{x,f}$ which is a set of simple roots. Moreover, the negative Weyl chamber $\mathfrak{C}\subset X_*(S)_{\mathbb{R}}$ defined by the simple roots $\Delta$ is a cone inside the negative Weyl chamber $\mathfrak{C}_{x,f}\subset X_*(\mathrm{S})_{\mathbb{R}}$ defined by $\Delta_{x,f}$.

\subsubsection{The induced rational action}
We will continue to use the convention of \Cref{integralunipotentconvention}.
We denote by $\mathsf{G}_{x,f}$ the maximal reductive quotient of the special fiber $(\mathcal{G}_{x,f})_{\kappa}$. To study the quotient sheaf $(\overline{\mathcal{G}_{x,f}})_{\kappa}$, we introduce more notations. 
Let $$R_u^+\subset (\mathcal{U}_{\Phi^+,x,f})_{\kappa} \;\;\text{and}\;\;R_u^-\subset (\mathcal{U}_{\Phi^-,x,f})_{\kappa}$$
be the images of the special fibers of the natural morphisms $\mathcal{U}_{\Phi^+,x,f^+}\rightarrow \mathcal{U}_{\Phi^+,x,f}$ and $\mathcal{U}_{\Phi^-,x,f^+}\rightarrow \mathcal{U}_{\Phi^-,x,f}$ respectively. Let 
$$\overline{\Omega_{\red}}\coloneq (\mathcal{U}_{\Phi^-,x,f})_{\kappa}/R_u^-\times_{\kappa}(\overline{\mathcal{T}})_{\kappa}/\overline{\mathscr{T}^+}\times_{\kappa} (\mathcal{U}_{\Phi^+,x,f})_{\kappa}/R_u^+.$$
For notational simplicity, in the following lemma and its proof, we will fix a test $\kappa$-scheme and all points exist with respect to this test scheme.

\begin{lemma}\label{lemmaquotientrational}
    Let $g_1, g_2\in (\mathcal{G}_{x,f})_{\kappa}$, $u\in (\mathcal{U}_{\Phi^-,x,f})_{\kappa}, v\in (\mathcal{U}_{\Phi^+,x,f})_{\kappa}$ and $\overline{t}\in(\overline{\mathcal{T}})_{\kappa}$ such that 
    $$g_1\cdot u\cdot \overline{t}\cdot v\cdot g_2\in (\overline{\Omega_{x,f}})_{\kappa}.$$ 
    Then, for $r_1, r_2\in \mathscr{R}_u((\mathcal{G}_{x,f})_{\kappa})$ and $r^+\in R_u^+,r^-\in R_u^-, t^+\in \overline{\mathscr{T}^+}$, we have 
    $$(g_1\cdot r_1)\cdot (u\cdot r^-)\cdot(\overline{t}\cdot t^+)\cdot (v\cdot r^+)\cdot (g_2\cdot r_2)\in (\overline{\Omega_{x,f}})_{\kappa}.$$
    Moreover, $E\coloneq g_1\cdot u\cdot \overline{t}\cdot v\cdot g_2$ and $\widetilde{E}\coloneq (g_1\cdot r_1)\cdot (u\cdot r^-)\cdot(\overline{t}\cdot t^+)\cdot (v\cdot r^+)\cdot (g_2\cdot r_2)$ have the same image in $\overline{\Omega_{\red}}$.
\end{lemma}

\begin{proof}
   The proof is simply a computation. More specific,
   first by our hypothesis, we can write 
   
   $$E=u'\cdot \overline{t}'\cdot t'^{+}\cdot v',$$ 
   where $u'\in (\mathcal{U}_{\Phi^-,x,f})_{\kappa}, v'\in (\mathcal{U}_{\Phi^+,x,f})_{\kappa}, \overline{t}'\in\overline{\mathscr{S}}$ and $t'^{+}\in\overline{\mathscr{T}^+}.$
   Since $\mathscr{R}_u((\mathcal{G}_{x,f})_{\kappa})$ is by definition normalized by $(\mathcal{G}_{x,f})_{\kappa}$ and we have the following equation
    \begin{align*}
       \widetilde{E}\coloneq &(g_1\cdot r_1)\cdot (u\cdot r^-)\cdot(\overline{t}\cdot t^+)\cdot (v\cdot r^+)\cdot (g_2\cdot r_2) \notag\\
        =& g_1\cdot (r_1\cdot (u\cdot r^-\cdot u^{-1}))\cdot g_1^{-1}\cdot( g_1\cdot u\cdot \overline{t}\cdot v\cdot g_2)\cdot (g_2^{-1}\cdot v^{-1}\cdot t^+\cdot v\cdot g_2)\cdot g_2^{-1}\cdot r^+\cdot g_2\cdot r_2,
    \end{align*}
   we can change notations by writing 
   \begin{align}
       \widetilde{E}=&\widetilde{r_1}\cdot (u'\cdot \overline{t}'\cdot t'^{+}\cdot v')\cdot \widetilde{r_2}\notag\\
       =&u'\cdot( u'^{-1}\cdot \widetilde{r_1}\cdot u')\cdot \overline{t}'\cdot t'^{+}\cdot (v'\cdot \widetilde{r_2}\cdot v'^{-1} )\cdot v'\label{equation1}
   \end{align}
   where $\widetilde{r_1},\widetilde{r_2}\in \mathscr{R}_u((\mathcal{G}_{x,f})_{\kappa})$.
   %$(r_1^- \cdot t_1^+\cdot r_1^+), (r_2^- \cdot t_2^+\cdot r_2^+)\in R_u^-\times\overline{\mathscr{T}^+}\times R^+_u$
   We can write
   $$u'^{-1}\cdot \widetilde{r_1}\cdot u'=  u_1^-\cdot \dot{t}^+_1 \cdot u_1^+ \in   R_u^-\times_{\kappa}\overline{\mathscr{T}^+}\times_{\kappa} R^+_u\;\;\text{and}\;\; v'\cdot \widetilde{r_2}\cdot v'^{-1}= v_1^-\cdot \dot{t}^+_2 \cdot v_1^+\in R_u^-\times_{\kappa}\overline{\mathscr{T}^+}\times_{\kappa} R^+_u.$$
   We substitute them into \Cref{equation1} and continue the computation of $\widetilde{E}$
   \begin{align*}
       \widetilde{E}=&u'\cdot(   u_1^-\cdot \dot{t}^+_1 \cdot u_1^+ )\cdot \overline{t}'\cdot t'^{+}\cdot (v_1^-\cdot \dot{t}^+_2 \cdot v_1^+)\cdot v'\\
       =&u'\cdot   u_1^-\cdot (\dot{t}^+_1 \cdot u_1^+ \cdot (\dot{t}^{+}_1)^{-1})\cdot \overline{t}'\cdot \dot{t}^+_1t'^{+}\dot{t}^+_2\cdot ((\dot{t}^+_2)^{-1}v_1^-\cdot \dot{t}^+_2 )\cdot v_1^+\cdot v'
   \end{align*}
   By \Cref{theoremrationalaction}, we can "switch" the positive part $\dot{t}^+_1 \cdot u_1^+ \cdot (\dot{t}^{+}_1)^{-1}$ and the negative part $(\dot{t}^+_2)^{-1}v_1^-\cdot \dot{t}^+_2 $ across $\overline{t}'\cdot \dot{t}^+_1t'^{+}\dot{t}^+_2$ without changing the $\overline{\mathscr{S}}$-component $\overline{t}'$ (since $R_u^+$ is normalized by $(\mathscr{S})_\kappa$, by \Cref{unipotentradicalspecialfiber} (1), we see that this is the case when $\overline{t}'\in (\mathscr{S})_{\kappa}$; then the invariance of $\overline{t}'$ follows from the density of $(\mathscr{S})_{\kappa}$ in $(\overline{\mathscr{S}})_{\kappa}$). More precisely,
   there exist $\dot{r}^-\in R^-_u, \dot{r}^+\in R^+_u, \dot{t}^+\in \overline{\mathscr{T}^+}$ such that 
   $$\widetilde{E}=u'\cdot   u_1^-\cdot \dot{r}^-\cdot \overline{t}'\cdot \dot{t}^+\cdot \dot{r}^+\cdot v_1^+\cdot v'.$$
   Finally we finish our proof by noting that $R_u^-$ (resp., $R_u^+$) is normalized by $(\mathcal{U}_{\Phi^-,x,f})_{\kappa}$ (resp., $(\mathcal{U}_{\Phi^+,x,f})_{\kappa}$). 
\end{proof}

Thanks to \Cref{lemmaquotientrational}, the special fiber $A_{\kappa}$ of the $\mathfrak{o}$-rational morphism $A$ naturally induces a rational action $$(A_{\kappa})_{\red}:\mathsf{G}_{x,f}\times_{\kappa}\overline{\Omega_{\red}}\times_{\kappa}\mathsf{G}_{x,f}\dashrightarrow\overline{\Omega_{\red}}.$$ and $\{e\}\times \overline{\Omega_{\red}}\times \{e\}\subset \Dom((A_{\kappa})_{\red})$.

\subsubsection{Quotient by a non-free action}
Note that the quotient sheaf in \Cref{introtheorem1} (1) is formed with respect to a \emph{non-free} action. It is worth to recall that the fppf sheaf quotient of an algebraic space over a base scheme with respect to a \emph{free} action by a flat and locally finitely presented algebraic group space over the same base is an algebraic space, see \cite[06PH]{stacks-project}. Without the freeness condition, it seems more subtle to determine the representability of the quotient sheaf by algebraic space or scheme. The following result provides an interesting example.

\begin{proposition}\label{toroidalspecialfiber}
The quotient sheaf $(\overline{\mathcal{G}_{x,f}})_{\kappa}/(\mathscr{R}_u((\mathcal{G}_{x,f})_{\kappa})\times_{\kappa}\mathscr{R}_u((\mathcal{G}_{x,f})_{\kappa})$ is isomorphic to the quotient sheaf of $\mathsf{G}_{x,f}\times_{\kappa}\overline{\Omega_{\red}}\times_{\kappa}\mathsf{G}_{x,f}$ with respect to the equivalence relation defined by the rational action $(A_{\kappa})_{\red}$.

Furthermore, $(\overline{\mathcal{G}_{x,f}})_{\kappa}/(\mathscr{R}_u((\mathcal{G}_{x,f})_{\kappa})\times_{\kappa}\mathscr{R}_u((\mathcal{G}_{x,f})_{\kappa})$ is isomorphic to the toroidal embedding $X_{\mathfrak{C}}$ of the $\mathsf{G}_{x,f}$ corresponding to the cone $\mathfrak{C}\subset X_*(S)_{\mathbb{R}}$ defined by the simple roots $\Delta$ which lies inside the negative Weyl chamber $\mathfrak{C}_{x,f}\subset X_*(\mathrm{S})_{\mathbb{R}}$ defined by $\Delta_{x,f}$.
        %\item[(2)] If $f(0)\textgreater 0$, then $(\overline{\mathcal{G}_{x,f}})_{\kappa}/(\mathscr{R}_u((\mathcal{G}_{x,f})_{\kappa})\times\mathscr{R}_u((\mathcal{G}_{x,f})_{\kappa})$ is isomorphic to $\overline{\Omega_{\red}}$.
   
\end{proposition}

\begin{proof}
   By the functoriality of the definition of $\overline{\mathcal{G}_{x,f}}$ \Cref{remarkfuncotriality}, the special fiber $(\overline{\mathcal{G}_{x,f}})_{\kappa}$ can be identified with the quotient sheaf of $(\mathcal{G}_{x,f})_{\kappa}\times_{\kappa}(\overline{\Omega_{x,f}})_{\kappa}\times_{\kappa}(\mathcal{G}_{x,f})_{\kappa}$ obtained by applying \Cref{definitionwonderfulembedding} to the $\kappa$-rational morphism $A_{\kappa}$. Then, by the definition of $(A_{\kappa})_{\red}$, there is a natural well-defined morphism of sheaves over $\Sch/\kappa$:
    $$H'\coloneq (\overline{\mathcal{G}_{x,f}})_{\kappa}\longrightarrow \mathsf{G}_{x,f}\times_{\kappa}\overline{\Omega_{\red}}\times_{\kappa}\mathsf{G}_{x,f}/\sim_{(A_{\kappa})_{\red}}.$$
    By \Cref{lemmaquotientrational}, $H'$ naturally factors through $(\overline{\mathcal{G}_{x,f}})_{\kappa}/(\mathscr{R}_u((\mathcal{G}_{x,f})_{\kappa})\times_{\kappa}\mathscr{R}_u((\mathcal{G}_{x,f})_{\kappa})$:
    $$\xymatrix{
    (\overline{\mathcal{G}_{x,f}})_{\kappa}  \ar[rr]^{H'} \ar[dr]^{Q} && \mathsf{G}_{x,f}\times_{\kappa}\overline{\Omega_{\red}}\times_{\kappa}\mathsf{G}_{x,f}/\sim_{(A_{\kappa})_{\red}}\\
    & (\overline{\mathcal{G}_{x,f}})_{\kappa}/(\mathscr{R}_u((\mathcal{G}_{x,f})_{\kappa})\times_{\kappa}\mathscr{R}_u((\mathcal{G}_{x,f})_{\kappa})  \ar[ur]^{H} &
    }$$
    where $Q$ is the quotient morphism.
    The $H$ is clearly an epimorphism. To show that $H$ is monomorphism, we need to verify if for two sections
    $(g_1, \omega, g_2), (g_1', \omega', g_2')\in(\mathcal{G}_{x,f})_{\kappa}\times_{\kappa}(\overline{\Omega_{x,f}})_{\kappa}\times_{\kappa}(\mathcal{G}_{x,f})_{\kappa}$ such that 
    $ g_1\cdot \omega\cdot g_2$ and $ g_1'\cdot \omega'\cdot g_2'$ both lie in $(\overline{\Omega_{x,f}})_{\kappa}$ and have the same image in $\overline{\Omega_{\red}}$, then there exists $(z_1, z_2)\in \mathscr{R}_u((\mathcal{G}_{x,f})_{\kappa})\times_{\kappa}\mathscr{R}_u((\mathcal{G}_{x,f})_{\kappa})$ such that $z_1\cdot g_1\cdot \omega\cdot g_2\cdot z_2 = g_1'\cdot \omega' \cdot g_2' \in (\overline{\mathcal{G}_{x,f}})_{\kappa}$. This follows from the normality of $\mathscr{R}_u((\mathcal{G}_{x,f})_{\kappa})$ in $(\mathcal{G}_{x,f})_{\kappa}$. More precisely, let $g_1\cdot \omega \cdot g_2=v^-\cdot\overline{t}\cdot v^+\in(\overline{\Omega_{x,f}})_{\kappa}$. Then, by our assumption, there exist $u^-\in(\mathcal{U}_{\Phi^-,x,f^+})_{\kappa}, u^+\in (\mathcal{U}_{\Phi^+,x,f^+})_{\kappa}, t^+\in\overline{\mathscr{T}^+}$ such that $g_1'\cdot \omega' \cdot g_2'=v^-\cdot u^-\cdot \overline{t}\cdot t^+\cdot v^+\cdot u^+$. Then we conclude that $(\overline{\mathcal{G}_{x,f}})_{\kappa}/(\mathscr{R}_u((\mathcal{G}_{x,f})_{\kappa})\times_{\kappa}\mathscr{R}_u((\mathcal{G}_{x,f})_{\kappa})\cong \mathsf{G}_{x,f}\times_{\kappa}\overline{\Omega_{\red}}\times_{\kappa}\mathsf{G}_{x,f}/\sim_{(A_{\kappa})_{\red}}$ by observing
    $$g_1'\cdot \omega' \cdot g_2'=(v^-\cdot u^-\cdot (v^-)^{-1})\cdot v^-\cdot \overline{t}\cdot v^+\cdot ((v^+)^{-1}\cdot t^+\cdot v^+)\cdot u^+.$$
    
    By \Cref{openimmersion}, $(\overline{\mathcal{G}_{x,f}})_{\kappa}/(\mathscr{R}_u((\mathcal{G}_{x,f})_{\kappa})\times_{\kappa}\mathscr{R}_u((\mathcal{G}_{x,f})_{\kappa})$ contains $\overline{\Omega_{\red}}$ as an open subscheme and $(\overline{\mathcal{G}_{x,f}})_{\kappa}/(\mathscr{R}_u((\mathcal{G}_{x,f})_{\kappa})\times_{\kappa}\mathscr{R}_u((\mathcal{G}_{x,f})_{\kappa})=(\mathsf{G}_{x,f}\times_{\kappa}\mathsf{G}_{x,f})\cdot \overline{\Omega_{\red}}$ holds as an equality of fppf-sheaves over $\Sch/\kappa$. Moreover, passing to the special fiber and modulo $\overline{\mathscr{T}^+}$, the open immersion $\overline{\nu}$ (\Cref{definitionofnubar}) induces an open immersion over $\kappa$
    \begin{align*}
        \mathrm{S}&\longrightarrow (\overline{\mathcal{T}})_{\kappa}/\overline{\mathscr{T}^+}=\overline{\mathscr{S}}=\prod_{a\in\Delta}\mathbb{A}_{1,\kappa}\\
        s &\longmapsto \;\;\;\;\; (a(s)^{-1})_{a\in \Delta}.
    \end{align*}
    Since $f(0)=0$, by \cite[Theorem~8.5.14]{Bruhattitsnewapproach}, $\mathrm{S}\times_{\kappa}((\mathcal{U}_{\Phi^+,x,f})_{\kappa}/R_u^+)\subset \mathsf{G}_{x,f}$ is a Borel subgroup corresponding to the simple roots $ \Delta_{x,f}$, whose opposite Borel subgroup is $((\mathcal{U}_{\Phi^-,x,f})_{\kappa}/R_u^-)\times_{\kappa}\mathrm{S}$. Then,
    by \Cref{toroidalembeddingtheorem}, the toroidal embedding $X_{\mathfrak{C}}$ contains $\overline{\Omega_{\red}}$ as an open subscheme. Combining with \Cref{definitionwonderfulembedding}, we have a natural morphism of $\kappa$-schemes
$$\mathsf{G}_{x,f}\times_{\kappa}\overline{\Omega_{\red}}\times_{\kappa}\mathsf{G}_{x,f}/\sim_{(A_{\kappa})_{\red}}\longrightarrow X_{\mathfrak{C}}$$
    which is given by group action of $\mathsf{G}_{x,f}\times_{\kappa}\mathsf{G}_{x,f}$ on $X_{\mathfrak{C}}$, and this morphism is a monomorphism. Again by \Cref{toroidalembeddingtheorem}, we have $X_{\mathfrak{C}}=(\mathsf{G}_{x,f}\times_{\kappa}\mathsf{G}_{x,f})\cdot \overline{\Omega_{\red}}$, thus the above morphism is an isomorphism.

    %To see (2), it suffices to note that, by \Cref{unipotentradicalspecialfiber} (2), the maximal reductive quotient $\mathsf{G}_{x,f}$ is trivial.
\end{proof}

\begin{example}
    We consider $G=\PGL_{2,k}$. The wonderful compactification $\overline{G}$ is $\mathbb{P}_{3,k}$. We take $T$ to be the subgroup of (the equivalence classes of) the diagonal matrices. Then we have the standard positive root and its coroot
    \begin{align*}
        \alpha:T\;\;\;\;&\longrightarrow  \mathbb{G}_{m,k}, \;\; \;\;\;\;\;\;\;     \alpha^\vee:\mathbb{G}_{m,k}   \longrightarrow \;\;\;\;\;\;\;T   \\
    \begin{bmatrix}         
                      t & 0\\
                      0 &  1
                   \end{bmatrix}  &\longmapsto     t   
   \;\;\;\;\;\;\;\;\;\;\;\;\;\;\;\;\;\;\;\; \;\;\;\;\;\;\;\;\;\;t \longrightarrow \;\;\begin{bmatrix}         
                      t^2 & 0\\
                      0 &  1
                   \end{bmatrix}\; .
    \end{align*}
We use $\alpha^\vee$ to identify $X_\ast(T)\cong \mathbb{Z}$. Then we have (with $f=0$)
\begin{itemize}
    \item For the group scheme $\mathcal{G}_{0}=\PGL_{2,\mathfrak{o}}$ together with the natural open immersion $\PGL_{2,k}\hookrightarrow\PGL_{2,\mathfrak{o}}$, we have $\overline{\mathcal{G}_0}=\mathbb{P}_{3,\mathfrak{o}}$ on which $\mathcal{G}_{0}\times_{\mathfrak{o}}\mathcal{G}_{0}$ acts via the multiplication of matrices.
    
    \item For the group scheme $\mathcal{G}_{x}$ with $0\textless x\textless 1/2$, its special fiber is the product $\mathbb{G}_{m,\kappa}\times_{\kappa} \mathbb{A}_{2,\kappa}$. The special fiber of $\overline{\mathcal{G}_x}$ is identified with the affine space $\mathbb{A}_{3,\kappa}$. For simplicity, we only write the left $\mathcal{G}_{x}$-action:
    \begin{align*}
        (\mathcal{G}_{x})_{\kappa}\times_{\kappa}(\overline{\mathcal{G}_{x}})_{\kappa}&\longrightarrow (\overline{\mathcal{G}_{x}})_{\kappa}\\
        ((u^-, t,u^+),(v^-, \overline{t}, v^+))&\longmapsto (t^{-1}v^-+u^-, t^{-1}\overline{t}, \overline{t}u^++v^+).
    \end{align*}
    
    \item For the group scheme $\mathcal{G}_{1/2}=\PGL_{2,\mathfrak{o}}$ together with the open immersion $$\PGL_{2,k}\xlongrightarrow{\Ad(J^{-1})}\PGL_{2,k}\hookrightarrow\PGL_{2,\mathfrak{o}}\;, \text{where}\; J=\begin{bmatrix}
        \pi & 0\\
        0 & 1
    \end{bmatrix},$$
    we have $\overline{\mathcal{G}_{1/2}}=\mathbb{P}_{3,\mathfrak{o}}$ together with the generic fiber $\mathbb{P}_{3,k}\xlongrightarrow{\Ad(J^{-1})}\mathbb{P}_{3,k}\hookrightarrow\mathbb{P}_{3,\mathfrak{o}}.$

\end{itemize} 
\end{example}

%As in \Cref{subsectionspecialfiber},let $e_{\widetilde{a}}$ be the ramification index of the field extension $k_{\widetilde{a}}/k$, and let $\pi_{\widetilde{a}}\in \mathfrak{o}_{\widetilde{a}}$ be a uniformizer. Since $k$ is strictly Henselian, by for instance \cite[Chapter~III, \S~5, Corollary~3]{localfieldsserre}, the field extension $k_{\widetilde{a}}/k$ is totally ramified. Hence, by \cite[Chapter~III, \S~6, Lemma~3]{localfieldsserre}, we can choose $\{1, \pi_{\widetilde{a}},\pi_{\widetilde{a}}^2,\cdots,\pi_{\widetilde{a}}^{e_{\widetilde{a}}-1} \}$ as an $\mathfrak{o}$-basis of $\mathfrak{o}_{\widetilde{a}}$.

\subsection{Boundary divisor}
We shall keep the notations in \Cref{subsectionspecialfiber}. For each $a\in\Delta$,
the $\mathfrak{o}$-basis of $\mathfrak{o}_{\widetilde{a}}$ $\{1, \pi_{\widetilde{a}},\pi_{\widetilde{a}}^2,\cdots,\pi_{\widetilde{a}}^{e_{\widetilde{a}}-1} \}$ gives an isomorphism
\begin{align*}
\mathbb{A}_{e_{\widetilde{a}},\mathfrak{o}}\;\;\;\;\;\;\; &\xlongrightarrow{\cong} \;\;\;\;\;\;\;\;\Res_{\mathfrak{o}_{\widetilde{a}}/\mathfrak{o}}(\mathbb{G}_{a,\mathfrak{o}_{\widetilde{a}}})\\
   (r_1,r_2,...,r_{e_{\widetilde{a}}}) &\longmapsto  r_1\otimes 1+ r_2\otimes\pi_{\widetilde{a}} +...+r_{e_{\widetilde{a}}}\otimes \pi_{\widetilde{a}}^{e_{\widetilde{a}}-1}
\end{align*}
Let $\mathbb{X}_a$ be the first canonical coordinate of $\Res_{\mathfrak{o}_{\widetilde{a}}/\mathfrak{o}}(\mathbb{G}_{a,\mathfrak{o}_{\widetilde{a}}})$ via the above isomorphism. It is clear that the principal open subscheme $D_{\Res_{\mathfrak{o}_{\widetilde{a}}/\mathfrak{o}}(\mathbb{G}_{a,\mathfrak{o}_{\widetilde{a}}})}(\mathbb{X}_a)$ can be identified with $\Res_{\mathfrak{o}_{\widetilde{a}}/\mathfrak{o}}(\mathbb{G}_{m,\mathfrak{o}_{\widetilde{a}}})$.

\begin{proposition}\label{boundarydivisor}
    The complement $\overline{\mathcal{G}_{x,f}}\backslash \mathcal{G}_{x,f}$ is covered by $(\mathcal{G}_{x,f}\times_{\mathfrak{o}}\mathcal{G}_{x,f})$-stable smooth $\mathfrak{o}$-relative effective Cartier divisor $S_{\alpha}$ for $\alpha\in \Delta$ with $\mathfrak{o}$-relative normal crossings.
\end{proposition}

\begin{proof}
    Since $\overline{\mathcal{G}_{x,f}}$ is separated over $\mathfrak{o}$ and $\mathcal{G}_{x,f}$ is affine over $\mathfrak{o}$, by \cite[01SG]{stacks-project}, the open immersion $i: \mathcal{G}_{x,f}\hookrightarrow \overline{\mathcal{G}_{x,f}}$ (\Cref{groupopenimmersion}) is affine. By \cite[corollaire~21.12.7]{EGAIV4}, each irreducible component of the complement $\overline{\mathcal{G}_{x,f}}\backslash \mathcal{G}_{x,f}$ is of pure codimension 1. We endow $\overline{\mathcal{G}_{x,f}}\backslash \mathcal{G}_{x,f}$ with the reduced closed subscheme structure. Let $C$ be an irreducible component of $\overline{\mathcal{G}_{x,f}}\backslash \mathcal{G}_{x,f}$. Since by construction $\overline{\mathcal{G}_{x,f}}=(\mathcal{G}_{x,f}\times_{\mathfrak{o}}\mathcal{G}_{x,f})\cdot \overline{\Omega_{x,f}}$, $C\bigcap \overline{\Omega_{x,f}}$ is also of codimension 1 in the open big cell $\overline{\Omega_{x,f}}$. By \Cref{constructionofunipotentintegralmodel} and the definition of $\overline{\mathcal{T}}$ and \cite[Chapter~III,\S~6,Proposition~12]{localfieldsserre}, the coordinate ring of $\overline{\Omega_{x,f}}$ is (non canonically) isomorphic to a polynomial ring over $\mathfrak{o}$, which is a unique factorization domain. Hence, by, for instance, \cite[Proposition~1.12A]{HartshorneGTM52}, $C\bigcap \overline{\Omega_{x,f}}$ is defined by a single element of the coordinate ring of $\overline{\Omega_{x,f}}$. Since $C\bigcap \overline{\Omega_{x,f}}$ is irreducible and is $((\mathcal{U}^+\cdot \mathcal{T})\times_{\mathfrak{o}}\mathcal{U}^-)$-stable, it is nothing else but the hyperplane cut out by $\mathbb{X}_a$ for a relative simple root $a\in \Delta$. Since, by \Cref{quasiprojectivityandsmoothness} (2), $\overline{\mathcal{G}_{x,f}}$ is covered by the translates of $\overline{\Omega_{x,f}}$ by the sections in $(\mathcal{G}_{x,f}\times_{\mathfrak{o}}\mathcal{G}_{x,f})(\mathfrak{o})$, we conclude that $C$ is a $(\mathcal{G}_{x,f}\times_{\mathfrak{o}}\mathcal{G}_{x,f})$-stable smooth $\mathfrak{o}$-relative effective Cartier divisor. These divisor are with $\mathfrak{o}$-relative normal crossings because it is the case over $\overline{\Omega_{x,f}}$.
\end{proof}

\subsection{Uniqueness of wonderful embedding $\overline{\mathcal{G}_{x,f}}$}\label{uniqueness}

Our proof of the uniqueness part of \Cref{introtheorem1} comes essentially from the proof of \cite[Chapter~5, Proposition~3]{BLR}.
Suppose that we have two $\mathfrak{o}$-schemes $\overline{\mathcal{G}_1}$ and $\overline{\mathcal{G}_2}$ whose generic fibers are both isomorphic to $\overline{G}$ such that
\begin{itemize}
    \item[(i)] the $(G\times_k G)$-equivariant open immersion $G\hookrightarrow \overline{G}$ extends to a $(\mathcal{G}_{x,f}\times_{\mathfrak{o}} \mathcal{G}_{x,f})$-equivariant open immersion $\mathcal{G}_{x,f}\hookrightarrow \overline{\mathcal{G}_{i}}$, $i=1,2$; 
    
    \item[(ii)] the canonical $k$-open immersion $\overline{\Omega}\hookrightarrow \overline{G}$ (see \Cref{propquasisplitwonderfulcomp}) extends to an $\mathfrak{o}$-open immersion 
    $\tau_i:\overline{\Omega_{x,f}}\longhookrightarrow\overline{\mathcal{G}_{i}}$, $i=1,2$;

    \item[(iii)]  $(\mathcal{G}_{x,f}\times_{\mathfrak{o}}\mathcal{G}_{x,f})\cdot \overline{\Omega_{x,f}}= \overline{\mathcal{G}_{i}}$, $i=1,2$.
\end{itemize}
We denote by $A_i$ the action of $(\mathcal{G}_{x,f}\times_{\mathfrak{o}} \mathcal{G}_{x,f})$ on $\overline{\mathcal{G}_{i}}$ for $i=1,2$.
We define an $\mathfrak{o}$-rational morphism $\tau\coloneq \tau_2\circ (\tau_1)^{-1}:\overline{\mathcal{G}_1}\dashrightarrow\overline{\mathcal{G}_2}$. Now, by the condition (iii), we have two surjective morphisms
$$\mathcal{G}_{x,f}\times_{\mathfrak{o}}\overline{\Omega_{x,f}}\times_{\mathfrak{o}}\mathcal{G}_{x,f}\xlongrightarrow{A_1\circ (\Id, \tau_1,\Id)} \overline{\mathcal{G}_1} \;\;\; \text{and}\;\;\;\mathcal{G}_{x,f}\times_{\mathfrak{o}}\overline{\Omega_{x,f}}\times_{\mathfrak{o}}\mathcal{G}_{x,f}\xlongrightarrow{A_2\circ (\Id, \tau_2,\Id)} \overline{\mathcal{G}_2} $$
which are both smooth morphisms (this follows from the smoothness of $\mathcal{G}_{x,f}$ over $\mathfrak{o}$). Moreover, by (i) and (ii), $\tau$ is compatible with the above two morphisms, i.e., 
$$\tau\circ (A_1\circ (\Id, \tau_1,\Id))=A_2\circ (\Id, \tau_2,\Id).$$
\noindent Then, by faithfully flat descent \cite[\S~2.5, Proposition~5]{BLR}, we conclude that $\tau$ is defined over the whole $\overline{\mathcal{G}_1}$. Similarly, we can show that the $\mathfrak{o}$-rational morphism $\tau'\coloneq \tau_1\circ (\tau_2)^{-1}$ is also defined everywhere. We conclude by observing that $\tau$ and $\tau'$ are inverse to each other.

\subsection{Compatibility with dilatation}\label{sectioncompatiblity}
Let us consider two concave functions $f,g:\hat{\Phi}\rightarrow \mathbb{R}$ such that $g\leq f\leq g+1$. Then, by the definition (see \Cref{filtrationsubgroup}), we have the inclusion $U_{a,x,f}\subset U_{a,x,g}$ for any $a\in \hat{\Phi}^{\red}$, where $U_{0,x,f}$ means $T(k)_{f(0)}$. The inclusion induces the morphism of integral models $\mathcal{U}_{a,x,f}\subset \mathcal{U}_{a,x,g}$. Let $V_a$ be the image of $\mathcal{U}_{a,x,f}(\kappa)$ in $\mathcal{G}_{x,g}(\kappa)$, and let $H(\kappa)\subset \mathcal{G}_{x,g}(\kappa)$ be the subgroup generated by $\{V_a\}_{a\in\hat{\Phi}(S)^{\red}}$. By for instance \cite[Lemma~8.1.2]{Yusmoothmodel}, $H(\kappa)$ is Zariski closed. Let $H\subset(\mathcal{G}_{x,g})_{\kappa} $ be the underlying algebraic closed subgroup of $H(\kappa)$. By \cite[Lemma~7.3.1]{Yusmoothmodel} (and its proof), $\mathcal{G}_{x,f}$ is the dilatation of $H$ on $\mathcal{G}_{x,g}$. 

It is natural to consider the dilatation of $H$ on $\overline{\mathcal{G}_{x,g}}$ which will be denoted by $(\overline{\mathcal{G}_{x,g}})_H$. By \Cref{dilitationproperty}, $(\overline{\mathcal{G}_{x,g}})_H$ contains, as an open subscheme, the dilatation of $H\bigcap (\overline{\Omega_{x,g}})_{\kappa}$ on $\overline{\Omega_{x,g}}$ which is identified with $\overline{\Omega_{x,f}}$ (this can be seen from $\mathfrak{o}$-points). Moreover, by \Cref{dilitationproperty} again, $(\overline{\mathcal{G}_{x,g}})_H$ is equipped with the action of $\mathcal{G}_{x,f}\times_{\mathfrak{o}}\mathcal{G}_{x,f}$ so that $(\mathcal{G}_{x,f}\times_{\mathfrak{o}}\mathcal{G}_{x,f})\cdot \overline{\Omega_{x,f}}= (\overline{\mathcal{G}_{x,g}})_H$. Thus we conclude that the dilatation $(\overline{\mathcal{G}_{x,g}})_H$ is actually isomorphic to $\overline{\mathcal{G}_{x,f}}$ by the uniqueness result in \Cref{uniqueness}.

\subsection{Condition for being proper}

\begin{proposition}
      The scheme $\overline{\mathcal{G}_{x,f}}$ is projective over $\Spec(\mathfrak{o})$ if and only if $\mathcal{G}_{x,f}$ is a reductive group scheme over $\mathfrak{o}$. 
\end{proposition}

\begin{proof}
    If $\mathcal{G}_{x,f}$ is a reductive group scheme over $\mathfrak{o}$, i.e., $\mathscr{R}_u((\mathcal{G}_{x,f})_{\kappa})$ is trivial, then by \cite[exposé~XXII, proposition~2.8]{SGA3III}, we have that $\Delta=\Delta_{x,f}$ and $(\mathcal{G}_{x,f})_{\kappa}$ is also adjoint. By the classification of toroidal embedding of $(\mathcal{G}_{x,f})_{\kappa}$ (\cite[Proposition~6.2.4 (iii)]{BrionKumar}), the special fiber $(\overline{\mathcal{G}_{x,f}})_{\kappa}$ is isomorphic to the classical wonderful compactification of $(\mathcal{G}_{x,f})_{\kappa}$ which is projective over $\kappa$ by definition. Then the $\mathfrak{o}$-properness of $\overline{\mathcal{G}_{x,f}}$ follows from the fibral criterion for properness \cite[corollaire~15.7.11]{EGAIV3}.

    For the converse direction, if $\overline{\mathcal{G}_{x,f}}$ is projective over $\Spec(\mathfrak{o})$, then the special fiber $(\overline{\mathcal{G}_{x,f}})_{\kappa}$ is projective over $\Spec(\kappa)$. Since $\mathscr{R}_u((\mathcal{G}_{x,f})_{\kappa})$ is by definition connected and unipotent (hence is solvable), by the \emph{Borel's fixed point theorem} (see, for instance, \cite[Theorem~6.2.6]{linearalgrpSpringer}), there exists a point $\overline{t}\in(\overline{\mathcal{G}_{x,f}})_{\kappa}$ which is fixed by the action of $\mathscr{R}_u((\mathcal{G}_{x,f})_{\kappa})\times_{\kappa}\mathscr{R}_u((\mathcal{G}_{x,f})_{\kappa})$. Since by construction $(\overline{\mathcal{G}_{x,f}})_{\kappa}$ is covered by the $((\mathcal{G}_{x,f})_{\kappa}\times_{\kappa}(\mathcal{G}_{x,f})_{\kappa})$-translates of the open subscheme 
    $$(\overline{\Omega_{x,f}})_{\kappa}\cong (\mathcal{U}_{\Phi^-,x,f})_{\kappa}\times_{\kappa}(\overline{\mathscr{S}}\times_{\kappa}\overline{\mathscr{T}^+})\times_{\kappa}(\mathcal{U}_{\Phi^+,x,f})_{\kappa}$$
    and $\mathscr{R}_u((\mathcal{G}_{x,f})_{\kappa})$ is normal in $(\mathcal{G}_{x,f})_{\kappa}$, we can assume that $\overline{t}\in \overline{\mathscr{S}}$. Then we can see that, inside the open subscheme $(\overline{\Omega_{x,f}})_{\kappa}$, the point $\overline{t}$ is fixed by $R_u^-\times_{\kappa} (\overline{\mathscr{T}^+})_{\kappa}\times_{\kappa}R_u^+$ only if $R_u^-$ and $(\overline{\mathscr{T}^+})_{\kappa}\times_{\kappa}R_u^+$ are both trivial. Then we can conclude $\mathscr{R}_u((\mathcal{G}_{x,f})_{\kappa})$ is trivial by appealing to \Cref{unipotentradicalspecialfiber}. 
\end{proof}

\subsection{Picard group}\label{subsectionpicardgroup}

In the study of wonderful varieties or more general spherical varieties, a certain kind of divisors which are called \emph{colors} plays an important role, see, for instance, the seminal paper \cite{LunatypeA}. For wonderful compactification $\overline{G}$, the colors are the closures $\overline{B\cdot s_i\cdot B^-}\subset \overline{G}$, and these colors freely generates the Picard group of $\overline{G}$, see, for instance, \cite[Lemma~6.1.9]{BrionKumar}. We now show the following analogue result for $\overline{\mathcal{G}_{x,f}}$ in which the divisors should be viewed as affine version of colors. 

\begin{theorem}\label{theoremcolors}
    The boundary $\overline{\mathcal{G}_{x,f}}\backslash \overline{\Omega_{x,f}}$ is covered by prime effective Cartier divisors 
    $$\mathbf{D}_{\alpha}\coloneq \overline{B\cdot \dot{s}_{\alpha}\cdot B^-},\alpha\in \Delta$$
    where the bar indicates taking the schematic closure in $\overline{\mathcal{G}_{x,f}}$ and $\dot{s}_{\alpha}\in N_G(S)(k)$ is a representative of the simple reflection $s_{\alpha}$ along $\alpha$ in the relative Weyl group $W$ of the couple $(G,S)$. Moreover the Picard group $\Pic(\overline{\mathcal{G}_{x,f}})$ is freely generated by $\{\mathbf{D}_{\alpha}\vert \alpha\in \Delta\}$.
\end{theorem}
\emph{We remark that $B\cdot \dot{s}_{\alpha}\cdot B^-\subset G$ dose not depend on the choice of the representative $\dot{s}_{\alpha}$.}

Before we prove this result, we recall the following structural result about the special fiber $\mathcal{G}_{x,f}$ as a preparation. 

\begin{theorem}\label{theoremstructureofreductivequotient}
    (\cite[Theorem~8.5.14]{Bruhattitsnewapproach}, \cite[\S~3.4 Proposition]{Yusmoothmodel}, \cite[corollaire~4.6.4, corollaire~4.6.12]{Bruhattits2}, \cite[Proposition~6.9]{Landvogtcompactification}) 
    If $f(0)=0$, then 
    \begin{itemize}
        \item[(1)] the root datum of $\mathsf{G}_{x,f}$ with respect to the maximal subtorus $\mathsf{S}$ is 
        $(X^*(\mathsf{S}),\Phi_{x,f},X_*(\mathsf{S}),\Phi_{x,f}^\vee);$
        \item[(2)] for a root $b\in \Phi_{x,f}$, the $b$-root subgroup is isomorphic to $(\mathcal{U}_{\overline{b},x,f})_{\kappa}/(\mathcal{U}_{\overline{b},x,f})_{\kappa}^+$, where $\overline{b}$ is the unique non-divisible root in $\mathbb{R}_{\textgreater0}b$.
    \end{itemize}
\end{theorem}

By \Cref{theoremstructureofreductivequotient}, we see that the quotient scheme $(\mathcal{U}_{a,x,f})_{\kappa}/(\mathcal{U}_{a,x,f^+})_{\kappa}$ is nontrivial if and only if $\{a,2a\}\bigcap \Phi_{x,f}\neq \emptyset$. If so, $\{a,2a\}\bigcap \Phi_{x,f}$ contains only one element. Hence, by\cite[Lemma~8.4.1]{Bruhattitsnewapproach}, for $a\in \Phi_{x,f}$, we have $U_{a,x,f^+} \subsetneqq U_{a,x,f}$. Thus we can choose $u\in U_{a,x,f}\backslash U_{a,x,f^+}$, and, by \cite[Theorem~8.2.9]{Bruhattitsnewapproach} combined with \Cref{unipotentradicalspecialfiber} (1), there exist $u', u''\in U_{-a,x,f}$ such that $m(a)\coloneq u'uu''\in \mathcal{G}_{x,f}(\mathfrak{o})$ such that $m(a)$ acts on $X^*(S)$ as a reflection along the root $a$ and $m(a)$ normalizes the maximal $k$-split torus $S\subset G$. Hence, by \Cref{extensionprinciple}, it also normalizes $\mathscr{S}\subset \mathcal{G}_{x,f}$. For each $a\in \Delta_{x,f}$, we choose such a $m(a)$. Then the element $\mathsf{m}(a)\subset \mathsf{G}_{x,f}(\kappa)$ (which depends on the choice of $u$) induced by the special fiber $m(a)_{\kappa}$ normalizes $\mathsf{S}$ and is a lifting of the simple reflection along $a$ in the Weyl group of $\mathsf{G}_{x,f}$. For $r\in \Delta\backslash \Delta_{x,f}$, we choose a representative $\dot{s_r}$ of the simple reflection corresponding to $r$ in $N_G(S)(k)$.

\begin{proof}[Proof of \Cref{theoremcolors}]
    Since $\overline{\mathcal{G}_{x,f}}$ is Noetherian, the boundary $\overline{\mathcal{G}_{x,f}}\backslash \overline{\Omega_{x,f}}$ has finitely many irreducible components $\mathbf{D}_{i}$. By \cite[0BCV]{stacks-project}, each component has codimension $1$.
   \begin{claim}\label{claimcomplementbigcell}
       $\mathcal{G}_{x,f}\backslash \Omega_{x,f}=(\bigcup_{a\in \Delta_{x,f}} \overline{B\cdot m(a)\cdot B^-}) \bigcup (\bigcup_{a\in\Delta\backslash\Delta_{x,f}}\overline{B\cdot \dot{s_a}\cdot B^-})$ holds as an equality of topological spaces, where the bar indicates taking the schematic closure in $\mathcal{G}_{x,f}$.
   \end{claim}
   \begin{proof}[Proof of \Cref{claimcomplementbigcell}]
       It suffices to check the equality on each fiber. For the generic fiber, it follows from the Bruhat order on $G$ (\cite[théorème~5.15, corollaire~5.18]{boreltits}). Let 
       $$\mathsf{U}^+_{x,f}\coloneq  \prod_{r\in\Phi_{x,f}^+} (\mathcal{U}_{\overline{r},x,f})_{\kappa}/(\mathcal{U}_{\overline{r},x,f})_{\kappa}^+\subset \mathsf{G}_{x,f} \;\;\;\;\text{and}\;\;\;\;\mathsf{U}^-_{x,f}\coloneq  \prod_{r\in\Phi_{x,f}^-} (\mathcal{U}_{\overline{r},x,f})_{\kappa}/(\mathcal{U}_{\overline{r},x,f})_{\kappa}^+ \subset \mathsf{G}_{x,f} $$
       Applying \emph{loc. cit.} to the maximal reductive quotient $\mathsf{G}_{x,f}$, we obtain 
       $$\mathsf{G}_{x,f}\backslash \mathsf{U}^-_{x,f}\mathsf{S}\mathsf{U}^+_{x,f}=\bigcup_{a\in\Delta_{x,f}}\overline{\mathsf{U}^-_{x,f}\mathsf{S} \cdot   \mathsf{m}(a)  \cdot  \mathsf{S}\mathsf{U}^+_{x,f}}.$$
       Since $(\mathcal{G}_{x,f})_{\kappa}$ is a $\mathscr{R}_u((\mathcal{G}_{x,f})_{\kappa}$-torsor over $\mathsf{G}_{x,f}$ (equipped with the fppf topology) \cite[exposé~VI, théorème~3.2(iv)]{SGA3I}, the quotient morphism $(\mathcal{G}_{x,f})_{\kappa}\longrightarrow\mathsf{G}_{x,f}$ is flat. Moreover, since $\mathscr{R}_u((\mathcal{G}_{x,f})_{\kappa})\subset(\mathcal{G}_{x,f})_{\kappa}$ is normal, by spelling out the quotients and flat base change, we have
       $$(\mathcal{G}_{x,f})_{\kappa}\backslash (\Omega_{x,f})_{\kappa}=\bigcup_{a\in\Delta_{x,f}}\overline{(\mathcal{U}_{\Phi^-,x,f})_{\kappa} \mathscr{S}_{\kappa} \cdot m(a)_{\kappa} \cdot\mathscr{S}_{\kappa} (\mathcal{U}_{\Phi^+,x,f})_{\kappa}},$$
       where the bar indicates taking closure in the special fiber $(\mathcal{G}_{x,f})_{\kappa}$. We conclude by observing that $\mathscr{S}, \mathcal{U}_{\Phi^-,x,f}, \mathcal{U}_{\Phi^+,x,f}$ are the schematic closures of $S, U^-,U^+$ in $\mathcal{G}_{x,f}$ respectively and $m(a)\in \mathcal{G}_{x,f}(\mathfrak{o})$ viewed as a closed subscheme has the discrete valuation ring $\mathfrak{o}$ as its coordinate ring.
   \end{proof}
   Now since $\mathcal{G}_{x,f}\backslash \Omega_{x,f}$ is dense in $\overline{\mathcal{G}_{x,f}}\backslash \overline{\Omega_{x,f}}$, by \Cref{claimcomplementbigcell}, we have 
   $$\overline{\mathcal{G}_{x,f}}\backslash \overline{\Omega_{x,f}}=(\bigcup_{a\in \Delta_{x,f}} \overline{B\cdot m(a)\cdot B^-}) \bigcup (\bigcup_{a\in\Delta\backslash\Delta_{x,f}}\overline{B\cdot \dot{s_a}\cdot B^-}),$$
   where the bar on the right side now indicates taking schematic closure in $\overline{\mathcal{G}_{x,f}}$. Since each $\overline{B\cdot \dot{s_a}\cdot B^-}$ and $\overline{B\cdot m(a)\cdot B^-}$ is irreducible, they are simply $\mathbf{D}_i$. As $\overline{\mathcal{G}_{x,f}}$ is locally factorial (\Cref{quasiprojectivityandsmoothness}(1) and \Cref{constructionofunipotentintegralmodel}), by \cite[Chapter~II, Proposition~6.11]{HartshorneGTM52}, $\overline{B\cdot \dot{s_a}\cdot B^-}$ and $\overline{B\cdot m(a)\cdot B^-}$ are effective Cartier divisors.
   Since the coordinate ring of $\overline{\Omega_{x,f}}$ is a unique factorization domain, the last result about the Picard group follows from, for instance, \cite[Proposition~11.40]{GortzWedhorn}.
\end{proof}

\section{\'{E}tale descent}\label{sectionetaledescent}
The main goal of this section is to prove \Cref{introtheorem2}. The most involved case is when $f(0)=0$.

\subsection{Setup}

We consider an adjoint reductive (not necessarily quasi-split) group $G$ over a Henselian discretely valued field $k$ with perfect residue field $\kappa$ and ring of integers $\mathfrak{o}$. Let $K$ be the maximal unramified extension of $k$ which is in a separable closure $k_s$ of $k$. Let $\overline{\kappa}$ be the residue field of $K$, then $\overline{\kappa}$ is an algebraic closure of $\kappa$. Let $\mathcal{O}\subset K$ and $\mathcal{O}_s\subset k_s$ be the rings of integers. By a theorem of Steinberg \cite[Theorem~1.9]{Steinbergregularelements}, $G_{K}$ is quasi-split over $K$, see \cite[Proposition~10.1]{Landvogtcompactification} or \cite[Corollary~2.3.8]{Bruhattitsnewapproach} for details. For notational simplicity, we denote the Galois group $\Gal(K/k)$ by $\Gamma$.

Let $\mathscr{B}$ be the Bruhat--Tits building of $G(k)$. Recall that $\mathscr{B}$ is defined to be the fixed point set $\mathcal{B}(G_K)^{\Gamma}$ equipped with a canonical polysimplicial structure (cf. \cite[9.2.4]{Bruhattitsnewapproach}), where $\mathcal{B}(G_K)$ is the Bruhat--Tits building of $G(K)$.

We choose a maximal $k$-split subtorus $S\subset G$ with $\Phi(S)\coloneq \Phi(S,G)$ and $\hat{\Phi}(S)\coloneq \Phi(S)\bigcup \{0\}$. Let $T$ be a special $k$-torus containing $S$ in the sense that $T_K$ is a maximal $K$-split torus of $G_K$. We fix a Borel subgroup $B\subset G_K$ containing $T_K$. 

Let $\Phi(T_K)\coloneq \Phi(T_K, G_K)$ and $\hat{\Phi}(T_K)\coloneq \Phi(T_K)\bigcup\{0\}$. Let $f:\hat{\Phi}(S)\rightarrow \widetilde{\mathbb{R}}$ be a concave function. We denote by $\widetilde{f}:\hat{\Phi}(T_K)\rightarrow \widetilde{\mathbb{R}}$ the concave function obtained by compositing $f$ with the restriction map $X^*(T_K)\rightarrow X^*(S_K)=X^*(S)$. 

Consider a point $x$ in the apartment $\mathcal{A}(S)$ of $\mathcal{B}(G)$ corresponding to $S$. Since $x$ is fixed by $\Gamma$, the natural action of $\Gamma$ on $G_K$ extends to an action of $\Gamma$ on the group scheme $\mathcal{G}_{x,\widetilde{f}}$ over $\mathcal{O}$ (whose definition is recalled in \Cref{subsectionintegralmodelofG}). By Galois descent, we can define the $\mathfrak{o}$-group scheme $\mathcal{G}_{x,f}$ to be the resulting scheme. A different construction of $\mathcal{G}_{x,f}$ is to describe its $\mathfrak{o}$-point by using the valuation of the root datum $(G,S)$ given by $x$, see \cite[\S~9.6, \S~9.8]{Bruhattitsnewapproach}.

Let $\widetilde{T}\coloneq Z_{G_K}(T_K) \subset B$ be the centralizer of $T_K$ in $G_K$, and let $N\subset G_K$ be the normalizer of $T_K$. Since $G_K$ is quasi-split over $K$, $\widetilde{T}$ is a maximal torus in $G_K$. Let $B^-\subset G_K$ be the opposite Borel subgroup such that $B\bigcap B^-=\widetilde{T}$. Let $U^+$ and $U^-$ be the unipotent radicals of $B$ and $B^-$ respectively.

Let $\mathcal{R}\coloneq(X,\widetilde{\Psi},\widetilde{\Delta},X^\vee,\widetilde{\Psi}^\vee,\widetilde{\Delta}^\vee)$ be the based root datum of $G_{k_s}$ defined by $\widetilde{T}\subset B$. Let $\Delta$ be the pullback of $\widetilde{\Delta}$ to $T_K$. 

When $f(0)\textgreater 0$, by \Cref{f0biggerthat0}, $\overline{\mathcal{G}_{x,\widetilde{f}}}$ is the gluing of $\overline{G_K}$ and $\mathcal{G}_{x,\widetilde{f}}$ along the open subscheme $G_K$. To see \Cref{introtheorem2} in this case, it suffices to let $\overline{\mathcal{G}_{x,f}}$ be the gluing of $\overline{G}$ with $\mathcal{G}_{x,f}$ along $G$, where $\overline{G}$ is descended from $\overline{G_{k_s}}$ by \cite[\S~6]{li2023equivariant}.

\emph{In the following, we assume that $f(0)=0$.}

\subsection{Descent datum of group schemes}\label{subsectiondescentdatumgrp}

By \cite[Theorem~7.1.9]{redctiveconrad}, we have a short exact sequence of schematic representable sheaves over the category $\Sch/K$:
\begin{equation}\label{automexactsequence}
    1\longrightarrow G_K\longrightarrow \Aut_{G_K/K}\longrightarrow \Out_{G_K/K}\longrightarrow 1,
\end{equation}
where, with the choices of the maximal torus $\widetilde{T}$ and of the Borel subgroup $B$ (and of a pinning), the outer automorphism scheme $(\Out_{G_K/K})_{k_s}\cong \Out_{G_{k_s}/k_s}$ is identified with the finite constant group scheme $\underline{\Aut(\mathcal{R})}_{k_s}$. Since $T_K$ is defined over $k$, the action of $\Gamma$ on $\mathcal{B}(G_K)$ preserves $\mathcal{A}(T_K)$. An element $\sigma\in \Gamma$ gives rise to an element of $\Aut_K(G_K)$ also denoted by $\sigma$. By \Cref{automexactsequence}, $\sigma$ is the composition of an inner automorphism given by an element $g\in G(K)$ and an automorphism $\xi$ of $G_K$ whose base change $\xi_{k_s}$ is an outer automorphism of $G_{k_s}$. Since, by the definition of outer automorphism, $\xi_{k_s}$ stabilizes the maximal torus $\widetilde{T}_{k_s}$, the $K$-automorphism $\xi$ also preserves the maximal $K$-split torus $T_K$. Hence $\xi$ stabilizes the apartment $\mathcal{A}(T_K)$. Then $g$ must also preserve $\mathcal{A}(T_K)$. Since the stabilizer of $\mathcal{A}(T_K)$ in $G(K)$ is $N(K)$ (\cite[Proposition~7.6.8]{Bruhattitsnewapproach}), we know that $g\in N(K)$.

\subsection{Extension of descent data}

Let us now extend the action of $\Gamma$ on $\mathcal{G}_{x,\widetilde{f}}$ to an action of $\Gamma$ on $\overline{\mathcal{G}_{x,\widetilde{f}}}$ which is constructed in \Cref{sectionparahoriccompactification}. Let $\overline{\mathcal{T}}$ be the connected lft-Néron model of the induced $K$-torus $\widetilde{T}$. Let $\mathcal{U}^-_{x,\widetilde{f}}$ and $\mathcal{U}^+_{x,\widetilde{f}}$ be the schematic closures of $U^-$ and $U^+$ in $\mathcal{G}_{x,\widetilde{f}}$ respectively.

By \Cref{subsectiondescentdatumgrp}, we see that the action of an element $\sigma\in \Gamma$ on $G_K$ is a combination of the conjugation by element $g\in N(K)$ and an automorphism $\xi$ descended from an outer automorphism of $G_{k_s}$. 

\subsubsection{Extension of outer part}
Recall that, since $G_{k_s}$ is adjoint, the lattice $X^*(\widetilde{T})\coloneq \Hom_{k_s-\text{grp}}(\widetilde{T}_{k_s}, \mathbb{G}_{m,k_s})$, as an abelian group, is generated by the (absolute) simple roots $\widetilde{\Delta}$. The outer automorphism $\xi_{k_s}$ acts on $\widetilde{T}_{k_s}\cong \Spec(k_s[X^*(\widetilde{T})])$ via a permutation $\mu$ of $\widetilde{\Delta}$. Since $\xi_{k_s}$ is the base change of the $K$-automorphism $\xi$, we have that the following commutative diagram
$$\xymatrix{
   &\widetilde{\Delta} \ar[r]^{\mu}\ar[d]^{\gamma}   &\widetilde{\Delta} \ar[d]^{\gamma}  \\
   &\widetilde{\Delta} \ar[r]^{\mu}   &\widetilde{\Delta} 
}$$
holds for any $\gamma\in \Gal(k_s/K)$. Then $\mu$ gives rise to an automorphism of $\prod_{\widetilde{\Delta}}\mathbb{A}_{1,\mathcal{O}_s}$ by permuting the index set $\widetilde{\Delta}$, and, by the above commutative diagram, this automorphism is compatible with the action of $\Gal(k_s/K)$ on $\prod_{\widetilde{\Delta}}\mathbb{A}_{1,\mathcal{O}_s}$ (which is naturally defined by the action of $\Gal(k_s/K)$ on $\widetilde{\Delta}$). Hence, by descent, $\mu$ gives a $\mathcal{O}$-automorphism of $\overline{\mathcal{T}}$.

\subsubsection{Extension of inner part}

For the element $g\in N(K)$, we denote by $w$ its image in the relative Weyl group $W\coloneq N(K)/S(K)$. Since now $f(0)=0$, by \cite[Theorem~8.2.9, Theorem~8.5.12]{Bruhattitsnewapproach}, we can find a lifting $\dot{g}\in\mathcal{G}_{x,\widetilde{f}}(\mathcal{O})\bigcap N(K)$ of $w$. Let $\mathcal{T}$ be the schematic closure of $\widetilde{T}$ in $\mathcal{G}_{x,\widetilde{f}}$. Then the conjugates by $g$ and by $\dot{g}$ on $\mathcal{T}$ coincide because they differ by an element in $Z(K)$.

\hspace*{\fill} 

The composition $\overline{\sigma} \coloneq\Ad(\dot{g})\circ \mu\in \Aut_{\mathcal{O}}(\overline{\mathcal{T}})$ is then an extension of the automorphism of $\mathcal{T}$ induced by $\sigma$. Now we consider the following morphism
\begin{align}\label{quotientequation}
    \mathcal{G}_{x,\widetilde{f}}\times_{\mathcal{O}} \overline{\Omega_{x,\widetilde{f}}}\times_{\mathcal{O}} \mathcal{G}_{x,\widetilde{f}}=\mathcal{G}_{x,\widetilde{f}}\times_{\mathcal{O}}(\mathcal{U}^-_{x,\widetilde{f}}\times_{\mathcal{O}}\overline{\mathcal{T}}\times_{\mathcal{O}}\mathcal{U}^+_{x,\widetilde{f}})\times_{\mathcal{O}} \mathcal{G}_{x,\widetilde{f}}&\longrightarrow \overline{\mathcal{G}_{x,\widetilde{f}}}\\
    (p_1, u^-, \overline{t},u^+, p_2)&\longmapsto \sigma(p_1u^-)\cdot\overline{\sigma}(\overline{t})\cdot\sigma(u^+p_2),
\end{align}
where $(p_1, u^-, \overline{t},u^+, p_2)\in( \mathcal{G}_{x,\widetilde{f}}\times_{\mathcal{O}} \overline{\Omega_{x,\widetilde{f}}}\times_{\mathcal{O}} \mathcal{G}_{x,\widetilde{f}})(S)$ for a test $\mathcal{O}$-scheme $S$. By the construction of $\overline{\mathcal{G}_{x,\widetilde{f}}}$ (\Cref{f0equalto0}), we have that \Cref{quotientequation} actually induces an $\mathcal{O}$-automorphism of $\overline{\mathcal{G}_{x,\widetilde{f}}}$. In this way, we get an action of $\Gamma$ on $\overline{\mathcal{G}_{x,\widetilde{f}}}$ which extends the $\Gamma$-action on $\mathcal{G}_{x,\widetilde{f}}$.

\begin{theorem}\label{theoremdescent}
    The $(\mathcal{G}_{x,\widetilde{f}}\times_{\mathcal{O}}\mathcal{G}_{x,\widetilde{f}})$-scheme $\overline{\mathcal{G}_{x,\widetilde{f}}}$ descends to a smooth quasi-projective $(\mathcal{G}_{x,f}\times_{\mathfrak{o}}\mathcal{G}_{x,f})$-scheme $\overline{\mathcal{G}_{x,f}}$ over $\mathfrak{o}$ which contains $\mathcal{G}_{x,f}$ as an open dense subscheme, and the boundary $\overline{\mathcal{G}_{x,f}}\backslash \mathcal{G}_{x,f}$ is an $\mathfrak{o}$-relative Cartier divisor with $\mathfrak{o}$-relative normal crossings. Moreover, $\overline{\mathcal{G}_{x,f}}$ is $\mathfrak{o}$-projective if and only if $\mathcal{G}_{x,f}$ is a reductive group scheme over $\mathfrak{o}$.
\end{theorem}

\begin{proof}
    The effectivity of the descent datum follows from \cite[\S~6.2, Example~B]{BLR}. The quasi-projectivity follows from \cite[\S~6.6, Theorem~2(d)]{BLR}. The other properties of $\overline{\mathcal{G}_{x,f}}$ follow from those of $\overline{\mathcal{G}_{x,\widetilde{f}}}$ (\Cref{introtheorem1}) by descent \cite[proposition~2.7.1]{EGAIV2}, \cite[corollaire~17.7.3 (ii)]{EGAIV4}, \cite[exposé~XIII,~2.1.0]{SGA1} and \cite[définition~2.7]{SGA3III}. 
\end{proof}

\appendix

\section{Dilatation}\label{appendixdilatation}
For the convenience of the reader, we recollect the basics of dilatation which we use in the main body of this paper. 

Let $k$ be a discrete valuation field with valuation ring $\mathfrak{o}$, with uniformizer $\pi$, with $\mathfrak{m}$ the maximal ideal, and with residue field $\kappa$. Let $X$ be a scheme of finite type over $\mathfrak{o}$, and let $Y\subset X_{\kappa}$ be a closed subscheme. 

\begin{definition-proposition}
\cite[\S~3.2, Proposition~1]{BLR}
There is a unique (up to canonical isomorphism) pair $$(X_Y,u: X_Y\rightarrow X),$$ 
where $X_Y$ is a flat scheme over $\mathfrak{o}$ and $u$ is a homomorphism of $\mathfrak{o}$-schemes, satisfying the following universal property:
for any flat scheme $Z$ over $\mathfrak{o}$ and any $\mathfrak{o}$-morphism $v:Z\rightarrow X$ such that the special fiber $v_{\kappa}: Z_{\kappa}\rightarrow X_{\kappa}$ factor through $Y\subset X_{\kappa}$, then there exists a unique $\mathfrak{o}$-morphism $v':Z\rightarrow X_Y$ such that $v=u \circ v'$. 

We call the pair $(X_Y,u\colon X_Y\rightarrow X)$ the \emph{dilatation} of $Y$ on $X$.
\end{definition-proposition}

Actually, $X_Y$ is the open locus of the blowing-up of $\Bl_Y(X)$ at which the ideal sheaf of the exceptional divisor is generated by the uniformizer $\pi$, see \emph{loc. cit} or \cite[Remark~A.5.9]{Bruhattitsnewapproach}.

\begin{proposition}\label{dilitationproperty}
    (Special case of \cite[Corollary~2.8]{Nerondilatation})
    Let $X'\rightarrow X$ be a \emph{flat} morphism of $\mathfrak{o}$-schemes, and let $Y'\subset X'_{\kappa}$ be the preimage of $Y$. If the morphism $X'\rightarrow X$ satisfies property $\mathcal{P}$ which is stable under base change, then the natural morphism $X'_{Y'}\rightarrow X_{Y}$ also satisfies $\mathcal{P}$ and is flat.
\end{proposition}

We now assume that $X$ is affine and flat over $\mathfrak{o}$.

\begin{lemma}\label{dilitationcordinatering}
\cite[Construction~A.5.2, Proposition~A.5.3]{Bruhattitsnewapproach}
    Let $I\subset \mathfrak{o}[X]$ be the ideal defining $Y$. Then $\mathfrak{o}[X][\pi^{-1}I]\subset \mathfrak{o}[X]\otimes_{\mathfrak{o}} k$ is isomorphic to the coordinate ring of $X_Y$. In particular, 
    $X_Y$ is an integral model of $X_{k}$, i.e., $(X_Y)_{k}=X_{k}$.
\end{lemma}

Let $k'/k$ be a finite separable field extension of ramification degree $e$, and let $\mathfrak{o}'$ be the ring of integers of $k'$. Let $\pi'\in \mathfrak{o}'$ be a uniformizer.
Let $\mathcal{A}_{(0)}$ (resp., $\mathcal{G}_{(0)}$) be $\Res_{\mathfrak{o}'/\mathfrak{o}}(\mathbb{A}_{1,\mathfrak{o}'})$ (resp., $\Res_{\mathfrak{o}'/\mathfrak{o}}(\mathbb{G}_{m,\mathfrak{o}'})$), and let $\mathfrak{e}\in(\mathcal{G}_{(0)})_{\kappa}(\kappa)$ be the identity section. By \Cref{dilitationproperty}, the canonical open immersion $\mathcal{G}_{(0)}\hookrightarrow \mathcal{A}_{(0)} $ induces an open immersion $\mathcal{G}_{(1)}\hookrightarrow \mathcal{A}_{(1)}$, where $\mathcal{G}_{(1)}$ (resp., $\mathcal{A}_{(1)}$) is defined as the dilatation of $\mathfrak{e}$ in $\mathcal{G}_{(0)}$ (resp., $\mathcal{A}_{(0)}$). Note that, by \cite[Lemma~A.5.7]{Bruhattitsnewapproach}, $\mathcal{G}_{(1)}$ is a group scheme, while $\mathcal{A}_{(1)}$ is not. Then, by induction and \Cref{dilitationproperty}, for $n\in \mathbb{Z}_{>0}$, we have the open immersion $\mathcal{G}_{(n)}\hookrightarrow \mathcal{A}_{(n)}.$
The group scheme $\mathcal{G}_{(n)}$ is called the $n^{th}$ congruence group scheme of $\Res_{\mathfrak{o}'/\mathfrak{o}}(\mathbb{G}_{m,\mathfrak{o}'})$, cf., \cite[Definition~A.5.12]{Bruhattitsnewapproach}.

\begin{lemma}\label{dilatationofaffineline}
    The open immersion $\mathcal{G}_{(n)}\hookrightarrow \mathcal{A}_{(n)}$ induces an $\kappa$-isomorphism of the special fibers 
    $$(\mathcal{G}_{(n)})_{\kappa}\xlongrightarrow{\cong} (\mathcal{A}_{(n)})_{\kappa}\;,\text{when} \; n\textgreater 0.$$
\end{lemma}

\begin{proof}
   This is a special case of \cite[Theorem~3.5 (3)]{Nerondilatation}. Here we present a concrete computation.
   We have the following isomorphisms:
   $$\Res_{\mathfrak{o}'/\mathfrak{o}}(\mathbb{A}_{1,\mathfrak{o}'})^{(n)}_{\kappa}\cong \Res_{\mathfrak{o}'/\mathfrak{o}}(\mathbb{A}_{1,\mathfrak{o}'}^{(en)})_{\kappa}\cong \Res_{(\mathfrak{o}'/\pi'^e)/\kappa}(\mathbb{A}_{1,\mathfrak{o}'}^{(en)}\times_{\mathfrak{o}'}\mathfrak{o}'/\pi'^e),$$
   $$\Res_{\mathfrak{o}'/\mathfrak{o}}(\mathbb{G}_{m,\mathfrak{o}'})^{(n)}_{\kappa}\cong \Res_{\mathfrak{o}'/\mathfrak{o}}(\mathbb{G}_{m,\mathfrak{o}'}^{(en)})_{\kappa}\cong \Res_{(\mathfrak{o}'/\pi'^e)/\kappa}(\mathbb{G}_{m,\mathfrak{o}'}^{(en)}\times_{\mathfrak{o}'}\mathfrak{o}'/\pi'^e)$$
   By \Cref{dilitationcordinatering} and induction, the coordinate ring $\mathfrak{o}'[\mathbb{A}_{1,\mathfrak{o}'}^{(en)}]$ (resp., $\mathfrak{o}'[\mathbb{G}_{m,\mathfrak{o}'}^{(en)}]$) is $\mathfrak{o}'[T,\pi'^{-en}(T-1)]\subset k'[T]$ (resp., $\mathfrak{o}'[T^{-1},\pi'^{-en}(T-1)]\subset k'[T, T^{-1}]$). Now we introduce the new variable $u\coloneq \pi'^{-en}(T-1)$. Then $\mathfrak{o}'[\mathbb{A}_{1,\mathfrak{o}'}^{(en)}]$ (resp., $\mathfrak{o}'[\mathbb{G}_{m,\mathfrak{o}'}^{(en)}]$) becomes $\mathfrak{o}'[u, 1+\pi'^{en}u]=\mathfrak{o}'[u]$ (resp., $\mathfrak{o}'[u, 1+\pi'^{en}u,v]/((1+\pi'^{en}u)v-1)$). Hence we have that $\mathbb{G}_{m,\mathfrak{o}'}^{(en)}\times_{\mathfrak{o}'}\mathfrak{o}'/\pi'^e\cong \mathbb{A}_{1,\mathfrak{o}'}^{(en)}\times_{\mathfrak{o}'}\mathfrak{o}'/\pi'^e$, as desired.
\end{proof}

Thus, since the smooth affine integral models of $\mathbb{A}_{1,k}$ are determined by their $\mathfrak{o}$-rational points (see, for instance, \cite[Corollary~2.10.11]{Bruhattitsnewapproach}), $\mathcal{A}_{(n)}$ can be characterized as the unique smooth affine integral model of $\mathbb{A}_{1,k}$ over $\mathfrak{o}$ such that $\mathcal{A}_{(n)}(\mathfrak{o})= 1+\mathfrak{m}^n$.

\section{Distribution}\label{appendixdistribution}

Consider a ring $R$, an $R$-affine scheme $X$ and $x\in X(R)$. Let $I_x\subset R[X]$ be the ideal which defines the closed immersion $x$. The distribution on $X$ with support in $x$ is defined to be the $R$-module:
$$\Dist(X,x)\coloneq \{\;\mu\in \Hom_{R-\Mod}(R[X], R)\;\vert\; \mu(I_x^n)=0 \;\text{for some}\; n\in \mathbb{Z}_{\geq 0}\;\}.$$

The construction of distribution is functorial in the sense that, if $f: X\rightarrow Y$ is a morphism of affine $R$-schemes and $x\in X(R)$, then we have a natural homomorphism of $R$-modules $$d(f)_x:\Dist(X,x)\rightarrow \Dist(Y,f(x)).$$

\begin{lemma}\label{distributionlocalproperty}
  (\cite[1.3.2]{Bruhattits2}) If $f$ is an immersion (resp., an open immersion), then $d(f)_x$ is an injection (resp., an bijection).  
\end{lemma}

\begin{lemma}\label{distributiongenericfiber}
    (\cite[1.3.3]{Bruhattits2}) Assume that $X$ is of finite type over $R$ and that $R$ is integral. Let $K$ be the field of fractions of $R$. For a section $x\in X(R)$, let $x'\in X_{K}(K)$ be the base change of $x$. Then, the base change along $R\rightarrow K$ gives an embedding
    \begin{equation}\label{distributiongenricfiber}
        \Dist(X,s) \hookrightarrow \Dist(X_K,s').
    \end{equation}
\end{lemma}

Assume that $X$ is smooth over $R$ and that $R$ is an integral ring with fraction field $K$. In our context, the following result is very useful, which roughly says that an integral model over a DVR can be probed by its distribution from its generic fiber. 

\begin{proposition}\label{propextension}
    (\cite[Proposition~0.4]{Landvogtcompactification}) Suppose that $X$ is an irreducible smooth affine scheme over a discrete valuation ring $R$ with irreducible generic fiber and $x\in X(R)$. Let $K$ be the fraction field of $R$. Then (under \Cref{distributiongenricfiber})
    $$R[X]=\{f\in K[X]\vert \mu(f)\in R \;\text{for any}\; \mu\in \Dist(X,x)\}.$$  
\end{proposition}

\begin{corollary}\label{corollaryextension}
    Suppose that $X$ and $Y$ are irreducible smooth affine schemes over a discrete valuation ring $R$ with irreducible generic fiber and $x\in X(R)$. Let $f: X_K\longrightarrow Y_K$ be a $K$-morphism. If $d(f)_x(\Dist(X,x))\subset \Dist(Y,f(x))$, then $f$ is the generic fiber of an $R$-morphism from $X$ to $Y$.
\end{corollary}

\begin{proof}
    Let $f^\#: K[Y_K]\longrightarrow K[X_K]$ be the corresponding morphism of coordinate rings. 
    Then it suffices to show that $f^\#(R[Y])\subset R[X]$. If there exists $g\in R[Y]$ such that $f^\#(g)\in K[X_K]\backslash R[X]$, then, by \Cref{propextension}, there exists $\mu\in \Dist(X,x)$ with $\mu(f^\#(g))\in K\backslash R$. This contradicts with the assumption $d(f)_x(\Dist(X,x))\subset \Dist(Y,f(x))$.
\end{proof}

\renewcommand{\bibname}{References}
%\bibliographystyle{myalpha}
%\bibliography{ref}

\printbibliography

@article {EGA1,
    AUTHOR = {Grothendieck, A.},
     TITLE = {\'{E}l\'{e}ments de g\'{e}om\'{e}trie alg\'{e}brique. {I}. {L}e langage des
              sch\'{e}mas},
   JOURNAL = {Inst. Hautes \'{E}tudes Sci. Publ. Math.},
   shorthand = {EGAI},
  FJOURNAL = {Institut des Hautes \'{E}tudes Scientifiques. Publications
              Math\'{e}matiques},
    NUMBER = {4},
      YEAR = {1960},
     PAGES = {228},
      ISSN = {0073-8301},
   MRCLASS = {14.55},
  MRNUMBER = {217083},
       URL = {http://www.numdam.org/item?id=PMIHES_1960__4__228_0},
SHORTHAND ={EGA~I}
}

@article {EGAIV2,
    AUTHOR = {Grothendieck, A.},
     TITLE = {\'{E}l\'{e}ments de g\'{e}om\'{e}trie alg\'{e}brique. {IV}. \'{E}tude locale des
              sch\'{e}mas et des morphismes de sch\'{e}mas. {II}},
   JOURNAL = {Inst. Hautes \'{E}tudes Sci. Publ. Math.},
  FJOURNAL = {Institut des Hautes \'{E}tudes Scientifiques. Publications
              Math\'{e}matiques},
    NUMBER = {24},
      YEAR = {1965},
     PAGES = {231},
      ISSN = {0073-8301},
   MRCLASS = {14.00},
  MRNUMBER = {199181},
MRREVIEWER = {H. Hironaka},
       URL = {http://www.numdam.org/item?id=PMIHES_1965__24__231_0},
SHORTHAND ={EGA~IV$_2$}
}

@article {EGAIV3,
    AUTHOR = {Grothendieck, A.},
     TITLE = {\'{E}l\'{e}ments de g\'{e}om\'{e}trie alg\'{e}brique. {IV}. \'{E}tude locale des
              sch\'{e}mas et des morphismes de sch\'{e}mas. {III}},
   JOURNAL = {Inst. Hautes \'{E}tudes Sci. Publ. Math.},
  FJOURNAL = {Institut des Hautes \'{E}tudes Scientifiques. Publications
              Math\'{e}matiques},
    NUMBER = {24},
      YEAR = {1965},
     PAGES = {231},
      ISSN = {0073-8301},
   MRCLASS = {14.00},
  MRNUMBER = {199181},
MRREVIEWER = {H. Hironaka},
       URL = {http://www.numdam.org/item?id=PMIHES_1965__24__231_0},
SHORTHAND ={EGA~IV$_3$}
}

@book{EGAIV4,
    AUTHOR = {Grothendieck, A.},
     TITLE = {\'{E}l\'{e}ments de g\'{e}om\'{e}trie alg\'{e}brique. {IV}. \'{E}tude locale des
              sch\'{e}mas et des morphismes de sch\'{e}mas {IV}},
   JOURNAL = {Inst. Hautes \'{E}tudes Sci. Publ. Math.},
  FJOURNAL = {Institut des Hautes \'{E}tudes Scientifiques. Publications
              Math\'{e}matiques},
    NUMBER = {32},
      YEAR = {1967},
     PAGES = {361},
      ISSN = {0073-8301},
   MRCLASS = {14.55},
  MRNUMBER = {238860},
MRREVIEWER = {J. P. Murre},
       URL = {http://www.numdam.org/item?id=PMIHES_1967__32__361_0},
SHORTHAND ={EGA~IV$_4$},
}

@book {SGA1,
      shorthand={SGA~1},
     TITLE = {Rev\^{e}tements \'{e}tales et groupe fondamental ({SGA} 1)},
    SERIES = {Documents Math\'{e}matiques (Paris) [Mathematical Documents
              (Paris)]},
    VOLUME = {3},
      NOTE = {S\'{e}minaire de g\'{e}om\'{e}trie alg\'{e}brique du Bois Marie 1960--61.
              [Algebraic Geometry Seminar of Bois Marie 1960-61],
              Directed by A. Grothendieck,
              With two papers by M. Raynaud,
              Updated and annotated reprint of the 1971 original [Lecture
              Notes in Math., 224, Springer, Berlin;  MR0354651 (50
              \#7129)]},
 PUBLISHER = {Soci\'{e}t\'{e} Math\'{e}matique de France, Paris},
      YEAR = {2003},
     PAGES = {xviii+327},
      ISBN = {2-85629-141-4},
   MRCLASS = {14E20 (14-06 14F35)},
  MRNUMBER = {2017446},
}

@book{SGA3I,
      shorthand={SGA~3$_{\text{I}}$},
     TITLE = {Sch\'{e}mas en groupes ({SGA} 3). {T}ome {I}. {P}ropri\'{e}t\'{e}s
              g\'{e}n\'{e}rales des sch\'{e}mas en groupes},
    SERIES = {Documents Math\'{e}matiques (Paris) [Mathematical Documents
              (Paris)]},
    VOLUME = {7},
    EDITOR = {Gille, Philippe and Polo, Patrick},
      NOTE = {S\'{e}minaire de G\'{e}om\'{e}trie Alg\'{e}brique du Bois Marie 1962--64.
              [Algebraic Geometry Seminar of Bois Marie 1962--64],
              A seminar directed by M. Demazure and A. Grothendieck with the
              collaboration of M. Artin, J.-E. Bertin, P. Gabriel, M.
              Raynaud and J-P. Serre,
              Revised and annotated edition of the 1970 French original},
 PUBLISHER = {Soci\'{e}t\'{e} Math\'{e}matique de France, Paris},
      YEAR = {2011},
     PAGES = {xxviii+610},
      ISBN = {978-2-85629-323-2},
   MRCLASS = {14L15},
  MRNUMBER = {2867621},
 
}

@book{SGA3II,
     shorthand={SGA~3$_{\text{II}}$},
     TITLE = {Sch\'{e}mas en groupes. {II}: {G}roupes de type multiplicatif, et
              structure des sch\'{e}mas en groupes g\'{e}n\'{e}raux},
    SERIES = {Lecture Notes in Mathematics, Vol. 152},
      NOTE = {S\'{e}minaire de G\'{e}om\'{e}trie Alg\'{e}brique du Bois Marie 1962/64 (SGA
              3),
              Dirig\'{e} par M. Demazure et A. Grothendieck},
 PUBLISHER = {Springer-Verlag, Berlin-New York},
     PAGES = {ix+654},
   MRCLASS = {14.50},
  MRNUMBER = {0274459},
}

@book{SGA3III,
     shorthand={SGA~3$_{\text{III}}$},
     TITLE = {Sch\'{e}mas en groupes ({SGA} 3). {T}ome {III}. {S}tructure des
              sch\'{e}mas en groupes r\'{e}ductifs},
    SERIES = {Documents Math\'{e}matiques (Paris) [Mathematical Documents
              (Paris)]},
    VOLUME = {8},
    EDITOR = {Gille, Philippe and Polo, Patrick},
      NOTE = {S\'{e}minaire de G\'{e}om\'{e}trie Alg\'{e}brique du Bois Marie 1962--64.
              [Algebraic Geometry Seminar of Bois Marie 1962--64],
              A seminar directed by M. Demazure and A. Grothendieck with the
              collaboration of M. Artin, J.-E. Bertin, P. Gabriel, M.
              Raynaud and J-P. Serre,
              Revised and annotated edition of the 1970 French original},
 PUBLISHER = {Soci\'{e}t\'{e} Math\'{e}matique de France, Paris},
      YEAR = {2011},
     PAGES = {lvi+337},
      ISBN = {978-2-85629-324-9},
   MRCLASS = {14L15},
  MRNUMBER = {2867622},
}

@book {BLR,
    AUTHOR = {Bosch, Siegfried and L\"{u}tkebohmert, Werner and Raynaud, Michel},
     TITLE = {N\'{e}ron models},
    SERIES = {Ergebnisse der Mathematik und ihrer Grenzgebiete (3) [Results
              in Mathematics and Related Areas (3)]},
    VOLUME = {21},
 PUBLISHER = {Springer-Verlag, Berlin},
      YEAR = {1990},
     PAGES = {x+325},
      ISBN = {3-540-50587-3},
   MRCLASS = {14K15 (11G10 14L15)},
  MRNUMBER = {1045822},
MRREVIEWER = {James Milne},
       DOI = {10.1007/978-3-642-51438-8},
       URL = {https://doi-org.ezproxy.universite-paris-saclay.fr/10.1007/978-3-642-51438-8},
}

@article {compactificationHilbertsch,
    AUTHOR = {Brion, Michel},
     TITLE = {Group completions via {H}ilbert schemes},
   JOURNAL = {J. Algebraic Geom.},
  FJOURNAL = {Journal of Algebraic Geometry},
    VOLUME = {12},
      YEAR = {2003},
    NUMBER = {4},
     PAGES = {605--626},
      ISSN = {1056-3911},
   MRCLASS = {14M17 (14C05)},
  MRNUMBER = {1993758},
MRREVIEWER = {P. E. Newstead},
       DOI = {10.1090/S1056-3911-03-00315-1},
       URL = {https://doi-org.ezproxy.universite-paris-saclay.fr/10.1090/S1056-3911-03-00315-1},
}

@incollection {redctiveconrad,
    AUTHOR = {Conrad, Brian},
     TITLE = {Reductive group schemes},
 BOOKTITLE = {Autour des sch\'{e}mas en groupes. {V}ol. {I}},
    SERIES = {Panor. Synth\`eses},
    VOLUME = {42/43},
     PAGES = {93--444},
 PUBLISHER = {Soc. Math. France, Paris},
      YEAR = {2014},
   MRCLASS = {14L15},
  MRNUMBER = {3362641},
}

@book {linearalgrpSpringer,
    AUTHOR = {Springer, T. A.},
     TITLE = {Linear algebraic groups},
    SERIES = {Modern Birkh\"{a}user Classics},
   EDITION = {second},
 PUBLISHER = {Birkh\"{a}user Boston, Inc., Boston, MA},
      YEAR = {2009},
     PAGES = {xvi+334},
      ISBN = {978-0-8176-4839-8},
   MRCLASS = {20G15 (14L10)},
  MRNUMBER = {2458469},
}

@article {Lunawonderfulvarieties,
    AUTHOR = {Luna, D.},
     TITLE = {Toute vari\'{e}t\'{e} magnifique est sph\'{e}rique},
   JOURNAL = {Transform. Groups},
  FJOURNAL = {Transformation Groups},
    VOLUME = {1},
      YEAR = {1996},
    NUMBER = {3},
     PAGES = {249--258},
      ISSN = {1083-4362},
   MRCLASS = {14L30 (14M17)},
  MRNUMBER = {1417712},
MRREVIEWER = {Lucy Moser-Jauslin},
       DOI = {10.1007/BF02549208},
       URL = {https://doi-org.ezproxy.universite-paris-saclay.fr/10.1007/BF02549208},
}

@book {toroidalembedding,
     shorthand={KKMS73},
    AUTHOR = {Kempf, G. and Knudsen, Finn Faye and Mumford, D. and
              Saint-Donat, B.},
     TITLE = {Toroidal embeddings. {I}},
    SERIES = {Lecture Notes in Mathematics, Vol. 339},
 PUBLISHER = {Springer-Verlag, Berlin-New York},
      YEAR = {1973},
     PAGES = {viii+209},
   MRCLASS = {14E15 (14D20 14E05 14M20 20G15)},
  MRNUMBER = {0335518},
MRREVIEWER = {G. Harder},
}

@book {BrionKumar,
    AUTHOR = {Brion, Michel and Kumar, Shrawan},
     TITLE = {Frobenius splitting methods in geometry and representation
              theory},
    SERIES = {Progress in Mathematics},
    VOLUME = {231},
 PUBLISHER = {Birkh\"{a}user Boston, Inc., Boston, MA},
      YEAR = {2005},
     PAGES = {x+250},
      ISBN = {0-8176-4191-2},
   MRCLASS = {14M15 (13A35 14C05 17B10 20G05)},
  MRNUMBER = {2107324},
MRREVIEWER = {Vikram B. Mehta},
}

@book {HartshorneGTM52,
    AUTHOR = {Hartshorne, Robin},
     TITLE = {Algebraic geometry},
    SERIES = {Graduate Texts in Mathematics, No. 52},
 PUBLISHER = {Springer-Verlag, New York-Heidelberg},
      YEAR = {1977},
     PAGES = {xvi+496},
      ISBN = {0-387-90244-9},
   MRCLASS = {14-01},
  MRNUMBER = {0463157},
MRREVIEWER = {Robert Speiser},
}

@misc{stacks-project,
    shorthand ={SP},
    author       = {A. J. de Jong et al.},
    title        = {\textit{Stacks Project}},
    year         = {2024},
    howpublished = { Available at \url{https://stacks.math.columbia.edu}},
   
  }

@incollection {anantharaman,
    AUTHOR = {Anantharaman, Sivaramakrishna},
     TITLE = {Sch\'emas en groupes, espaces homog\`enes et espaces
              alg\'ebriques sur une base de dimension 1},
 BOOKTITLE = {Sur les groupes alg\'ebriques},
    SERIES = {Suppl\'ement au Bull. Soc. Math. France},
    VOLUME = {Tome 101},
     PAGES = {5--79},
 PUBLISHER = {Soc. Math. France, Paris},
      YEAR = {1973},
   MRCLASS = {14L15},
  MRNUMBER = {335524},
MRREVIEWER = {J.-E.\ Bertin},
       DOI = {10.24033/msmf.109},
       URL = {https://doi.org/10.24033/msmf.109},
}

@book {Champsalgebrique,
    AUTHOR = {Laumon, G\'{e}rard and Moret-Bailly, Laurent},
     TITLE = {Champs alg\'{e}briques},
    SERIES = {Ergebnisse der Mathematik und ihrer Grenzgebiete. 3. Folge. A
              Series of Modern Surveys in Mathematics [Results in
              Mathematics and Related Areas. 3rd Series. A Series of Modern
              Surveys in Mathematics]},
    VOLUME = {39},
 PUBLISHER = {Springer-Verlag, Berlin},
      YEAR = {2000},
     PAGES = {xii+208},
      ISBN = {3-540-65761-4},
   MRCLASS = {14A20 (14D20)},
  MRNUMBER = {1771927},
MRREVIEWER = {Dan Edidin},
}

@incollection {completesymmetricvarieties,
    AUTHOR = {De Concini, C. and Procesi, C.},
     TITLE = {Complete symmetric varieties},
 BOOKTITLE = {Invariant theory ({M}ontecatini, 1982)},
    SERIES = {Lecture Notes in Math.},
    VOLUME = {996},
     PAGES = {1--44},
 PUBLISHER = {Springer, Berlin},
      YEAR = {1983},
   MRCLASS = {14L30 (14N10 20G05)},
  MRNUMBER = {718125},
MRREVIEWER = {Klaus Pommerening},
       DOI = {10.1007/BFb0063234},
       URL = {https://doi-org.ezproxy.universite-paris-saclay.fr/10.1007/BFb0063234},
}

@book {localfieldsserre,
    AUTHOR = {Serre, Jean-Pierre},
     TITLE = {Local fields},
    SERIES = {Graduate Texts in Mathematics},
    VOLUME = {67},
      NOTE = {Translated from the French by Marvin Jay Greenberg},
 PUBLISHER = {Springer-Verlag, New York-Berlin},
      YEAR = {1979},
     PAGES = {viii+241},
      ISBN = {0-387-90424-7},
   MRCLASS = {12Bxx},
  MRNUMBER = {554237},
}

@book {GortzWedhorn,
    AUTHOR = {G\"{o}rtz, Ulrich and Wedhorn, Torsten},
     TITLE = {Algebraic geometry {I}. {S}chemes---with examples and
              exercises},
    SERIES = {Springer Studium Mathematik---Master},
      NOTE = {Second edition [of  2675155]},
 PUBLISHER = {Springer Spektrum, Wiesbaden},
      YEAR = {2020},
     PAGES = {vii+625},
      ISBN = {978-3-658-30732-5; 978-3-658-30733-2},
   MRCLASS = {14-01},
  MRNUMBER = {4225278},
       DOI = {10.1007/978-3-658-30733-2},
       URL = {https://doi-org.ezproxy.universite-paris-saclay.fr/10.1007/978-3-658-30733-2},
}

@incollection {SpringerICM,
    AUTHOR = {Springer, Tonny A.},
     TITLE = {Some results on compactifications of semisimple groups},
 BOOKTITLE = {International {C}ongress of {M}athematicians. {V}ol. {II}},
     PAGES = {1337--1348},
 PUBLISHER = {Eur. Math. Soc., Z\"{u}rich},
      YEAR = {2006},
   MRCLASS = {14M17 (20G15)},
  MRNUMBER = {2275648},
MRREVIEWER = {Benjamin M. S. Martin},
}

@article {Lusztigparabolicsheaves,
    AUTHOR = {Lusztig, G.},
     TITLE = {Parabolic character sheaves. {II}},
   JOURNAL = {Mosc. Math. J.},
  FJOURNAL = {Moscow Mathematical Journal},
    VOLUME = {4},
      YEAR = {2004},
    NUMBER = {4},
     PAGES = {869--896, 981},
      ISSN = {1609-3321},
   MRCLASS = {20G99},
  MRNUMBER = {2124170},
MRREVIEWER = {Bhama Srinivasan},
       DOI = {10.17323/1609-4514-2004-4-4-869-896},
       URL = {https://doi-org.ezproxy.universite-paris-saclay.fr/10.17323/1609-4514-2004-4-4-869-896},
}

@article {XuhuaHelocalmodel,
    AUTHOR = {He, Xuhua},
     TITLE = {Normality and {C}ohen-{M}acaulayness of local models of
              {S}himura varieties},
   JOURNAL = {Duke Math. J.},
  FJOURNAL = {Duke Mathematical Journal},
    VOLUME = {162},
      YEAR = {2013},
    NUMBER = {13},
     PAGES = {2509--2523},
      ISSN = {0012-7094},
   MRCLASS = {14G35 (13H10 14M15 14M27)},
  MRNUMBER = {3127807},
MRREVIEWER = {Marc-Hubert Nicole},
       DOI = {10.1215/00127094-2371864},
       URL = {https://doi-org.ezproxy.universite-paris-saclay.fr/10.1215/00127094-2371864},
}

@book {Bruhattitsnewapproach,
    AUTHOR = {Kaletha, Tasho and Prasad, Gopal},
     TITLE = {Bruhat-{T}its theory---a new approach},
    SERIES = {New Mathematical Monographs},
    VOLUME = {44},
 PUBLISHER = {Cambridge University Press, Cambridge},
      YEAR = {2023},
     PAGES = {xxx+718},
      ISBN = {978-1-108-83196-3},
   MRCLASS = {20E42 (11F70 20G25 22E50)},
  MRNUMBER = {4520154},
MRREVIEWER = {Corina\ Ciobotaru},
}

@article {Bruhattits1,
    AUTHOR = {Bruhat, F. and Tits, J.},
     TITLE = {Groupes r\'{e}ductifs sur un corps local . {I}.},
   JOURNAL = {Inst. Hautes \'{E}tudes Sci. Publ. Math.},
  FJOURNAL = {Institut des Hautes \'{E}tudes Scientifiques. Publications
              Math\'{e}matiques},
    NUMBER = {41},
      YEAR = {1972},
     PAGES = {5--251},
      ISSN = {0073-8301,1618-1913},
   MRCLASS = {20G25 (22E20)},
  MRNUMBER = {327923},
MRREVIEWER = {M.\ Krusemeyer},
       URL = {http://www.numdam.org/item?id=PMIHES_1972__41__5_0},
}

@article {Bruhattits2,
    AUTHOR = {Bruhat, F. and Tits, J.},
     TITLE = {Groupes r\'{e}ductifs sur un corps local. {II}. {S}ch\'{e}mas
              en groupes. {E}xistence d'une donn\'{e}e radicielle
              valu\'{e}e},
   JOURNAL = {Inst. Hautes \'{E}tudes Sci. Publ. Math.},
  FJOURNAL = {Institut des Hautes \'{E}tudes Scientifiques. Publications
              Math\'{e}matiques},
    NUMBER = {60},
      YEAR = {1984},
     PAGES = {197--376},
      ISSN = {0073-8301,1618-1913},
   MRCLASS = {20G25 (14L15)},
  MRNUMBER = {756316},
MRREVIEWER = {James\ E.\ Humphreys},
       URL = {http://www.numdam.org/item?id=PMIHES_1984__60__5_0},
}

@book {Landvogtcompactification,
    AUTHOR = {Landvogt, Erasmus},
     TITLE = {A compactification of the {B}ruhat-{T}its building},
    SERIES = {Lecture Notes in Mathematics},
    VOLUME = {1619},
 PUBLISHER = {Springer-Verlag, Berlin},
      YEAR = {1996},
     PAGES = {viii+152},
      ISBN = {3-540-60427-8},
   MRCLASS = {20G25 (20E42)},
  MRNUMBER = {1441308},
MRREVIEWER = {Curtis\ D.\ Bennett},
       DOI = {10.1007/BFb0094594},
       URL = {https://doi.org/10.1007/BFb0094594},
}

@incollection {Yusmoothmodel,
    AUTHOR = {Yu, Jiu-Kang},
     TITLE = {Smooth models associated to concave functions in
              {B}ruhat-{T}its theory},
 BOOKTITLE = {Autour des sch\'{e}mas en groupes. {V}ol. {III}},
    SERIES = {Panor. Synth\`eses},
    VOLUME = {47},
     PAGES = {227--258},
 PUBLISHER = {Soc. Math. France, Paris},
      YEAR = {2015},
      ISBN = {978-2-85629-820-6},
   MRCLASS = {14L15 (20E42 20G25)},
  MRNUMBER = {3525846},
MRREVIEWER = {Zinovy\ Reichstein},
}

@article {boreltits,
    AUTHOR = {Borel, Armand and Tits, Jacques},
     TITLE = {Groupes r\'{e}ductifs},
   JOURNAL = {Inst. Hautes \'{E}tudes Sci. Publ. Math.},
  FJOURNAL = {Institut des Hautes \'{E}tudes Scientifiques. Publications
              Math\'{e}matiques},
    NUMBER = {27},
      YEAR = {1965},
     PAGES = {55--150},
      ISSN = {0073-8301,1618-1913},
   MRCLASS = {14.50 (20.75)},
  MRNUMBER = {207712},
MRREVIEWER = {F.\ D.\ Veldkamp},
       URL = {http://www.numdam.org/item?id=PMIHES_1965__27__55_0},
}

@book {CGP,
    AUTHOR = {Conrad, Brian and Gabber, Ofer and Prasad, Gopal},
     TITLE = {Pseudo-reductive groups},
    SERIES = {New Mathematical Monographs},
    VOLUME = {26},
   EDITION = {Second},
 PUBLISHER = {Cambridge University Press, Cambridge},
      YEAR = {2015},
     PAGES = {xxiv+665},
      ISBN = {978-1-107-08723-1},
   MRCLASS = {20G15 (14L15)},
  MRNUMBER = {3362817},
       DOI = {10.1017/CBO9781316092439},
       URL = {https://doi.org/10.1017/CBO9781316092439},
}

@incollection {deconciniprocesi,
    AUTHOR = {De Concini, C. and Procesi, C.},
     TITLE = {Complete symmetric varieties},
 BOOKTITLE = {Invariant theory ({M}ontecatini, 1982)},
    SERIES = {Lecture Notes in Math.},
    VOLUME = {996},
     PAGES = {1--44},
 PUBLISHER = {Springer, Berlin},
      YEAR = {1983},
      ISBN = {3-540-12319-9},
   MRCLASS = {14L30 (14N10 20G05)},
  MRNUMBER = {718125},
MRREVIEWER = {Klaus\ Pommerening},
       DOI = {10.1007/BFb0063234},
       URL = {https://doi.org/10.1007/BFb0063234},
}

@incollection {Vinbergmonoid,
    AUTHOR = {Vinberg, E. B.},
     TITLE = {On reductive algebraic semigroups},
 BOOKTITLE = {Lie groups and {L}ie algebras: {E}. {B}. {D}ynkin's {S}eminar},
    SERIES = {Amer. Math. Soc. Transl. Ser. 2},
    VOLUME = {169},
     PAGES = {145--182},
 PUBLISHER = {Amer. Math. Soc., Providence, RI},
      YEAR = {1995},
      ISBN = {0-8218-0454-5},
   MRCLASS = {20G15 (20M20)},
  MRNUMBER = {1364458},
MRREVIEWER = {Lex\ Renner},
       DOI = {10.1090/trans2/169/10},
       URL = {https://doi.org/10.1090/trans2/169/10},
}

@article{li2023equivariant,
  title={An equivariant compactification for adjoint reductive group schemes},
  author={Li, Shang},
  journal={arXiv preprint arXiv:2308.01715},
  year={2023}
}

@article {Steinbergregularelements,
    AUTHOR = {Steinberg, Robert},
     TITLE = {Regular elements of semisimple algebraic groups},
   JOURNAL = {Inst. Hautes \'Etudes Sci. Publ. Math.},
  FJOURNAL = {Institut des Hautes \'Etudes Scientifiques. Publications
              Math\'ematiques},
    NUMBER = {25},
      YEAR = {1965},
     PAGES = {49--80},
      ISSN = {0073-8301,1618-1913},
   MRCLASS = {14.50 (20.75)},
  MRNUMBER = {180554},
MRREVIEWER = {J.\ L.\ Koszul},
       URL = {http://www.numdam.org/item?id=PMIHES_1965__25__49_0},
}

@article {schneiderstuhler,
    AUTHOR = {Schneider, Peter and Stuhler, Ulrich},
     TITLE = {Representation theory and sheaves on the {B}ruhat-{T}its
              building},
   JOURNAL = {Inst. Hautes \'Etudes Sci. Publ. Math.},
  FJOURNAL = {Institut des Hautes \'Etudes Scientifiques. Publications
              Math\'ematiques},
    NUMBER = {85},
      YEAR = {1997},
     PAGES = {97--191},
      ISSN = {0073-8301,1618-1913},
   MRCLASS = {22E50 (11F70 20G25)},
  MRNUMBER = {1471867},
MRREVIEWER = {Ernst-Wilhelm\ Zink},
       URL = {http://www.numdam.org/item?id=PMIHES_1997__85__97_0},
}

@article {Moyprasad,
    AUTHOR = {Moy, Allen and Prasad, Gopal},
     TITLE = {Jacquet functors and unrefined minimal {$K$}-types},
   JOURNAL = {Comment. Math. Helv.},
  FJOURNAL = {Commentarii Mathematici Helvetici},
    VOLUME = {71},
      YEAR = {1996},
    NUMBER = {1},
     PAGES = {98--121},
      ISSN = {0010-2571,1420-8946},
   MRCLASS = {22E50 (22E35)},
  MRNUMBER = {1371680},
MRREVIEWER = {Mark\ Reeder},
       DOI = {10.1007/BF02566411},
       URL = {https://doi.org/10.1007/BF02566411},
}

@book {bourbakigal,
    AUTHOR = {Bourbaki, N.},
     TITLE = {\'El\'ements de math\'ematique. {F}asc. {XXXIV}. {G}roupes et
              alg\`ebres de {L}ie. {C}hapitre {IV}: {G}roupes de {C}oxeter
              et syst\`emes de {T}its. {C}hapitre {V}: {G}roupes engendr\'es
              par des r\'eflexions. {C}hapitre {VI}: syst\`emes de racines},
    SERIES = {Actualit\'es Scientifiques et Industrielles [Current
              Scientific and Industrial Topics]},
    VOLUME = {No. 1337},
 PUBLISHER = {Hermann, Paris},
      YEAR = {1968},
     PAGES = {288 pp. (loose errata)},
   MRCLASS = {22.50 (17.00)},
  MRNUMBER = {240238},
MRREVIEWER = {G.\ B.\ Seligman},
}

@article {wonderfulembeddingloop,
    AUTHOR = {Solis, Pablo},
     TITLE = {A wonderful embedding of the loop group},
   JOURNAL = {Adv. Math.},
  FJOURNAL = {Advances in Mathematics},
    VOLUME = {313},
      YEAR = {2017},
     PAGES = {689--717},
      ISSN = {0001-8708,1090-2082},
   MRCLASS = {14M27 (20G44 22E67)},
  MRNUMBER = {3649235},
MRREVIEWER = {Paolo\ Bravi},
       DOI = {10.1016/j.aim.2016.10.016},
       URL = {https://doi.org/10.1016/j.aim.2016.10.016},
}

@article {wonderfulcompactificationquatum,
    AUTHOR = {Ganev, Iordan},
     TITLE = {The wonderful compactification for quantum groups},
   JOURNAL = {J. Lond. Math. Soc. (2)},
  FJOURNAL = {Journal of the London Mathematical Society. Second Series},
    VOLUME = {99},
      YEAR = {2019},
    NUMBER = {3},
     PAGES = {778--806},
      ISSN = {0024-6107,1469-7750},
   MRCLASS = {14L24 (14M27 16T99 20G42)},
  MRNUMBER = {3977890},
MRREVIEWER = {Scott\ R.\ Nollet},
       DOI = {10.1112/jlms.12193},
       URL = {https://doi.org/10.1112/jlms.12193},
}

@article {theoreticalphysics,
    AUTHOR = {Berghoff, Marko},
     TITLE = {Wonderful compactifications in quantum field theory},
   JOURNAL = {Commun. Number Theory Phys.},
  FJOURNAL = {Communications in Number Theory and Physics},
    VOLUME = {9},
      YEAR = {2015},
    NUMBER = {3},
     PAGES = {477--547},
      ISSN = {1931-4523,1931-4531},
   MRCLASS = {81T17 (05E15 14M27)},
  MRNUMBER = {3399925},
       DOI = {10.4310/CNTP.2015.v9.n3.a2},
       URL = {https://doi.org/10.4310/CNTP.2015.v9.n3.a2},
}

@article {demazureautomorphismborel,
    AUTHOR = {Demazure, M.},
     TITLE = {Automorphismes et d\'eformations des vari\'et\'es de {B}orel},
   JOURNAL = {Invent. Math.},
  FJOURNAL = {Inventiones Mathematicae},
    VOLUME = {39},
      YEAR = {1977},
    NUMBER = {2},
     PAGES = {179--186},
      ISSN = {0020-9910,1432-1297},
   MRCLASS = {14M15 (20G35)},
  MRNUMBER = {435092},
MRREVIEWER = {George\ R.\ Kempf},
       DOI = {10.1007/BF01390108},
       URL = {https://doi.org/10.1007/BF01390108},
}

@book {Weilbirationalgrouplaw,
    AUTHOR = {Weil, Andr\'e},
     TITLE = {Vari\'et\'es ab\'eliennes et courbes alg\'ebriques},
    SERIES = {Publications de l'Institut de Math\'ematiques de
              l'Universit\'e{} de Strasbourg [Publications of the
              Mathematical Institute of the University of Strasbourg]},
    VOLUME = {8 (1946)},
      NOTE = {Actualit\'es Scientifiques et Industrielles, No. 1064.
              [Current Scientific and Industrial Topics]},
 PUBLISHER = {Hermann \& Cie, Paris},
      YEAR = {1948},
     PAGES = {165},
   MRCLASS = {14.0X},
  MRNUMBER = {29522},
MRREVIEWER = {O.\ F. G. Schilling},
}

@article {groupschemeG2,
    AUTHOR = {Gan, Wee Teck and Yu, Jiu-Kang},
     TITLE = {Sch\'emas en groupes et immeubles des groupes exceptionnels
              sur un corps local. {I}. {L}e groupe {$G_2$}},
   JOURNAL = {Bull. Soc. Math. France},
  FJOURNAL = {Bulletin de la Soci\'et\'e{} Math\'ematique de France},
    VOLUME = {131},
      YEAR = {2003},
    NUMBER = {3},
     PAGES = {307--358},
      ISSN = {0037-9484,2102-622X},
   MRCLASS = {14L15 (20E42 20G25)},
  MRNUMBER = {2017142},
MRREVIEWER = {Nguy\cftil en Qu\^oc Th\'ang},
       DOI = {10.24033/bsmf.2445},
       URL = {https://doi.org/10.24033/bsmf.2445},
}

@article {groupschemeF4E6,
    AUTHOR = {Gan, Wee Teck and Yu, Jiu-Kang},
     TITLE = {Sch\'emas en groupes et immeubles des groupes exceptionnels
              sur un corps local. {II}. {L}es groupes {$F_4$} et {$E_6$}},
   JOURNAL = {Bull. Soc. Math. France},
  FJOURNAL = {Bulletin de la Soci\'et\'e{} Math\'ematique de France},
    VOLUME = {133},
      YEAR = {2005},
    NUMBER = {2},
     PAGES = {159--197},
      ISSN = {0037-9484,2102-622X},
   MRCLASS = {14L15 (20G15 20G25)},
  MRNUMBER = {2172264},
MRREVIEWER = {Nguy\cftil en Qu\^oc Th\'ang},
       DOI = {10.24033/bsmf.2483},
       URL = {https://doi.org/10.24033/bsmf.2483},
}

@incollection {Titscorvallis1977,
    AUTHOR = {Tits, J.},
     TITLE = {Reductive groups over local fields},
 BOOKTITLE = {Automorphic forms, representations and {$L$}-functions
              ({P}roc. {S}ympos. {P}ure {M}ath., {O}regon {S}tate {U}niv.,
              {C}orvallis, {O}re., 1977), {P}art 1},
    SERIES = {Proc. Sympos. Pure Math.},
    VOLUME = {XXXIII},
     PAGES = {29--69},
 PUBLISHER = {Amer. Math. Soc., Providence, RI},
      YEAR = {1979},
      ISBN = {0-8218-1435-4},
   MRCLASS = {20G25 (20G10)},
  MRNUMBER = {546588},
}

@article {BruhatTitsclassic1,
    AUTHOR = {Bruhat, F. and Tits, J.},
     TITLE = {Sch\'emas en groupes et immeubles des groupes classiques sur
              un corps local},
   JOURNAL = {Bull. Soc. Math. France},
  FJOURNAL = {Bulletin de la Soci\'et\'e{} Math\'ematique de France},
    VOLUME = {112},
      YEAR = {1984},
    NUMBER = {2},
     PAGES = {259--301},
      ISSN = {0037-9484},
   MRCLASS = {20G25},
  MRNUMBER = {788969},
MRREVIEWER = {Guy\ Rousseau},
       URL = {http://www.numdam.org/item?id=BSMF_1984__112__259_0},
}

@article {BruhatTitsclassic2,
    AUTHOR = {Bruhat, F. and Tits, J.},
     TITLE = {Sch\'emas en groupes et immeubles des groupes classiques sur
              un corps local. {II}. {G}roupes unitaires},
   JOURNAL = {Bull. Soc. Math. France},
  FJOURNAL = {Bulletin de la Soci\'et\'e{} Math\'ematique de France},
    VOLUME = {115},
      YEAR = {1987},
    NUMBER = {2},
     PAGES = {141--195},
      ISSN = {0037-9484},
   MRCLASS = {20G25 (11E57 22E50)},
  MRNUMBER = {919421},
MRREVIEWER = {Guy\ Rousseau},
       URL = {http://www.numdam.org/item?id=BSMF_1987__115__141_0},
}

@article {Twistedloopgroupsaffineflagvarieties,
    AUTHOR = {Pappas, G. and Rapoport, M.},
     TITLE = {Twisted loop groups and their affine flag varieties},
      NOTE = {With an appendix by T. Haines and Rapoport},
   JOURNAL = {Adv. Math.},
  FJOURNAL = {Advances in Mathematics},
    VOLUME = {219},
      YEAR = {2008},
    NUMBER = {1},
     PAGES = {118--198},
      ISSN = {0001-8708,1090-2082},
   MRCLASS = {22E67 (14M15 17B67 20G25)},
  MRNUMBER = {2435422},
MRREVIEWER = {Dmitry\ A.\ Timash\"ev},
       DOI = {10.1016/j.aim.2008.04.006},
       URL = {https://doi.org/10.1016/j.aim.2008.04.006},
}

@article {thebehaviouratinfinityofbruhatdecomposition,
    AUTHOR = {Brion, Michel},
     TITLE = {The behaviour at infinity of the {B}ruhat decomposition},
   JOURNAL = {Comment. Math. Helv.},
  FJOURNAL = {Commentarii Mathematici Helvetici},
    VOLUME = {73},
      YEAR = {1998},
    NUMBER = {1},
     PAGES = {137--174},
      ISSN = {0010-2571,1420-8946},
   MRCLASS = {14L30 (14M15 14M17)},
  MRNUMBER = {1610599},
MRREVIEWER = {Lucy\ Moser-Jauslin},
       DOI = {10.1007/s000140050049},
       URL = {https://doi.org/10.1007/s000140050049},
}

@article {geometryofsecondadjointness,
    AUTHOR = {Bezrukavnikov, Roman and Kazhdan, David},
     TITLE = {Geometry of second adjointness for {$p$}-adic groups},
      NOTE = {With an appendix by Yakov Varshavsky, Bezrukavnikov and
              Kazhdan},
   JOURNAL = {Represent. Theory},
  FJOURNAL = {Representation Theory. An Electronic Journal of the American
              Mathematical Society},
    VOLUME = {19},
      YEAR = {2015},
     PAGES = {299--332},
      ISSN = {1088-4165},
   MRCLASS = {20G15 (14G20 14L15 14L24 20G25 22E35 22E50)},
  MRNUMBER = {3430373},
MRREVIEWER = {Marko\ Tadi\'c},
       DOI = {10.1090/ert/471},
       URL = {https://doi.org/10.1090/ert/471},
}

@incollection {springerintersectioncohomology,
    AUTHOR = {Springer, T. A.},
     TITLE = {Intersection cohomology of {$B\times B$}-orbit closures in
              group compactifications},
      NOTE = {With an appendix by Wilberd van der Kallen,
              Special issue in celebration of Claudio Procesi's 60th
              birthday},
   JOURNAL = {J. Algebra},
  FJOURNAL = {Journal of Algebra},
    VOLUME = {258},
      YEAR = {2002},
    NUMBER = {1},
     PAGES = {71--111},
      ISSN = {0021-8693,1090-266X},
   MRCLASS = {14F43 (14M17 20C08 20G15)},
  MRNUMBER = {1958898},
MRREVIEWER = {Michel\ Brion},
       DOI = {10.1016/S0021-8693(02)00515-X},
       URL = {https://doi.org/10.1016/S0021-8693(02)00515-X},
}

@article {LunatypeA,
    AUTHOR = {Luna, D.},
     TITLE = {Vari\'et\'es sph\'eriques de type {$A$}},
   JOURNAL = {Publ. Math. Inst. Hautes \'Etudes Sci.},
  FJOURNAL = {Publications Math\'ematiques. Institut de Hautes \'Etudes
              Scientifiques},
    NUMBER = {94},
      YEAR = {2001},
     PAGES = {161--226},
      ISSN = {0073-8301,1618-1913},
   MRCLASS = {14L30 (32M12 57S20)},
  MRNUMBER = {1896179},
MRREVIEWER = {Friedrich\ Knop},
       DOI = {10.1007/s10240-001-8194-0},
       URL = {https://doi.org/10.1007/s10240-001-8194-0},
}

@incollection {Chvealleysemisimplegroup,
    AUTHOR = {Chevalley, Claude},
     TITLE = {Certains sch\'emas de groupes semi-simples},
 BOOKTITLE = {S\'eminaire {B}ourbaki, {V}ol.\ 6},
     PAGES = {Exp. No. 219, 219--234},
 PUBLISHER = {Soc. Math. France, Paris},
      YEAR = {1995},
      ISBN = {2-85629-039-6},
   MRCLASS = {14L15},
  MRNUMBER = {1611814},
}

@article {Nerondilatation,
    AUTHOR = {Mayeux, Arnaud and Richarz, Timo and Romagny, Matthieu},
     TITLE = {N\'eron blowups and low-degree cohomological applications},
   JOURNAL = {Algebr. Geom.},
  FJOURNAL = {Algebraic Geometry},
    VOLUME = {10},
      YEAR = {2023},
    NUMBER = {6},
     PAGES = {729--753},
      ISSN = {2313-1691,2214-2584},
   MRCLASS = {14L15},
  MRNUMBER = {4673395},
MRREVIEWER = {Alan\ Koch},
       DOI = {10.14231/ag-2023-026},
       URL = {https://doi.org/10.14231/ag-2023-026},
}

\end{document}